\documentclass[twoside]{amsart}

\numberwithin{equation}{section}

\usepackage{amsmath,amsfonts,amsthm,amssymb,mathrsfs,latexsym,paralist,xypic}
\usepackage{enumerate}
\usepackage{graphicx}
\usepackage{nomencl,mathrsfs}
\usepackage{color,pdfpages}
\usepackage{tikz, soul}
\usepackage{lipsum}
\usepackage{dsfont}
\usepackage{mathabx}
\usepackage{subsupscripts}
\usepackage[margin=3cm]{geometry}
\usepackage{hyperref}
\usepackage[OMLmathsfit]{isomath}
\usepackage[all]{xy}

\newtheorem{theorem}{Theorem}[section]
\newtheorem{lemma}[theorem]{Lemma}
\newtheorem{proposition}[theorem]{Proposition}
\newtheorem{corollary}[theorem]{Corollary}

\theoremstyle{definition}
\newtheorem{definition}[theorem]{Definition}
\newtheorem{example}[theorem]{Example}
\newtheorem{notation}[theorem]{Notation}
\theoremstyle{remark}
\newtheorem{remark}[theorem]{Remark}

\AtBeginDocument{%
   \def\MR#1{}
}

\newcommand{\Sp}{\mathbb{S}}

\newcommand{\R}{\mathbb{R}}

\newcommand{\Z}{\mathbb{Z}}
\newcommand{\N}{\mathbb{N}}

\newcommand{\GL}{\operatorname{GL}}
\newcommand{\SO}{\operatorname{SO}}
\newcommand{\co}{\colon\thinspace}

\newcommand{\HOM}{\operatorname{Hom}}
\newcommand{\Ad}{\mathrm{Ad}}

\newcommand{\Sat}{\operatorname{Sat}}
\newcommand{\pr}{\operatorname{pr}}
\newcommand{\ftimes}[2]{{\lrsubscripts{\times}{#1}{#2}}}

\begin{document}

\title{Orbifold Euler characteristics of non-orbifold groupoids}

\author{Carla Farsi}
\address{Department of Mathematics, University of Colorado at Boulder,
    UCB 395, Boulder, CO 80309-0395}
\email{farsi@euclid.colorado.edu}

\author{Christopher Seaton}
\address{Department of Mathematics and Computer Science,
    Rhodes College, 2000 N. Parkway, Memphis, TN 38112}
\email{seatonc@rhodes.edu}

\keywords{orbifold Euler characteristic, Euler characteristic,
proper topological groupoid, proper Lie groupoid, inertia groupoid,
inertia space, semialgebraic groupoid.}
\subjclass[2010]{Primary 22A22, 57S15; Secondary 14P10, 57R18}
\thanks{C.F. was partially supported by the sabbatical program of the University of Colorado-Boulder,
and the Simons Foundation Collaboration Grant for Mathematicians \#523991.
C.S. was supported by the E.C.~Ellett Professorship in Mathematics and the sabbatical program of Rhodes College.}

\begin{abstract}
For a finitely presented discrete group $\Gamma$, we introduce two generalizations of the orbifold Euler characteristic and $\Gamma$-orbifold Euler characteristic to a class of proper topological groupoids large enough to include all cocompact proper Lie groupoids. The \emph{$\Gamma$-Euler characteristic} is defined as an integral with respect to the Euler characteristic over the orbit space of the groupoid, and the \emph{$\Gamma$-inertia Euler characteristic} is the usual Euler characteristic of the \emph{$\Gamma$-inertia space} associated to the groupoid. A key ingredient is the application of o-minimal structures to study orbit spaces of topological groupoids. Our main result is that the $\Gamma$-Euler characteristic and $\Gamma$-inertia Euler characteristic coincide and generalize the higher-order orbifold Euler characteristics of Gusein-Zade, Luengo, and Melle-Hern\'{a}ndez from the case of a translation groupoid by a compact Lie group and $\Gamma = \mathbb{Z}^\ell$. By realizing the $\Gamma$-Euler characteristic as the usual Euler characteristic of a topological space, we demonstrate that it is Morita invariant in the category of topological groupoids and satisfies familiar properties of the classical Euler characteristic. We give an additional formulation of the $\Gamma$-Euler characteristic for a cocompact proper Lie groupoid in terms of a finite covering by orbispace charts. In the case that the groupoid is an abelian extension of a translation groupoid by a bundle of groups, we relate the $\Gamma$-Euler characteristics to those of the translation groupoid and bundle of groups.
\end{abstract}

\maketitle

\tableofcontents


\section{Introduction}
\label{sec:Intro}

The string-theoretic orbifold Euler characteristic, often called simply the \emph{orbifold Euler characteristic},
was introduced for a global quotient orbifold (the quotient of a manifold by a finite group) in \cite{DixonHarveyEA}
and generalized to non-global quotient orbifolds in \cite{Roan}. For the quotient of a manifold $X$ by the finite
group $G$, it is given by
\begin{equation}
\label{eq:StringOrbEuler}
    \chi^{orb}(X, G)
    =
    \sum\limits_{[g]\in \Ad_G\!\backslash G} \chi\big( C_G(g)\backslash X^{\langle g \rangle}\big).
\end{equation}
See Notation~\ref{not:ConjClassCent}.
It was shown to be equal to the Euler characteristic of orbifold $K$-theory for
global quotients in \cite{AtiyahSegal} and effective orbifolds in \cite{AdemRuanTwisted}, see also
\cite{AdemLeidaRuan,KuhnCharacter,Slominska,tomDieckTGRepThry}, and to the usual Euler characteristic of a resolution
in \cite{HirzebruchHoefer}. As well, the orbifold Euler characteristic is equal to the usual Euler characteristic of
the \emph{inertia orbifold}; see \cite{AdemLeidaRuan}.

For global quotients, the orbifold Euler characteristic was identified as part of
a sequence of higher-order orbifold Euler characteristics in \cite{BryanFulman,AtiyahSegal}.
Tamanoi \cite{TamanoiMorava,TamanoiCovering} identified these higher-order orbifold Euler characteristics
as the $\Z^\ell$-extensions of the orbifold Euler characteristic in a context where the $\Gamma$-extension
was defined for each finitely generated group $\Gamma$. These $\Gamma$-orbifold Euler characteristics were
generalized to non-global quotient orbifolds by the authors in \cite{FarsiSeaGenOrbEuler} using the
generalized twisted sectors defined in \cite{FarsiSeaVectorFld,FarsiSeaGenTwistSec}. The $\Gamma$-orbifold
Euler characteristics are the usual Euler characteristic of the \emph{orbifold of $\Gamma$-sectors},
with $\Gamma = \Z$ corresponding to the inertia orbifold and hence the orbifold Euler characteristic
of Equation~\eqref{eq:StringOrbEuler}. Note that another, generally rational, numerical invariant, here called
the \emph{Euler-Satake characteristic} (sometimes also called the \emph{orbifold Euler characteristic}, among
other names) has played a significant role in some of the above references, though does not appear to be as
relevant in the context considered here. See Remarks~\ref{rem:EulerSatake} and \ref{rem:EulerSatake2}.

In \cite{GZEMH-HigherOrbEulerCompactGroup}, Gusein-Zade, Luengo, and Melle-Hern\'{a}ndez defined
a far-reaching generalization of the orbifold Euler characteristic and higher-order orbifold Euler
characteristics, which we will see in Theorem~\ref{thrm:TranslationECequal} correspond to
$\Gamma = \Z^\ell$, to the case of a sufficiently
nice $G$-space $X$ where $G$ is a compact Lie group; see Definition~\ref{def:GZEC} below.
Using integration with respect to the Euler characteristic \cite{ViroIntegralEulerChar}, the sum in
Equation~\eqref{eq:StringOrbEuler} is reinterpreted as an integral over the space of conjugacy classes
of the group, yielding a definition that is well-defined in the presence of infinite isotropy groups.
The $\Z^\ell$-orbifold Euler characteristics are then defined recursively as an iterated integral.

One natural question about this definition, which is the initial motivation for this paper,
is whether the generalized orbifold Euler characteristics depend on the specific presentation of a quotient
$G\backslash X$ as a $G$-space or are invariant under equivalences of such presentations.
An appropriate framework in which to make this question precise is that of proper topological groupoids,
where the appropriate notion of equivalence is that of Morita equivalence.
There are multiple possible approaches to extending the definitions of the generalized orbifold Euler characteristics of
\cite{GZEMH-HigherOrbEulerCompactGroup} to topological groupoids, and the first main goals of this paper are to
give these formulations and show that they coincide. In order to restrict to topological groupoids whose orbit spaces and strata
are sufficiently tame so that their Euler characteristics are defined, we introduce the use of o-minimal structures
to study the orbit spaces of topological groupoids, which may be of interest for other applications.
Specifically, we define the $\Gamma$-Euler characteristics for \emph{orbit space definable groupoids}
introduced in Section~\ref{subsec:OrbitDefin}, defined with respect to an o-minimal structure on $\R$
which we will always assume contains the semialgebraic sets. These groupoids satisfy minimal hypotheses required to define
the $\Gamma$-Euler characteristics and include the case of cocompact proper Lie groupoids as well
as other important cases; see Proposition~\ref{prop:LieGpdOrbitDefin} and
Corollary~\ref{cor:DefinGpdOrbitDefin}. We in particular consider the case where the spaces of objects and arrows have
compatible structures as affine definable spaces, inducing a definable structure on the orbit space; see
Proposition~\ref{prop:DefinGpdOrbitDefin}.
The additional structure they carry, an embedding of the orbit space into $\R^n$,
ensures that the orbit space and its orbit type strata have well-defined Euler characteristics, but by the results of
\cite{BekeInvarEC} recalled in Theorem~\ref{thrm:ECHomeoInvar} below, the Euler characteristics do not depend on
the embedding nor on the choice of o-minimal structure. Hence, we view the orbit space definable structure of a groupoid
as a device used to define and compute Euler characteristics, which are invariants of the underlying topological groupoids.
The o-minimal structure of semialgebraic sets is adequate for
almost all of the cases we have in mind, and the reader unfamiliar with more general notions of definability
is welcome to replace ``definable" with ``semialgebraic" throughout the paper. However, the minor technical
challenges of considering a larger o-minimal structure avoid unnecessary restrictions on the descriptions of the
spaces to which these results can be applied; see Section~\ref{subsec:OrbitDefin}.

For a translation groupoid $G\ltimes X$ associated to a $G$-space as above, the
Euler characteristics of \cite{GZEMH-HigherOrbEulerCompactGroup} are, loosely speaking, defined by integrating
the Euler characteristic of the fixed points of a group element over the group factor. One approach to generalizing
to a topological groupoid $\mathcal{G}$, yielding the \emph{$\Gamma$-inertia Euler characteristic} of
Definition~\ref{def:GamInertiaDefinable}(ii), is to realize
this integral as the usual Euler characteristic of a topological space, the \emph{$\Gamma$-inertia space},
which is the orbit space of the \emph{$\Gamma$-inertia groupoid}; see
Definition~\ref{def:GamInertia}. The $\Gamma$-inertia space of $\mathcal{G}$ depends only on the Morita equivalence class of
$\mathcal{G}$ as a topological groupoid by Lemma~\ref{lem:InertHomomorph}, so this definition is easily seen to be Morita
invariant, settling the question above. Because the $\Gamma$-inertia space is a generalization of the
inertia orbifold and orbifold of $\Gamma$-sectors to arbitrary topological groupoids,
this generalizes the characterization of the orbifold Euler characteristic of an orbifold as the Euler characteristic
of the inertia orbifold. The $\Z^\ell$-inertia groupoid has been considered in the context of equivariant $K$-theory in
\cite{AdemGomez}. In the case $\Gamma = \Z$, the $\Z$-inertia groupoid (called simply the \emph{inertia groupoid})
has appeared implicitly in the study of cyclic homology of convolution algebras of translation groupoids in
\cite{BrylinskiCyclic} and plays an important role in string topology
\cite{BehrendGinoteaStringTop,BehrendXuDiffStackGerbes}.
It was studied for translation groupoids in \cite{FarsiPflaumSeaton1,PPTGrauertGrothendieck},
and proper Lie groupoids in \cite{FarsiPflaumSeaton2} where the inertia groupoid was shown explicitly
to have the structure of a \emph{differentiable stratified groupoid}; see also \cite{CrainicMestre}.

Another approach to generalization, yielding the \emph{$\Gamma$-Euler characteristic} of Definition~\ref{def:ECGroupoid},
is to mimic the approach of \cite{GZEMH-HigherOrbEulerCompactGroup}, taking advantage of the product structure of a translation
groupoid $G\ltimes X$ to integrate along one factor. This is done by reinterpreting the integral over $\Ad_G\!\backslash G$
as an integral over the quotient space $G\backslash X$, yielding a formulation that generalizes readily to groupoids.
The two approaches are shown to yield Euler characteristics that coincide by an application of Fubini's Theorem, see Theorem~\ref{thrm:GamECInertia},
and to equal the Euler characteristics of Gusein-Zade, Luengo, and Melle-Hern\'{a}ndez in the case of a translation
groupoid by a second application of Fubini's Theorem; see Theorem~\ref{thrm:TranslationECequal}.
This yields a definition of orbifold Euler characteristics for a large class of proper topological groupoids,
generalizing both beyond the class of translation groupoids and to
$\Gamma$-Euler characteristics for finitely presented discrete groups $\Gamma$ that need not be
free abelian. Moreover, even in the case where the groupoid is (Morita equivalent to) the extension of a
translation groupoid by a bundle of compact groups, the generalized $\Gamma$-Euler characteristics do not
appear to be directly related to that of the translation groupoid; see Example~\ref{ex:NonabelExtension}.
Note that for non-orbifold groupoids, $\Gamma$ must be finitely presented, not merely finitely generated,
as explained in Remark~\ref{rem:FinitePres}.

One of our primary motivations is the case of a cocompact proper Lie groupoid, and we consider this case explicitly
throughout the paper. However, as was emphasized in
\cite{FarsiPflaumSeaton2}, the construction of the inertia groupoid leaves the category
of Lie groupoids. Moreover, as the Euler characteristic depends only on the underlying topology,
restricting consideration to proper Lie groupoids is artificial for our objectives.

A synopsis of the outline and main results of this paper is as follows.
In Section~\ref{sec:Background}, we summarize relevant background information
on topological and Lie groupoids, Morita equivalence, o-minimal structures, affine definable spaces, integration with respect to the Euler
characteristic, Fubini's theorem, and orbifold Euler characteristics. In Section~\ref{sec:OrbitDefinGpoids}, we introduce
orbit space definable groupoids and the $\Gamma$-Euler characteristics of these groupoids, showing that
they include cocompact proper Lie groupoids as well as semialgebraic translation groupoids. Section~\ref{sec:ECSpace}
introduces the $\Gamma$-inertia groupoid of a topological groupoid and establishes its basic properties, and the
$\Gamma$-inertia Euler characteristic is introduced in Section~\ref{subsec:GpoidECSpace}. We then
prove our first main result, Theorem~\ref{thrm:GamECInertia}, stating that the $\Gamma$-Euler characteristic of an
orbit space definable groupoid coincides with the $\Gamma$-inertia Euler characteristic when both are defined.
This is used to establish several properties of both Euler characteristics including Morita
invariance (Corollary~\ref{cor:ChiMoritaInvar}), additivity (Corollary~\ref{cor:GamECAdditive}), and
multiplicativity (Lemma~\ref{lem:GamECMultiplic}). In Section~\ref{sec:TranslationProperLie}, we consider the
special cases of cocompact proper Lie groupoids and translation groupoids. In Section~\ref{subsec:GpoidECTranslation},
we prove Theorem~\ref{thrm:TranslationECequal}, demonstrating that the $\Z^\ell$- and $\Z^\ell$-inertia Euler characteristics of a
translation groupoid coincide with the higher-order orbifold Euler characteristics of
\cite{GZEMH-HigherOrbEulerCompactGroup},
as well as Theorem~\ref{thrm:TranslationECNonIterat}, giving a non-iterative definition of the higher-order orbifold Euler
characteristics of a semialgebraic translation groupoid as well as the analogous $\Gamma$-extensions for $\Gamma$ that are not
free abelian. Section~\ref{subsec:GpoidECLie} considers the case of cocompact proper Lie groupoids and uses
the slice theorem of \cite{PPTOrbitSpace}, see also
\cite{Zung,WeinsteinLineariz,CrainicStruchiner,delHoyoFernandesMetricLieGpd},
to give an alternate formulation of the $\Gamma$-Euler
characteristics in this case: Theorem~\ref{thrm:GamECLieGIntegral}, which is closer to the original definition
of \cite{GZEMH-HigherOrbEulerCompactGroup}.
In Section~\ref{subsec:Extensions}, we consider extensions of translation groupoids by bundles of compact
Lie groups, which by the work of Trentinaglia \cite{TrentinagliaGlobStruc,TrentinagliaRoleReps} include a
large class of, and conjecturally all, proper Lie groupoids. We show in Theorem~\ref{thrm:AbelExtension}
that when such an extension is abelian, there is a simple relationship between the $\Gamma$-Euler
characteristics of the extension, the translation groupoid, and the bundle of compact Lie groups. However,
we indicate that in the non-abelian case, such a relationship cannot be expected, and hence the generalization
of the orbifold Euler characteristics given here is a nontrivial generalization from the translation groupoid case.


\section*{Acknowledgements}
We would like to thank the referee for helpful comments and suggestions that significantly improved
this manuscript.


\section{Background and Definitions}
\label{sec:Background}

In this section, we review the essential background information to fix the language and notation.


\subsection{Topological and Lie groupoids}
\label{subsec:BackGpoids}

In this section, we briefly recall facts about topological and Lie groupoids. For more details, the reader is referred to
\cite{BrownSurvey,MrcunThesis,MoerdijkToposGpoid,RenaultGpdCstar,MetzlerTopSmoothStacks} for topological groupoids and
\cite{MackenzieLGLADiffGeom,MoerdijkMrcun,WangThesis} for Lie groupoids.

Recall that a \emph{groupoid} $\mathcal{G}$ is given by two sets
$\mathcal{G}_0$ of objects and $\mathcal{G}_1$ of arrows along with the \emph{structure functions}:
the source $s_{\mathcal{G}}\co\mathcal{G}_1\to\mathcal{G}_0$,
target $t_{\mathcal{G}}\co\mathcal{G}_1\to\mathcal{G}_0$,
unit $u_{\mathcal{G}}\co\mathcal{G}_0\to\mathcal{G}_1$ written $u_{\mathcal{G}}(x) = 1_x$,
inverse $i_{\mathcal{G}}\co\mathcal{G}_1\to\mathcal{G}_1$,
and multiplication $m_{\mathcal{G}}\co\mathcal{G}_1 \ftimes{s_\mathcal{G}}{t_{\mathcal{G}}}\mathcal{G}_1\to\mathcal{G}_1$
written $m_{\mathcal{G}}(g,h) = gh$. For a point $x\in\mathcal{G}_0$, the \emph{isotropy group of $x$},
denoted $\mathcal{G}_x^x$, is the set $(s,t)^{-1}(x,x)$.
A \emph{topological groupoid} is a groupoid such that $\mathcal{G}_0$ and $\mathcal{G}_1$ are
topological spaces and the structure functions are maps, i.e., continuous functions.
The \emph{orbit space} $|\mathcal{G}|$ of a topological groupoid $\mathcal{G}$ is the quotient
space of $\mathcal{G}_0$ by the equivalence relation identifying $x, y\in\mathcal{G}_0$ if there
is a $g\in\mathcal{G}_1$ such that $s_{\mathcal{G}}(g) = x$ and $t_{\mathcal{G}}(g) = y$. We let
$\pi_{\mathcal{G}}\co\mathcal{G}_0\to|\mathcal{G}|$ denote the orbit map, which sends $x\in\mathcal{G}_0$
to its \emph{$\mathcal{G}$-orbit}, the equivalence class of $x$, denoted $\mathcal{G}x$.
Note that the subscript $\mathcal{G}$ on the structure maps and $\pi$ will
be omitted when there is no possibility of confusion.
In this paper, we will always assume that a topological groupoid is \emph{open}, meaning
$s$ is an open map. This ensures that the $\pi$ is open and hence
a quotient map; see \cite[Proposition 2.11]{TuNonHausdorff}. We will also always assume that
$\mathcal{G}_0$ and $\mathcal{G}_1$ are Hausdorff and paracompact.

A topological groupoid $\mathcal{G}$ is \emph{proper} if the map
$(s,t)\co\mathcal{G}_1\to\mathcal{G}_0\times\mathcal{G}_0$ is a proper map, i.e., the preimage of a compact
set is compact. It is \emph{cocompact} if the orbit space $|\mathcal{G}|$ is compact, and it is
\emph{\'{e}tale} if $s$ is a local homeomorphism.
A \emph{Lie groupoid} is a topological groupoid such that $\mathcal{G}_0$ and $\mathcal{G}_1$ are
smooth manifolds without boundary, the structure maps are smooth, and $s$ is a submersion.

If $G$ is a topological group and $X$ is a topological $G$-space, i.e., a topological space on which
$G$ acts continuously, then the \emph{translation groupoid}
$G\ltimes X$ has space of objects $X$ and space of arrows $G\times X$ with
$s(g,x) = x$, $t(g,x) = gx$, $u(x) = (1,x)$, $i(g,x) = (g^{-1}, gx)$, and composition
$(h,gx)(g,x) = (hg,x)$. The orbit space of $G\ltimes X$ is denoted $|G\ltimes X|$ or
$G\backslash X$. The translation groupoid is a topological groupoid. If
$X$ is a smooth $G$-manifold without boundary, i.e., the map $G\times X\to X$ is smooth, then
$G\ltimes X$ is a Lie groupoid.

If $\mathcal{G}$ is a topological groupoid and $X$ is a topological space, a \emph{left action}
of $\mathcal{G}$ on $X$ is given by an anchor map $\alpha_X\co X\to\mathcal{G}_0$ and an
action map $\mathcal{G}_1\ftimes{s}{\alpha_X}X \to X$, written $(g,x)\mapsto g\ast x$, such that
$\alpha_X(g\ast x) = t(g)$, $h\ast(g\ast x) = (hg)\ast x$, and $1_{\alpha_X(x)} \ast x = x$
for all $x\in X$ and $g,h\in\mathcal{G}_1$ such that these expressions are defined. We then refer to $X$ as
a \emph{$\mathcal{G}$-space}. The \emph{translation groupoid} $\mathcal{G}\ltimes X$ is defined similarly to
that of a group action;
the space of objects is $X$, space of arrows is $\mathcal{G}_1 \ftimes{s_{\mathcal{G}}}{\alpha_X} X$,
$s_{\mathcal{G}\ltimes X}(g,x) = x$, $t_{\mathcal{G}\ltimes X}(g,x) = g\ast x$,
$u_{\mathcal{G}\ltimes X}(x) = (1_{\alpha_X(x)},x)$, $i_{\mathcal{G}\ltimes X}(g,x) = (g^{-1},g\ast x)$,
the product is given by $(h, g\ast x)(g,x) = (hg, x)$, and the orbit space is denoted either
$|\mathcal{G}\ltimes X|$ or $\mathcal{G}\backslash X$. A Lie groupoid action
on a smooth manifold $X$ is defined similarly, requiring the anchor and action maps to be smooth, and then
$\mathcal{G}\ltimes X$ is a Lie groupoid.

If $\mathcal{G}$ and $\mathcal{H}$ are topological (respectively, Lie) groupoids, a \emph{homomorphism} $\Phi\co\mathcal{G}\to\mathcal{H}$ consists of two continuous (respectively, smooth) maps $\Phi_0\co\mathcal{G}_0\to\mathcal{H}_0$
and $\Phi_1\co\mathcal{G}_1\to\mathcal{H}_1$ that preserve each of the structure maps. A homomorphism of topological
(respectively, Lie) groupoids is an \emph{essential equivalence}, also known as a \emph{weak equivalence}, if the map
$t_\mathcal{H}\circ\pr_1\co\mathcal{H}_1\ftimes{s_{\mathcal{H}}}{\Phi_0}\mathcal{G}_0\to\mathcal{H}_0$, where $\pr_1$
is the projection onto the first factor, is a surjection admitting local sections (respectively, a surjective submersion) and the diagram
\[
    \xymatrix{
        \mathcal{G}_1
            \ar[d]_{(s_\mathcal{G},t_\mathcal{G})}
            \ar@{>}[r]^{\Phi_1}
        &\mathcal{H}_1
            \ar[d]^{(s_\mathcal{H},t_\mathcal{H})}
        \\
        \mathcal{G}_0\times\mathcal{G}_0
            \ar[r]^{\Phi_0\times\Phi_0}
        &
        \mathcal{H}_0\times\mathcal{H}_0,
    }
\]
is a fibred product of spaces (respectively, of smooth manifolds); see \cite[Definition~58 and Remark~59]{MetzlerTopSmoothStacks}.
Note that a surjective map admits local sections if each point in the codomain is contained in an open neighborhood on which the map
has a right inverse.
Two groupoids $\mathcal{G}$ and $\mathcal{H}$ are \emph{Morita equivalent}
as topological (Lie) groupoids if there is a third topological (Lie) groupoid $\mathcal{K}$ and essential equivalences
(of Lie groupoids) $\mathcal{G}\leftarrow\mathcal{K}\rightarrow\mathcal{H}$.


\subsection{Definable sets and integration with respect to the Euler characteristic}
\label{subsec:BackDefin}

In order to define the integral with respect to the Euler characteristic, we need to work within
a fixed o-minimal structure on $\R$. We give a very brief summary of the background and refer the reader
to \cite{vandenDriesBook} for more details on o-minimal structures. For the integral with respect
to the Euler characteristic; see \cite{ViroIntegralEulerChar} or \cite[Sections 2--4]{CurryGhristEAEulerCalc}.

Recall \cite[Chapter 1, Definition~(2.1)]{vandenDriesBook} that a \emph{structure} on $\R$ is a collection
$(\mathcal{S}_n)_{n\in\N}$ with the following properties: Each $\mathcal{S}_n$ is a boolean algebra of subsets of
$\R^n$ (i.e. is closed under unions and set differences) such that $A\times\R\in\mathcal{S}_{n+1}$
and $\R\times A\in\mathcal{S}_{n+1}$ for each $A\in\mathcal{S}_n$. Moreover, each $\mathcal{S}_n$ contains
the diagonal $\{(x,\ldots,x):x\in\R\}\subset\R^n$, and the projection of any $A\in\mathcal{S}_{n+1}$ to $\R^n$
obtained by dropping the last coordinate is contained in $\mathcal{S}_n$. A structure $(\mathcal{S}_n)_{n\in\N}$
is \emph{o-minimal} if $\mathcal{S}_1$
is the set of finite unions of points and open intervals and $\mathcal{S}_2$ contains $\{(x,y)\in\R^2 : x < y\}$.
Once an o-minimal structure is fixed, a set $A\subseteq\R^n$ is a \emph{definable set} if it is an element of $\mathcal{S}_n$.
If $A\subset\R^n$ is a definable set, a function $f\co A\to\R^m$ is a \emph{definable function} if its graph
$\{(x,f(x))\}\subset\R^{m+n}$ is definable.

An important example of an o-minimal structure, and the one that we primarily have in mind, is that
of \emph{semialgebraic sets}; see \cite[p.~1]{vandenDriesBook}.
A subset $A\subseteq\R^n$ is \emph{semialgebraic} if it is a finite union of solution sets
to finite systems of polynomial equations and inequalities, i.e., sets of the form
$f_1(x_1,\ldots,x_n) = \cdots = f_k(x_1,\ldots,x_n) = 0$ and
$g_1(x_1,\ldots,x_n) > 0,\ldots,g_m(x_1,\ldots,x_n) > 0$ (where $k$ or $m$ may be $0$).
We consider $A$ as a topological space using the subspace topology from $\R^n$.
We will sometimes restrict to the o-minimal structure of semialgebraic sets, and readers
are welcome to consider only this o-minimal structure as noted in the introduction.
When we work more generally, \emph{we will always assume that an o-minimal structure contains
the semi-algebraic sets}.

\begin{definition}[{\cite[Chapter~1, Section~3; Chapter~10, Section~1]{vandenDriesBook}}]
\label{def:definable}
Fix an o-minimal structure on $\R$.  An \emph{affine definable space} is a topological space $X$ along with a topological embedding
$\iota_X\co X\to \R^n$, i.e. a continuous function that is a homeomorphism onto its image, such that $\iota_X(X)$
is definable. We will often denote $\iota_X$ simply as $\iota$ if there is no danger of confusion.
If $(X,\iota_X)$ is an affine definable space, then a \emph{definable subset} $A$ of $X$ is a $A\subset X$ such that
$\iota_X(A)$ is definable. If $(X,\iota_X)$ and $(Y,\iota_Y)$ are affine definable spaces with $\iota_Y\co Y\to\R^m$, then a continuous
function $f\co X\to Y$ is a \emph{morphism of affine definable spaces} if the function $\iota_Y\circ f\circ\iota_X^{-1}\co\iota_X(X)\to\R^m$
is a definable function.
\end{definition}

\begin{remark}
\label{rem:DefSpacesOperations}
As the collection of definable sets is closed under finite unions, intersections, closures, and interiors
\cite[Chapter~1, Lemma~(3.4)]{vandenDriesBook}, the collection of definable subsets of an affine definable space is as well.
Similarly, the product of affine definable spaces is an affine definable space using the product embedding.
If $f\co X\to Y$ is a morphism of affine definable spaces $(X,\iota_X)$ and $(Y,\iota_Y)$, then $\iota_Y\circ f(X)$ is the projection
of the graph of $\iota_Y\circ f\circ\iota_X^{-1}$ onto the second factor so that $f(X)$ is a definable subset of $Y$.
Similarly, images and preimages of definable subsets via morphisms of affine definable spaces are definable, and restrictions of
morphisms to definable subsets are morphism of affine definable spaces.
\end{remark}

\begin{remark}
\label{rem:GenDefinableSpaces}
There is a more general notion of a (non-affine) definable space as a topological space that is covered by an atlas of definable sets with
definable transition functions; see \cite[Chapter 10, Section~(1.2)]{vandenDriesBook}. However, by \cite[Chapter 10, Theorem~(1.8)]{vandenDriesBook},
every definable space $(X,\iota)$ such that $X$ is a regular topological space is affine. Since our intended applications are to
regular spaces, we restrict consideration to affine definable spaces.
\end{remark}

\begin{remark}
\label{rem:DefinableSpaceNotUnique}
Of course, a topological space $X$ may admit many distinct structures as an affine definable space even using the same o-minimal structure,
and the definable subsets of $X$ need not coincide. However, our intended application of this structure is to define the Euler characteristic,
which will not depend on these choices; see Theorem~\ref{thrm:ECHomeoInvar} below.
\end{remark}

We now recall the definition of the Euler characteristic of a definable set.
Fix an o-minimal structure on $\R$. Then for each definable set $A\subseteq\R^n$, the
\emph{Euler characteristic} $\chi(A)$ of $A$ is defined. Specifically, any definable set $A$ can be
partitioned into finitely many simple sets called \emph{cells}, the Euler characteristic
of a cell of dimension $d$ is defined to be $(-1)^d$, and then $\chi(A)$ is defined as the
sum of the Euler characteristics of its cells; see \cite[Chapters 3, 4]{vandenDriesBook}.
We extend the definition of the Euler characteristic to affine definable sets $(X,\iota)$ by defining
$\chi(X) = \chi\big(\iota(X)\big)$.
When $X$ is compact, this definition of the Euler characteristic coincides with the usual definition of the Euler
characteristic of a triangulable topological space, but $\chi$ is not homotopy-invariant in general.
For instance, the Euler characteristic of a finite or infinite open interval (a single $1$-dimensional cell)
is $-1$, and the Euler characteristic of a finite or infinite half-open interval (the union of a $0$-cell,
with Euler characteristic $1$, and a $1$-cell, with Euler characteristic $-1$) is $0$.
This definition yields an Euler characteristic that is invariant under definable homeomorphisms
and is finitely (but not countably) additive. Moreover, by the following result of Beke, it does not depend
on the choice of o-minimal structure and hence depends only on the underlying topological space $X$; see also
\cite{McCroryParusinski}.

\begin{theorem}[{\cite[Theorem 2.2]{BekeInvarEC}}]
\label{thrm:ECHomeoInvar}
Let $\mathcal{S}$ and $\mathcal{S}^\prime$ be o-minimal structures on $\R$, let $A$ be a definable subset of $\R^n$ with respect to
$\mathcal{S}$, and let $B$ be a definable subset of $\R^m$ with respect to $\mathcal{S}^\prime$. Let $\chi(A)$ be the Euler characteristic
of $A$ with respect to $\mathcal{S}$ and let $\chi^\prime(B)$ be the Euler characteristic
of $B$ with respect to $\mathcal{S}^\prime$. If $A$ and $B$ are homeomorphic, then $\chi(A) = \chi^\prime(B)$.
\end{theorem}

\begin{remark}
\label{reem:ECHomeoInvar}
Because the Euler characteristic of an affine definable space $(X,\iota)$ is defined to be $\chi(X) = \chi\big(\iota(X)\big)$,
it follows immediately from Theorem~\ref{thrm:ECHomeoInvar} that $\chi(X)$ does not depend on $\iota$ nor on the choice of o-minimal structure,
and hence is an invariant of topological spaces that admit the structure of affine definable sets.
\end{remark}

We will need the notion of a definably proper equivalence relation and quotient;
see also \cite[Theorem~1.4]{BrumfielQuot} for the semialgebraic case. Recall
\cite[Chapter 6, Definition~(4.4)]{vandenDriesBook} that a definable continuous function $f\co A\to B$ between
definable sets is \emph{definably proper} if the preimage of every compact definable subset of $B$ is compact.
Clearly, if $f$ is definable, continuous, and proper, then it is definably proper.
An equivalence relation $E\subseteq A\times A$ is a \emph{definably proper equivalence relation} if
$E$ is definable and the projection $\pr_1\co E\to A$ onto the first factor is definable
\cite[Chapter 10, Definition~(2.13)]{vandenDriesBook}. If $E$ is an equivalence relation
on $A$, then a \emph{definably proper quotient of $A$ by $E$} is a definable set $B\subset\R^k$ and a definable
continuous surjective map $\pi\co A\to B$ such that $\pi(x_1) = \pi(x_2)$ if and only if $(x_1,x_2)\in E$,
and $\pi$ is definably proper. An equivalence relation admits a definably proper quotient if and only if
it is a definably proper equivalence relation \cite[Chapter 10, (2.13) and Theorem~(2.15)]{vandenDriesBook}.
We extend this notion to an affine definable space $(X,\iota)$ by identifying $X$ with $\iota(X)$:
a \emph{definably proper equivalence relation} $E$ on $X$ is an equivalence relation $E\subset X\times X$
such that $(\iota\times\iota)(E)$ is a definably proper equivalence relation on $\iota(X)$,
and a \emph{definably proper quotient of $X$ by $E$} is a definably proper quotient of $\iota(X)$
by $(\iota\times\iota)(E)$.

\begin{remark}
\label{rem:QuotUnique}
By \cite[Chapter 10, Definition~(2.2) and Remark~(2.3)]{vandenDriesBook}, a
definably proper quotient is unique up to definable homeomorphism.
If $\pi\co\iota(X)\to B\subset\R^k$ is a definably proper quotient of $X$ and
$\rho\co X\to E\backslash X$ denotes the usual topological quotient of $X$ by the equivalence relation $E$,
then the embedding $\iota$ induces a homeomorphism $E\backslash\iota\co E\backslash X\to B\subset\R^k$ such that
$E\backslash\iota\circ\rho=\pi\circ\iota$; see \cite[Corollary ~22.3]{Munkres}. Hence, $(E\backslash X, E\backslash\iota)$
is an affine definable spaces whose image in $\R^k$ as a definably proper quotient of $X$ by $E$.
\end{remark}

An \emph{affine definable topological group} is a pair $(G,\lambda)$ such that $G$ is a topological group, $\lambda\co G\to\R^\ell$
is a topological embedding, and the product and inverse maps are morphisms of definable spaces; equivalently,
the induced product $\lambda(G)\times\lambda(G)\to\lambda(G)\subset\R^\ell$ given by $\big(\lambda(g),\lambda(h)\big)\mapsto\lambda(gh)$
and inverse $\lambda(G)\to\lambda(G)\subset\R^\ell$ given by $\lambda(g)\mapsto\lambda(g^{-1})$ are definable maps.
For a definable topological group $(G,\lambda)$, a \emph{definable $G$-set $A$} is topological $G$-space $A$ such that $A\subset\R^n$ is
a definable set and the action $\lambda(G)\times A\to A$ defined by $\lambda(g)a = ga$ is definable. An \emph{affine definable $G$-space}
is an affine definable space $(X,\iota)$ such that $X$ is a $G$-space and $G\times X\to X$ is a morphism of affine definable spaces,
i.e., the map $\lambda(G)\times\iota(X)\to\iota(X)$ given by $\big(\lambda(g),\iota(x)\big)\mapsto\iota(gx)$ is definable.
If $G$ is compact, then the orbit space $G\backslash X$ admits the structure of an affine definable space whose image in $\R^k$
is a definably proper quotient by \cite[Chapter 10, Corollary~(2.18)]{vandenDriesBook} and Remark~\ref{rem:QuotUnique}.

Following \cite{ChoiParkSuhSemialgSlice,ParkSuhLinEmbed}, we will consider semialgebraic groups and $G$-sets to be subsets of Euclidean space. That is,
a \emph{semialgebraic group $G$} is a semialgebraic set $G\subset\R^\ell$ that is a topological group with the subspace topology such that the multiplication
and inversion maps are definable in the o-minimal structure of semialgebraic sets. If $G$ is a semialgebraic group, a \emph{semialgebraic $G$-set $A$}
is a semialgebraic set $A\subset\R^n$ that is also a $G$-space (with the subspace topologies on $G$ and $A$) such that the action map $G\times A\to A$
is definable in the o-minimal structure of semialgebraic sets. See \cite[Section~2]{ChoiParkSuhSemialgSlice} and \cite[Section~2]{ParkSuhLinEmbed}
for additional details and background. If one uses any o-minimal structure that contains the semialgebraic sets, a semialgebraic group $G$ is an
affine definable topological group, and a semialgebraic $G$-set $A$ is an affine definable $G$-space, with respect to the identity embeddings
of $G$ and $A$.

Any compact Lie group $G$ admits a faithful representation on $\R^m$ \cite[Chapter~VI, Section~XII, Theorem~4]{ChevalleyThryLieGrps}
and hence admits an isomorphism $\lambda\co G\to \lambda(G)\leq\GL(m,\R)\subset\R^{m^2}$ to a compact linear algebraic group $\lambda(G)$, i.e., a subgroup
$\lambda(G)\leq\GL(m,\R)$ that is the solution of finitely many polynomial equations \cite[Section~99, Theorem~3 and Corollary]{Zelobenko}. It follows that
$\lambda(G)$ is a semialgebraic group and hence $(G,\lambda)$ is an affine definable topological group with respect to any o-minimal structure
that contains the semialgebraic sets (as the graphs of the multiplication and inverse maps are semialgebraic, and hence contained in any larger
o-minimal structure). If $\lambda^\prime\co G\to\lambda^\prime(G)\leq\GL(m^\prime,\R)$ is another choice of isomorphism of $G$ such that $\lambda^\prime(G)$
is a semialgebraic group, then $\lambda^\prime\circ\lambda^{-1}$ is a group isomorphism and hence a semialgebraic function by
\cite[Corollary~2.3]{ParkSuhLinEmbed}; hence, the semialgebraic structure a compact Lie group inherits as a linear algebraic group is unique.
For this reason, and particularly because most
of the groups we consider will be compact, we will usually consider compact Lie groups with their unique structure as definable topological groups with no loss of
generality. We will not use this fact, but it is interesting to note: if $A\subset\R^n$ is a definable set that is also a group such
that the multiplication and inverse functions are definable but not necessarily continuous, then $A$ admits a topology with respect to which it is a Lie group; see \cite{Pillay}.

We will make frequent use of the following and hence include a proof for clarity.

\begin{lemma}
\label{lem:SemialgEquivarFunc}
Suppose $(G, \lambda)$ is a compact affine definable topological group, $(X, \iota_X)$ and $(Y, \iota_Y)$ are affine definable $G$-spaces, and
$f\co X\to Y$ is a $G$-equivariant morphism of affine definable spaces. Giving the orbit spaces $G\backslash X$ and $G\backslash Y$ the structures
of affine definable spaces as in Remark~\ref{rem:QuotUnique}, the induced map
$\overline{f}\co G\backslash X\to G\backslash Y$ is a morphism of affine definable spaces.
\end{lemma}
\begin{proof}
By Remark~\ref{rem:QuotUnique}, the structures of $G\backslash X$ and $G\backslash Y$ as affine
definable spaces do not depend on the choice of definably proper quotients.
For simplicity, we identify $X$ with $\iota_X(X)$, $Y$ with $\iota_Y(Y)$, $G$ with $\lambda(G)$, and $f$ with
$\iota_Y\circ f\circ\iota_X^{-1}$, and hence work with the corresponding definable sets in Euclidean space.
Define a $(G\times G$)-action on the graph of $f$ in $X\times Y$ by
$(g_1,g_2)(x,f(x)) = (g_1 x, g_2 f(x))$; the graph of this action is the product of the graphs of
the $G$-actions on $X$ and $f(Y)$ and hence is a definable set; see Remark~\ref{rem:DefSpacesOperations}. Then
$\operatorname{graph}(f)$ is a definable $(G\times G)$-set, and the quotient
$(G\times G)\backslash\operatorname{graph}(f)$ is the graph of $\overline{f}$.
\end{proof}

We now review the integration of constructible functions on an affine definable space with respect to the Euler characteristic.
Let $(X,\iota)$ be an affine definable space. A function $f\co X\to\Z$ is \emph{constructible} if $f^{-1}(k)$
is a definable subset of $X$ for each $k\in\Z$. If $f$ is a bounded constructible function, then the
\emph{integral of $f$ over $X$ with respect to the Euler characteristic} is defined to be
\[
    \int_X f(x)\,d\chi(x) =  \sum\limits_{k\in\Z} k \chi\big( f^{-1}(k)\big).
\]
Note that the image of $f$ in $\Z$ is finite so that this sum is finite as well. Hence the integral
is well-defined by the finite additivity of $\chi$.

As noted in \cite[Proposition 4.2]{CurryGhristEAEulerCalc}, it is often easier to compute the above integral using
the equivalent formulation
\begin{equation}
\label{eq:IntEC}
    \int_X f(x)\,d\chi(x)
        =   \sum\limits_{k=0}^\infty \Big[\chi\big( f^{-1}(\{ j : j > k\})\big) - \chi\big(f^{-1}(\{ j : j < -k\} ) \big)\Big].
\end{equation}
Note that if $f(x) = c$ is constant, then $\int_X f(x) \, d\chi(x) = c \chi(X)$.
Note further that if $X$ is partitioned into finitely many disjoint definable subsets $X_1,\ldots, X_m$, then
$\int_X f(x)\,d\chi(x) = \sum_{i=1}^m \int_{X_i} f(x)\,d\chi(x)$.

\begin{remark}
\label{rem:intEChomeoInvar}
An immediate consequence of Theorem~\ref{thrm:ECHomeoInvar} and either the definition or Equation~\eqref{eq:IntEC}
is that $\int_X f(x)\,d\chi(x)$ does not depend on the choice of o-minimal structure nor the embedding $\iota$.
\end{remark}

The following version of Fubini's Theorem will play an important role; see \cite[Theorem 4.5]{CurryGhristEAEulerCalc}
and \cite[3.A]{ViroIntegralEulerChar}.

\begin{theorem}[(Fubini's Theorem)]
\label{thrm:Fubini}
Let $X$ and $Y$ be affine definable spaces, $\varphi\co X\to Y$ a morphism of affine definable spaces, and $f\co X\to\Z$
a bounded constructible function. Then
\[
    \int_X f(x)\,d\chi(x) =   \int_Y \left( \int_{\varphi^{-1}(y)}f(x) \, d\chi(x) \right) \,d\chi(y).
\]
\end{theorem}

When $X$ and $Y$ are definable sets, the proof of Theorem~\ref{thrm:Fubini} is straightforward, because the integral is defined as a finite sum.
The Hardt Trivialisation theorem for definable functions
\cite[Chapter 9, Theorem~(1.7)]{vandenDriesBook} yields a finite cell decomposition of $B$ such that the preimage
of each cell under $\varphi$ is a product of the cell and a definable set on which $\varphi$ coincides with
the projection onto the cell. The theorem extends easily to the case affine definable spaces by considering the images of $X$ and $Y$ under
the corresponding embeddings and applying Remark~\ref{rem:intEChomeoInvar}.


\subsection{Orbifold Euler characteristics}
\label{subsec:BackOrbEuler}

In this section, we briefly recall the construction of the $\Gamma$-Euler characteristic of a
closed orbifold as well as their generalization in \cite{GZEMH-HigherOrbEulerCompactGroup}; see
\cite[Section 2.1 and Definition 4.1]{FarsiSeaGenOrbEuler} for more details.

Recall that an \emph{orbifold} $Q$ (without boundary) can be defined to be the Lie groupoid Morita equivalence class
of a proper \'{e}tale Lie
groupoid, and a choice of groupoid $\mathcal{G}$ from this Morita equivalence class is a \emph{presentation}
of $Q$; see \cite{AdemLeidaRuan}. The \emph{underlying space} of $Q$ is the orbit space $|\mathcal{G}|$
of a presentation $\mathcal{G}$, which does not depend on the choice of $\mathcal{G}$ (up to homeomorphism).
An orbifold is a \emph{global quotient} if it admits a presentation of the form $G\ltimes X$ where $G$ is a
finite group and $X$ is a smooth $G$-manifold without boundary.

Let $\mathcal{G}$ be a proper \'{e}tale Lie groupoid presenting the orbifold
$Q$, and assume that $Q$ is closed, i.e., $\mathcal{G}$ is cocompact.
Let $\Gamma$ be a finitely generated discrete group, considered as a groupoid with a single object,
and let $\HOM(\Gamma,\mathcal{G})$ denote the space of groupoid homomorphisms $\Gamma\to\mathcal{G}$
equipped with the compact-open topology; see Section~\ref{subsec:GamInertia} for a description of this space
in a more general context. The space $\HOM(\Gamma,\mathcal{G})$ inherits the structure
of a smooth $\mathcal{G}$-manifold (possibly with components of different dimensions);
see \cite[Lemma~2.2]{FarsiSeaVectorFld}. Specifically, if $x\in\mathcal{G}_0$ and $V_x$ is a neighborhood of
$x$ in $\mathcal{G}_0$ such that $\mathcal{G}_{|V_x}$ is isomorphic to $\mathcal{G}_x^x\ltimes V_x$,
then a neighborhood in $\HOM(\Gamma,\mathcal{G})$ of the element corresponding to a group homomorphism
$\phi\co\Gamma\to\mathcal{G}_x^x$ is diffeomorphic to the set of fixed points $V_x^{\langle\phi\rangle}$ of
$\phi(\Gamma)$ in $V_x$, and the restriction of the $\mathcal{G}$-action corresponds to the action by conjugation
of the centralizer $C_{\mathcal{G}_x^x}(\phi)$ of $\phi(\Gamma)$ in $\mathcal{G}_x^x$. The translation
groupoid $\mathcal{G}\ltimes\HOM(\Gamma,\mathcal{G})$ is a cocompact orbifold groupoid representing the
\emph{orbifold of $\Gamma$-sectors of $Q$}, denoted $\widetilde{Q}_\Gamma$. When $\Gamma = \Z$,
$\widetilde{Q}_\Z$ is called the \emph{inertia orbifold}, with connected components called \emph{twisted
sectors} \cite[Section~2.5]{AdemLeidaRuan}. The \emph{$\Gamma$-Euler characteristic
$\chi_\Gamma(Q)$ }is defined to be $\chi(\widetilde{Q}_\Gamma) = \chi(|\mathcal{G}\ltimes\HOM(\Gamma,\mathcal{G})|)$,
the Euler characteristic of the orbit space of $\mathcal{G}\ltimes\HOM(\Gamma,\mathcal{G})$.
Many of the various Euler characteristics that have been defined in the literature for orbifolds occur as
$\chi_\Gamma$ for a specific choice $\Gamma$; see \cite[p.~524]{FarsiSeaGenOrbEuler}. In particular, the
string-theoretic orbifold Euler characteristic recalled in Equation~\eqref{eq:StringOrbEuler}, introduced
in \cite{DixonHarveyEA} for global quotients and \cite{Roan} more generally, corresponds to $\chi_\Z$.
For global quotients, the sequence of orbifold Euler characteristics of \cite{BryanFulman} coincides with
$\{ \chi_{\Z^\ell} \}_{\ell\in\N}$, and the generalized orbifold Euler characteristic of
\cite{TamanoiMorava,TamanoiCovering} coincide with the $\chi_\Gamma$ for finitely generated $\Gamma$.
Note that $\HOM(\Z,\mathcal{G})$ is often called the \emph{loop space of $\mathcal{G}$},
$\mathcal{G}\ltimes\HOM(\Z,\mathcal{G})$ is called the \emph{inertia groupoid of $\mathcal{G}$},
and $|\mathcal{G}\ltimes\HOM(\Z,\mathcal{G})|$ is the \emph{inertia space of $\mathcal{G}$}, particularly when
$\mathcal{G}$ is not an orbifold groupoid; we will extend this definition and language to an arbitrary topological groupoid
$\mathcal{G}$ and a finitely presented discrete group $\Gamma$ in Definition~\ref{def:GamInertia} below.
When $\mathcal{G}$ is an orbifold groupoid, the language of twisted sectors or $\Gamma$-sectors is used to
emphasize the fact that $\widetilde{Q}_\Gamma$ is a disjoint union of orbifolds.

\begin{remark}
\label{rem:EulerSatake}
Another important invariant for orbifolds, the \emph{Euler-Satake characteristic
$\chi^{ES}$}, has appeared under various names \cite{SatakeGB,Thurston,LeinsterEulerCharCategory}.
In the simplest case of
a \emph{global quotient orbifold}, presented by $G\ltimes M$ where $G$ is a finite group and $M$
a smooth $G$-manifold, $\chi^{ES}(G\backslash M)$
is given by $\chi(M)/|G|$. A rational number in general, the Euler-Satake characteristic is defined for an
arbitrary closed orbifold by
\begin{equation}
\label{eq:EulerSatakeDef}
    \chi^{ES}(Q)    =   \sum\limits_{\sigma\in\mathcal{T}} (-1)^{\dim\sigma} \frac{1}{|G_\sigma|},
\end{equation}
where $\mathcal{T}$ is a triangulation of the underlying space of $Q$ such that the order of the
isotropy group is constant on the interior of each simplex $\sigma\in\mathcal{T}$ and
$|G_\sigma|$ denotes this order; see \cite[Section~2.2]{FarsiSeaGenOrbEuler}.
Using \cite[Theorem~3.2]{Sea2GBPH}, the (usual) Euler characteristic of the underlying topological
space of $Q$ can be expressed as $\chi_\Z^{ES}(Q)$, and hence for any finitely presented discrete group
$\Gamma$, the Euler characteristic $\chi_\Gamma(Q)$ can be expressed as $\chi_{\Gamma\times\Z}^{ES}(Q)$;
see \cite[Equations~(8) and (9)]{FarsiSeaGenOrbEuler} (see also \cite[Proposition~2-2(2)]{TamanoiCovering}
for the case of a global quotient). Hence, the Euler-Satake characteristic is
in some sense more primitive than the usual Euler characteristic, as all $\Gamma$-Euler characteristics
can be expressed in terms of $\Gamma$-Euler-Satake characteristics, but not conversely. For the present
purpose, however, the Euler-Satake characteristic will be less relevant; see Remark~\ref{rem:EulerSatake2}
below.
\end{remark}

Using the language of integration with respect to the Euler characteristic recalled in Section~\ref{subsec:BackDefin},
Gusein-Zade, Luengo, Melle-Hern\'{a}ndez \cite{GZEMH-HigherOrbEulerCompactGroup} gave the following generalization of
orbifold Euler characteristics to the non-orbifold setting of a $G$-space $A$ where $G$ is a compact Lie group and $A$
is a topological $G$-space that does not necessarily have finite isotropy groups. Using this definition, they generalize the generating series
for higher-order orbifold Euler characteristics of wreath symmetric products given for orbifolds in
\cite{TamanoiMorava,FarsiSeaGenOrbEuler}.
The hypotheses on $A$ in \cite{GZEMH-HigherOrbEulerCompactGroup} are deliberately loose; they assume that
$A$ is ``good enough," e.g., a quasi-projective variety, with finitely many orbit types,
each having a well-defined Euler characteristic. For our purposes,
it will be sufficient to assume that $A$ is a semialgebraic $G$-set, though this is purely for convenience.
Note that \cite{GZEMH-HigherOrbEulerCompactGroup} considered right $G$-actions, while we state the definition
here in terms of left $G$-actions.

\begin{definition}[({Orbifold Euler characteristics for semialgebraic $G$-sets,
\cite[Definitions~2.1 and 2.2]{GZEMH-HigherOrbEulerCompactGroup}})]
\label{def:GZEC}
Let $G$ be a compact Lie group and $A$ a semialgebraic $G$-set.
For each $g\in G$, let $A^{\langle g\rangle}$ denote the points fixed by $g$, let
$C_G(g)$ denote the centralizer of $g$ in $G$, and let $[g]_G$ denote the conjugacy class of $g$ in $G$.
Let $\Ad_G$ denote the adjoint action of $G$ on itself so that $\Ad_G\!\backslash G$ is the space of conjugacy classes
in $G$. The \emph{orbifold Euler characteristic of $(A,G)$} is given by
\begin{equation}
\label{eq:EulerCharGZDef}
    \chi^{(1)}(A,G) =   \int_{\Ad_G\!\backslash G} \chi\big(C_G(g)\backslash A^{\langle g\rangle}\big)
        \, d\chi([g]_G).
\end{equation}
If $\ell\in\N$, the \emph{orbifold Euler characteristic of order $\ell$ of $(A,G)$} is given by
\begin{equation}
\label{eq:HigherEulerCharGZDef}
    \chi^{(\ell)}(A,G) =   \int_{\Ad_G\!\backslash G} \chi^{(\ell-1)}\big(A^{\langle g\rangle},C_G(g)\big)
        \, d\chi([g]_G),
\end{equation}
with $\chi^{(0)}(A,G) = \chi(G\backslash A)$.
\end{definition}

Note that the homeomorphism type of $A^{\langle g\rangle}$ as a $C_G(g)$-space depends only on the conjugacy
class of $g$ so that the integrands are well-defined.
Note further that when $G$ acts on $A$ with finite isotropy groups so that $G\ltimes A$ presents an orbifold $Q$,
$\chi^{(\ell)}(A,G)$ coincides with $\chi_{\Z^\ell}(Q)$ defined in Section~\ref{subsec:BackOrbEuler};
see \cite[Definition~1.1]{GZEMH-HigherOrbEulerCompactGroup} and \cite[p.524]{FarsiSeaGenOrbEuler}. In particular,
when $G$ is finite, $\chi^{(1)}(A,G)$ reduces to $\chi^{orb}(A, G)$ of Equation~\eqref{eq:StringOrbEuler}.

\begin{remark}
\label{rem:OrbECNotation}
In \cite{GZEMH-HigherOrbEulerCompactGroup}, $\chi^{(1)}(A,G)$ is denoted $\chi^{orb}(A, G)$.
We will use the notation $\chi^{(1)}(A,G)$ to avoid confusion with other Euler characteristics.
\end{remark}


\section{$\Gamma$-Euler characteristics for groupoids}
\label{sec:OrbitDefinGpoids}

In this section, we define the $\Gamma$-Euler characteristic of a groupoid
$\mathcal{G}$ where $\Gamma$ is a finitely presented discrete group; see Definition~\ref{def:ECGroupoid}.
We will see in Section~\ref{sec:TranslationProperLie} that these generalize the orbifold Euler characteristics
of \cite{GZEMH-HigherOrbEulerCompactGroup} recalled in Definition~\ref{def:GZEC} as well as the other $\Gamma$-Euler
characteristics and orbifold Euler characteristics recalled above. First, we describe the
groupoids for which these $\Gamma$-Euler characteristics are defined.

For the remainder of this paper, we fix an o-minimal structure on $\R$ that contains the semialgebraic sets, and
$\Gamma$ will always denote a finitely presented group equipped with the discrete topology.


\subsection{Orbit space definable groupoids}
\label{subsec:OrbitDefin}

We begin with the following.

\begin{definition}[(Orbit space definable groupoids)]
\label{def:OrbDefinGpoid}
An \emph{orbit space definable groupoid} (with respect to a fixed o-minimal structure) is a topological groupoid $\mathcal{G}$
and an embedding $\iota_{|\mathcal{G}|}\co|\mathcal{G}|\to\R^n$ of the orbit space $|\mathcal{G}|$ such that
\begin{enumerate}
\item[(i)]      The pair $(|\mathcal{G}|, \iota_{|\mathcal{G}|})$ is an affine definable space,
\item[(ii)]     For each $x\in\mathcal{G}_0$, the isotropy group $\mathcal{G}_x^x$ is a compact Lie group, and
\item[(iii)]    The equivalence relation $x\sim y$ if and only if $\mathcal{G}_x^x$ and $\mathcal{G}_y^y$ are isomorphic
                is a finite partition of $\mathcal{G}_0$ inducing a finite partition of the orbit space $|\mathcal{G}|$
                into definable subsets.
\end{enumerate}
We will often refer to an orbit space definable groupoid as $\mathcal{G}$ and assume without clarification that $\iota_{|\mathcal{G}|}$
denotes the corresponding embedding. Following \cite{PflaumOrbispace}, we refer to the partition in (iii) as the partition by
\emph{weak orbit types}.
\end{definition}

Note that we do not require that the space of objects $\mathcal{G}_0$ and arrows $\mathcal{G}_1$ of an orbit space definable
groupoid to be definable sets or spaces. For instance, suppose $M$ is a topological space that does not admit the structure of an affine
definable space. Then the pair groupoid $M\times M\to M$ has orbit space a point and trivial isotropy groups so that it is an orbit
space definable groupoid. We will on occasion consider groupoids whose spaces of objects and arrows have definable structures compatible
with that of the quotient; however, an orbit space definable groupoid structure will be sufficient to define the $\Gamma$-Euler characteristics
we introduce in Definition~\ref{def:ECGroupoid} below. It is frequently the case that the o-minimal structure of semialgebraic sets will be sufficient.
However, larger o-minimal structures allow the consideration of sets that are described
more analytically, e.g., using exponential functions, see \cite[Section~2.5]{vandenDriesMillerGeomCat},
and there is no reason to preclude this flexibility.

If $\mathcal{G}$ and $\mathcal{H}$ are orbit space definable groupoids, then the product groupoid $\mathcal{G}\times\mathcal{H}$ is easily seen
to admit the structure of an orbit space definable groupoid. Specifically, $(x,y),(x^\prime,y^\prime)\in\mathcal{G}_0\times\mathcal{H}_0$ are in
the same orbit in $\mathcal{G}\times\mathcal{H}$
if and only if there is an arrow $(g,h)\in\mathcal{G}_1\times\mathcal{G}_1$ such that $s(g,h) = (x,y)$ and $t(g,h) = (x^\prime,y^\prime)$, which
is equivalent to $x$ and $y$ being in the same orbit in $\mathcal{G}$ and $x^\prime$ and $y^\prime$ being in the same orbit in $\mathcal{H}$.
Therefore, $|\mathcal{G}\times\mathcal{H}| = |\mathcal{G}|\times|\mathcal{H}|$ so that $\iota_{|\mathcal{G}|}\times\iota_{|\mathcal{H}|}$ is an embedding
of $|\mathcal{G}\times\mathcal{H}|$ with respect to which $|\mathcal{G}\times\mathcal{H}|$ is an affine definable space. Similarly, the isotropy
group of $(x,y)\in\mathcal{G}_0\times\mathcal{H}_0$ is $\mathcal{G}_x^x\times\mathcal{H}_y^y$ so that the partition of
$|\mathcal{G}\times\mathcal{H}|$ into $\sim$-classes as in condition (iii) of the definition is given by products of the corresponding equivalence
classes in $|\mathcal{G}|$ and $|\mathcal{H}|$.

If $\mathcal{G}$ is an orbit space definable groupoid and $X\subset|\mathcal{G}|$ is a definable subset, then letting
$\pi\co\mathcal{G}_0\to|\mathcal{G}|$ denote the orbit map, the restriction $\mathcal{G}_{|\pi^{-1}(X)}$ of $\mathcal{G}$ to the preimage
of $X$ is as well orbit space definable. Clearly, $|\mathcal{G}_{|\pi^{-1}(X)}| = X$ is an affine definable space via the restriction
of $\iota_{|\mathcal{G}|}$ to $X$, the isotropy groups are the same with respect to both $\mathcal{G}$ and $\mathcal{G}_{|\pi^{-1}(X)}$,
and the $\sim$-classes in $|\mathcal{G}_{|\pi^{-1}(X)}|$ are the intersections with $X$ of those in $|\mathcal{G}|$.

We summarize these observations in the following.

\begin{lemma}
\label{lem:OrbDefinGpoidConstructions}
Let $\mathcal{G}$ and $\mathcal{H}$ be orbit space definable groupoids with embeddings $\iota_{|\mathcal{G}|}$ and $\iota_{|\mathcal{H}|}$.
\begin{enumerate}
\item[(i)]      The groupoid $\mathcal{G}\times\mathcal{H}$ admits the structure of an orbit space definable groupoid with embedding
                $\iota_{|\mathcal{G}\times\mathcal{H}|} = \iota_{|\mathcal{G}|}\times\iota_{|\mathcal{H}|}$.
\item[(ii)]     If $X\subset|\mathcal{G}|$, then $\mathcal{G}_{|\pi^{-1}(X)}$ admits the structure of an orbit space definable groupoid
                with embedding given by the restriction of $\iota_{|\mathcal{G}|}$ to $X$.
\end{enumerate}
\end{lemma}

\begin{remark}
\label{rem:DefinableMorita}
It is natural to consider \emph{definable homomorphisms} of orbit space definable groupoids to be homomorphisms of the underlying topological
groupoids such that the induced map on orbit spaces is a morphism of affine definable spaces. In particular, the notion of a \emph{definable
essential equivalence}, a definable homomorphism that is as well an essential equivalence of the underlying topological groupoids, yields a
notion of \emph{definable Morita equivalence} between orbit space definable groupoids. We will not need these notions, as our interest
is in invariants of the underlying topological groupoids. Hence, we view an orbit space definable structure as a device used to define
and compute Euler characteristics, which will not depend on the choice of such structure by Theorem~\ref{thrm:ECHomeoInvar} and
Remark~\ref{rem:intEChomeoInvar}; see Section~\ref{subsec:GpoidECDef}.
\end{remark}

We are primarily interested in the special case of cocompact proper Lie groupoids, which we now show always admit the structure of an orbit space
definable groupoid. In the proof below, we define this structure using a triangulation of the orbit space whose existence was demonstrated in
\cite{PPTOrbitSpace}. Note that the class of orbit space definable groupoids is much larger than that of cocompact proper Lie groupoids.

\begin{proposition}[(Cocompact proper Lie groupoids are orbit space definable)]
\label{prop:LieGpdOrbitDefin}
Let $\mathcal{G}$ be a cocompact proper Lie groupoid. Then there is an embedding $\iota_{|\mathcal{G}|}\co|\mathcal{G}|\to\R^n$ with respect to which
$\mathcal{G}$ is orbit space definable.
\end{proposition}
\begin{proof}
Let $x\in\mathcal{G}_0$, and then it is well known that $\mathcal{G}_x^x$ is a Lie group;
see \cite[Theorem~5.4]{MoerdijkMrcun}. As $\mathcal{G}_x^x = (s,t)^{-1}(x)$, it is compact by
the properness of $\mathcal{G}$. By \cite[Corollary~7.2]{PPTOrbitSpace}, $|\mathcal{G}|$ admits a
triangulation that is compatible with the stratification by normal orbit types; see
\cite[Definition~5.6(ii)]{PPTOrbitSpace} for the definition. That is, the closure
of each stratum corresponds to a simplicial subcomplex, and the local compactness of $|\mathcal{G}|$ implies that the triangulating simplicial
complex is locally finite \cite[Proposition~7.3]{PflaumOrbispace}, \cite[Lemma~2.6]{MunkresAlg}.
As $\mathcal{G}$ is cocompact so that $|\mathcal{G}|$ is actually compact, the simplicial complex is finite,
and therefore can be embedded into the real finite-dimensional vector space $\R^n$ spanned by its vertices in
such a way that the image of each simplex is a semialgebraic set \cite[page~956]{DelfsKnebuschIntro}. The composition
of the homeomorphism of $|\mathcal{G}|$ with this finite simplicial complex and the semialgebraic embedding
of the complex yields an embedding $\iota_{|\mathcal{G}|}\co|\mathcal{G}|\to\R^n$ whose image is semialgebraic
and hence definable, making $|\mathcal{G}|$ into an affine definable space.

By \cite[Theorem~5.7]{PPTOrbitSpace}, $|\mathcal{G}|$ is stratified by weak orbit types, whose definition \cite[Definition~5.6(i)]{PPTOrbitSpace}
coincides with the equivalence relation in Definition~\ref{def:OrbDefinGpoid}~(iii). Each normal orbit type is open and closed
in its weak orbit type so that the connected components of weak orbit types are exactly the connected components of
normal orbit types; hence, both stratifications are compatible with the triangulation above. Moreover, these connected
components form a decomposition of $|\mathcal{G}|$ in the sense of \cite[Definition~2.1]{PPTOrbitSpace}, in particular
implying that this partition is locally finite. As $|\mathcal{G}|$ is compact, it follows that the decomposition is in fact
finite so that Definition~\ref{def:OrbDefinGpoid}~(iii) is satisfied.
\end{proof}

An embedding $\iota_{|\mathcal{G}|}$ making $\mathcal{G}$ into an orbit space definable groupoid as in Proposition~\ref{prop:LieGpdOrbitDefin}
is usually far from unique, so the orbit space $|\mathcal{G}|$ may admit many inequivalent structures as an affine definable space.
As explained in Remark~\ref{rem:DefinableMorita}, our interest is in invariants that will not depend on this choice.
The hypothesis of cocompactness in Proposition~\ref{prop:LieGpdOrbitDefin} is sufficient but not necessary, which can
be seen by considering $\mathcal{G}_{|\pi^{-1}(U)}$ where $\mathcal{G}$ is cocompact, $U$ is an open definable subset
of $|\mathcal{G}|$, and $\pi\co\mathcal{G}_0\to|\mathcal{G}|$ again denotes the orbit map.

It is natural to consider topological groupoids $\mathcal{G}$ such that the spaces of objects and arrows are affine definable spaces
and the structure maps are morphisms of definable spaces. Using the results of \cite[Section~10.2]{vandenDriesBook} on the existence of
definably proper quotients of definably proper equivalence relations recalled above, we now demonstrate that such a groupoid admits a
definable orbit space assuming it is proper and the source map is proper.
Note that as the inverse map $i\co\mathcal{G}_1\to\mathcal{G}_1$ is its own inverse function and hence
a homeomorphism, if $s$ is proper then $t = s\circ i$ is proper as well.

\begin{proposition}
\label{prop:DefinGpdOrbitDefin}
Let $\mathcal{G}$ be a proper topological groupoid such that the source map $s$ is proper.
Let $\iota_1\co\mathcal{G}_1\to\R^n$ be a topological embedding such that
$(\mathcal{G}_1, \iota_1)$ is an affine definable space, assume $u(\mathcal{G}_0)\subset\mathcal{G}_1$ is a definable
subset, and let $\iota_0 = \iota_1\circ u\co\mathcal{G}_0\to\R^n$ denote the induced topological embedding. Assume
the structure maps $s,t,u,i,m$ are morphisms of affine
definable spaces. Then there is an embedding $\iota_{|\mathcal{G}|}$ of $|\mathcal{G}|$ into
Euclidean space with respect to which $\mathcal{G}$ satisfies Definition~\ref{def:OrbDefinGpoid}(i) and (ii). In addition,
if $\mathcal{G}$ satisfies Definition~\ref{def:OrbDefinGpoid}(iii), then $\mathcal{G}$ is orbit space definable.
\end{proposition}
\begin{proof}
Let $\widehat{\mathcal{G}_i}$ denote $\iota_i(\mathcal{G}_i)$ for $i = 0, 1$, and let
$\hat{s} = \iota_0\circ s\circ\iota_1^{-1}\co\widehat{\mathcal{G}_1}\to\widehat{\mathcal{G}_0}$,
$\hat{t} = \iota_0\circ t\circ\iota_1^{-1}\co\widehat{\mathcal{G}_1}\to\widehat{\mathcal{G}_0}$, etc.,
so that $\widehat{\mathcal{G}_0}$ and $\widehat{\mathcal{G}_1}$ along with the structure maps
$\hat{s}$, $\hat{t}$, $\hat{u}$, $\hat{i}$, and $\hat{m}$ is a proper topological groupoid embedded in $\R^n$ as definable sets
such that the structure maps are definable functions.
Let $E = \{(x,y): x\in\widehat{\mathcal{G}_0}, y\in\widehat{\mathcal{G}}x\}$ denote the equivalence relation corresponding to the
partition of $\widehat{\mathcal{G}_0}$ into orbits. Then $E$ is the image of definable function
$(\hat{s},\hat{t})\co\widehat{\mathcal{G}_1}\to\widehat{\mathcal{G}_0}\times\widehat{\mathcal{G}_0}$ and hence definable.
Let $\pr_1\co E\to\widehat{\mathcal{G}_0}$ denote the projection onto the first factor and let $K\subseteq\widehat{\mathcal{G}_0}$
be definable and compact. Then $\pr_1^{-1}(K)$ is given by $(\hat{s},\hat{t})\big(\hat{s}^{-1}(K)\big)$. As $\hat{s}$ is proper so that
$\hat{s}^{-1}(K)$ is definable and compact, the image $(\hat{s},\hat{t})\big(\hat{s}^{-1}(K)\big) = \pr_1^{-1}(K)$ of $\hat{s}^{-1}(K)$
under $(\hat{s},\hat{t})$ is definable and compact so that $\pr_1$ is definably proper. Then by \cite[Chapter 10, Theorem~(2.15)]{vandenDriesBook},
there is a definably proper quotient $\pi\co\widehat{\mathcal{G}_0}\to B\subset\R^k$, and the embedding
$\iota_0\co\mathcal{G}_0\to\widehat{\mathcal{G}_0}\subset\R^n$ induces a homeomorphism
$\iota_{|\mathcal{G}|}=E\backslash\iota_0\co|\mathcal{G}|\to B\subset\R^k$;
see Remark~\ref{rem:QuotUnique}. Hence $(|\mathcal{G}|,\iota_{|\mathcal{G}|})$ is an affine definable space.
That $(\hat{s},\hat{t})$ is proper and definable, hence definably proper, implies that $(s,t)^{-1}(x) = \mathcal{G}_x^x$
is compact for each $x\in\mathcal{G}_0$.
\end{proof}

An important case of Proposition~\ref{prop:DefinGpdOrbitDefin} is that of the translation groupoid
$G\ltimes A$ where $G$ is a compact Lie group identified with a compact linear algebraic group so that $G\subset\R^{m^2}$
and $A\subset\R^n$ is a semialgebraic $G$-set; see
Section~\ref{subsec:BackDefin}. In this case, using the identity functions as embeddings of
$A$ and $G\times A\subset\R^{m^2+n}$, the structure maps are obviously semialgebraic, and
$G\ltimes A$ is proper. For any compact $K\subseteq A$, the set $s_{G\ltimes A}^{-1}(K) = G\times K$
is compact so that $s_{G\ltimes A}$ is proper. As well, a semialgebraic $G$-set $A$ has finitely many orbit types,
each a semialgebraic set, see \cite[Theorem~2.6 and p.~635]{ChoiParkSuhSemialgSlice}, yielding the
following.

\begin{corollary}[(Semialgebraic $G$-sets are orbit space definable)]
\label{cor:DefinGpdOrbitDefin}
Let $G$ be a compact Lie group and $A$ a semialgebraic $G$-set. Then there is an embedding of $G\ltimes A$
into $\R^n$ with respect to which $G\ltimes A$ is an orbit space definable groupoid.
\end{corollary}


\subsection{$\Gamma$-Euler characteristics for orbit space definable groupoids}
\label{subsec:GpoidECDef}

In this section, we define the $\Gamma$-Euler characteristics for orbit space definable groupoids.
Recall that we fix an o-minimal structure on $\R$ that contains the semialgebraic sets with respect to which
definable sets and spaces are defined.
Let $G$ be a compact Lie group and let $\Gamma$ be a finitely presented discrete group.
We will frequently make use of the following.

\begin{notation}
\label{not:FinitePres}
Assume $\Gamma$ has finite presentation $\Gamma = \langle\gamma_1,\ldots,\gamma_\ell\mid R_1,\ldots,R_k\rangle$.
Then $\HOM(\Gamma, G)$ can be identified with the subset of $G^\ell$ satisfying the relations $R_i$
via the map $\HOM(\Gamma, G)\ni\phi\mapsto(\phi(\gamma_1),\ldots,\phi(\gamma_\ell))\in G^\ell$.
With respect to this identification, the action of $G$ on $\HOM(\Gamma, G)$ by conjugation
corresponds to the action of $G$ on $G^\ell$ by simultaneous conjugation.
\end{notation}

If $(G, \lambda)$ is an affine topological group, then it is clear from the above description that
$\HOM(\Gamma,G)$ is a definable subset of the affine topological space $(G^\ell, \lambda^\ell)$.
The action of $G$ on $G^\ell$ by simultaneous conjugation makes $G^\ell$ into an affine definable $G$-space,
and hence $\HOM(\Gamma, G)$ is an affine definable $G$-space with respect to the $G$-action by conjugation
and the restriction of the embedding $\lambda^\ell$. The quotient space $G\backslash\!\HOM(\Gamma, G)$
admits an embedding with respect to which it is homeomorphic to a definably proper quotient by
\cite[Chapter 10, Corollary~(2.18)]{vandenDriesBook}; see Remark~\ref{rem:QuotUnique}.

With the above notation established, we have the following, which is the primary definition introduced in this section.

\begin{definition}[($\Gamma$-Euler characteristics for orbit space definable groupoids)]
\label{def:ECGroupoid}
Let $(\mathcal{G},\iota_{|\mathcal{G}|})$ be an orbit space definable groupoid, and let $\Gamma$ be a finitely presented discrete group.
Define the \emph{$\Z$-Euler characteristic of $\mathcal{G}$} to be
\begin{equation}
\label{eq:ECGroupoid}
    \chi_\Z(\mathcal{G}) =   \int_{|\mathcal{G}|}
        \chi\big(\Ad_{\mathcal{G}_{x}^x}\!\backslash\mathcal{G}_{x}^x  \big) \, d\chi(\mathcal{G}x),
\end{equation}
where we recall that $\mathcal{G}x$ denotes the $\mathcal{G}$-orbit of $x$.
More generally, the \emph{$\Gamma$-Euler characteristic of $\mathcal{G}$} is given by
\begin{equation}
\label{eq:GammaECGroupoid}
    \chi_\Gamma(\mathcal{G}) =   \int_{|\mathcal{G}|}
        \chi\big(\mathcal{G}_{x}^x\backslash\!\HOM(\Gamma,\mathcal{G}_{x}^x) \big) \, d\chi(\mathcal{G}x),
\end{equation}
where the action of $\mathcal{G}_x^x$ on $\HOM(\Gamma,\mathcal{G}_{x}^x)$ is given by pointwise
conjugation.
Note that Equation~\eqref{eq:GammaECGroupoid} reduces to Equation~\eqref{eq:ECGroupoid} in the case
$\Gamma = \Z$ via the identification of $\HOM(\Z,\mathcal{G}_x^x)$ with $\mathcal{G}_x^x$ by choosing
a generator for $\Z$.
\end{definition}

By the following remark, the above Euler characteristics do not depend on the choice of o-minimal structure or embeddings.

\begin{remark}
\label{rem:ECGroupoidWellDef}
Note that as $\mathcal{G}$ is orbit space definable, $|\mathcal{G}|$ is an affine definable space, and each isotropy group
$\mathcal{G}_x^x$ admits the structure of an affine definable topological group.
If $x,y\in\mathcal{G}_0$ are in the same orbit,
i.e., there is an $h\in\mathcal{G}_1$ such that $s(h) = x$ and $t(h) = y$,
then conjugation by $h$ induces an isomorphism of $\mathcal{G}_x^x$ onto $\mathcal{G}_y^y$.
Hence $\Ad_{\mathcal{G}_{x}^x}\!\backslash\mathcal{G}_{x}^x$
and $\mathcal{G}_{x}^x\backslash\!\HOM(\Gamma,\mathcal{G}_{x}^x)$ admit embeddings identifying them with definable proper quotients;
see Remark~\ref{rem:QuotUnique}. As well,
$\chi\big(\Ad_{\mathcal{G}_{x}^x}\!\backslash\mathcal{G}_{x}^x \big)$ and
$\chi\big(\mathcal{G}_{x}^x\backslash\!\HOM(\Gamma,\mathcal{G}_{x}^x) \big)$ are defined
and by Theorem~\ref{thrm:ECHomeoInvar} do not depend on the choice of $x$ in an orbit.
Moreover, as the weak orbit type partition of $|\mathcal{G}|$ into orbits of points with isomorphic isotropy groups
is a finite partition into definable sets, the two integrands in Equations~\eqref{eq:ECGroupoid}
and \eqref{eq:GammaECGroupoid} of Definition~\ref{def:ECGroupoid} are constructible functions on
$|\mathcal{G}|$. It follows that $\chi_\Z(\mathcal{G})$ and $\chi_\Gamma(\mathcal{G})$
are defined. By Remark~\ref{rem:intEChomeoInvar}, $\chi_\Z(\mathcal{G})$ and $\chi_\Gamma(\mathcal{G})$ do not depend on the
o-minimal structure nor the choice of embedding $\iota_{|\mathcal{G}|}$ of $|\mathcal{G}|$ and hence are invariants
of the topological groupoid $\mathcal{G}$.
\end{remark}

We will see in Section~\ref{subsec:GpoidECTranslation} that when $\mathcal{G} = G\ltimes X$ is a translation groupoid,
the definition of $\chi_\Z(\mathcal{G})$ in Equation~\eqref{eq:ECGroupoid} coincides with the orbifold
Euler characteristic $\chi^{(1)}(X,G)$ of \cite{GZEMH-HigherOrbEulerCompactGroup} recalled in
Equation~\eqref{eq:EulerCharGZDef}. If in addition $\Gamma = \Z^\ell$, the definition of
$\chi_{\Z^\ell}(\mathcal{G})$ in Equation~\eqref{eq:GammaECGroupoid} coincides
with the higher-order orbifold Euler characteristic $\chi^{(\ell)}(X,G)$ of \cite{GZEMH-HigherOrbEulerCompactGroup}
recalled in Equation~\eqref{eq:HigherEulerCharGZDef}; see Theorem~\ref{thrm:TranslationECequal} below.
In particular, if $G$ is finite so that $G\ltimes X$ presents a global quotient orbifold,
the definition of $\chi^{(1)}(X,G)$ reduces to the orbifold Euler characteristic of \cite{DixonHarveyEA}
recalled in Equation~\eqref{eq:StringOrbEuler}, which is therefore equal to $\chi_\Z(G\ltimes X)$. More generally,
if $\mathcal{G}$ is an orbifold groupoid presenting the orbifold
$Q$, we will see that $\chi_\Gamma(\mathcal{G})$ coincides with the $\Gamma$-Euler characteristic
$\chi_\Gamma(Q)$ of \cite{FarsiSeaGenOrbEuler}
recalled in Section~\ref{subsec:BackOrbEuler}; this is a consequence of Theorem~\ref{thrm:GamECInertia}
below and the fact that $\chi_\Gamma(Q)$ is the Euler characteristic of the $\Gamma$-inertia space of $Q$,
in the orbifold context called the orbifold of $\Gamma$-sectors $\widetilde{Q}_\Gamma$.
See Section~\ref{subsec:Examples}
for example computations of $\chi_\Z(\mathcal{G})$ and $\chi_\Gamma(\mathcal{G})$.

\begin{remark}
\label{rem:FinitePres}
In the case of an orbifold groupoid, see Section~\ref{subsec:BackOrbEuler}, the isotropy groups
are finite, so it is sufficient to assume $\Gamma$ is finitely generated. In the present context,
we require that $\Gamma$ is finitely presented to ensure that $\HOM(\Gamma,G)$ can be described
as in Notation~\ref{not:FinitePres} by a finite collection of relations and hence given the structure of an affine
definable space.
\end{remark}

Note that if $\Gamma = \{0\}$ is the trivial group, then
$\chi_{\{0\}} = \int_{|\mathcal{G}|} \, d\chi(\mathcal{G}x) = \chi(|\mathcal{G}|)$ is the usual Euler characteristic
of the orbit space.

\begin{remark}
\label{rem:EulerSatake2}
For the Euler-Satake characteristic $\chi^{ES}$ of an orbifold, see Remark~\ref{rem:EulerSatake},
a natural generalization to the case of orbit space definable groupoids considered in this paper
would be to replace $|G_\sigma|$ in Equation~\eqref{eq:EulerSatakeDef}
with the number of connected components of the isotropy group.
Because the Euler-Satake characteristic is finitely additive, integration with respect to the Euler-Satake
characteristic can be defined in the same way as integration with respect to the Euler characteristic;
we could therefore define $\chi_\Gamma^{ES}(\mathcal{G})$ by replacing $\chi$ with $\chi^{ES}$ in
Equations~\eqref{eq:ECGroupoid} and \eqref{eq:GammaECGroupoid}. However, we find the above definition
difficult to motivate, and
the fundamental relationship between the Euler-Satake characteristic and the usual Euler characteristic
breaks down in this more general setting. That is, the identity
$\chi_{\Gamma\times\Z}^{ES}(\mathcal{G}) = \chi_\Gamma(\mathcal{G})$ no longer holds for non-orbifold
groupoids $\mathcal{G}$, as is illustrated by Example~\ref{ex:EulerSatake}. Hence, while the results of
this paper could also be formulated with $\chi$ replaced by $\chi^{ES}$, we restrict our attention to the former.
See also \cite{LeinsterEulerCharCategory}, where the Euler-Satake characteristic of an orbifold is identified as
a case of the Euler characteristic of a finite category.
\end{remark}

\begin{example}
\label{ex:EulerSatake}
Let $X$ be a point with the trivial action of the circle $G = \Sp^1$, and let $\mathcal{G} = G\ltimes X$.
Then $\chi_{\{0\}}(\mathcal{G}) = \chi(|\mathcal{G}|) = 1$ and, using the definition suggested in
Remark~\ref{rem:EulerSatake}, $\chi_{\{0\}}^{ES}(\mathcal{G}) = \chi_{ES}(|\mathcal{G}|) = 1$. As
$\mathcal{G}_x^x$ for the single point $x\in X$ is given by $\Sp^1$ and the conjugation action is trivial,
$\chi_\Z^{ES}(\mathcal{G}) = \chi_\Z(\mathcal{G}) = 0$. Hence, the identity
$\chi_{\Gamma\times\Z}^{ES}(\mathcal{G}) = \chi_\Gamma(\mathcal{G})$ is not satisfied when
$\Gamma = \{ 0\}$ is the trivial group. Similarly, as the isotropy group of any
element of $\HOM(\Gamma,\mathcal{G})$ is simply $G$ for any finitely presented discrete group $\Gamma$,
$\chi_\Gamma^{ES}(\mathcal{G}) = \chi_\Gamma(\mathcal{G})$.
\end{example}


\section{$\chi_\Gamma(\mathcal{G})$ as the Euler characteristic of a topological space}
\label{sec:ECSpace}

In this section, we identify the $\Gamma$-Euler characteristic of an orbit space definable groupoid
$\mathcal{G}$ with the usual Euler characteristic of a topological space, denoted $\Lambda\chi_\Gamma(\mathcal{G})$,
that depends only on the Morita equivalence class of $\mathcal{G}$ as a topological groupoid, and hence prove that
$\chi_\Gamma(\mathcal{G})$ is Morita invariant. To begin, we define the $\Gamma$-inertia groupoid of an arbitrary
topological groupoid, whose orbit space will be the topological space in question.


\subsection{The $\Gamma$-inertia groupoid of a topological groupoid}
\label{subsec:GamInertia}

Let $\mathcal{G}$ be a topological groupoid and $\Gamma$ a finitely presented discrete group.
As a groupoid, the object space of $\Gamma$ is a single point.
Hence, a groupoid homomorphism $\phi\co\Gamma\to\mathcal{G}$ is given by a choice of $x\in\mathcal{G}_0$,
the value of $\phi_0$ at the unique object of the groupoid $\Gamma$, and a group homomorphism
$\phi_1\co\Gamma\to\mathcal{G}_x^x$. In particular, for each $\gamma\in\Gamma$,
$s\circ\phi_1(\gamma) = t\circ\phi_1(\gamma) = x$.

\begin{notation}
\label{not:HomGammaGpoid}
If $\phi\in\HOM(\Gamma,\mathcal{G})$, we will for simplicity use $\phi_0$ to denote
the map of $\phi$ on objects as well as its single value in $\mathcal{G}_0$ at the unique object
of $\Gamma$.
\end{notation}

The space $\HOM(\Gamma,\mathcal{G})$ inherits the compact-open topology from $\Gamma$ and $\mathcal{G}$
with subbase given by the set of $\phi\in\HOM(\Gamma,\mathcal{G})$ such that $\phi_0\in U$ and
$\phi_1(K)\subseteq V$ where $K$ is a finite subset of $\Gamma$, $U\subseteq\mathcal{G}_0$ and $V\subseteq\mathcal{G}_1$
are open, and $s(V)\subseteq U$. For a fixed
$x\in\mathcal{G}_0$, the relative topology on the set of $\phi\in\HOM(\Gamma,\mathcal{G})$ such that
$\phi_0 = x$ coincides with the topology on $\HOM(\Gamma,\mathcal{G}_x^x)$ as a subset of
$(\mathcal{G}_x^x)^\ell$ as described in Notation~\ref{not:FinitePres}.

We now state the following. See \cite{FarsiPflaumSeaton1,FarsiPflaumSeaton2} for the case $\Gamma = \Z$.

\begin{definition}[(The $\Gamma$-inertia groupoid)]
\label{def:GamInertia}
Let $\mathcal{G}$ be a topological groupoid and $\Gamma$ a finitely presented discrete group.
The \emph{$\Gamma$-loop space of $\mathcal{G}$} is the space $\HOM(\Gamma,\mathcal{G})$ with
the compact-open topology described above.
The \emph{$\Gamma$-inertia groupoid of $\mathcal{G}$}, denoted $\Lambda_\Gamma\mathcal{G}$,
is the translation groupoid $\mathcal{G}\ltimes\HOM(\Gamma,\mathcal{G})$ for the action of $\mathcal{G}$
on $\HOM(\Gamma,\mathcal{G})$ by conjugation. Hence, the space of objects of $\Lambda_\Gamma\mathcal{G}$
is the $\Gamma$-loop space $(\Lambda_\Gamma\mathcal{G})_0 = \HOM(\Gamma,\mathcal{G})$,
the anchor map $\alpha_{\Lambda_\Gamma\mathcal{G}}\co\HOM(\Gamma,\mathcal{G})\to\mathcal{G}_0$ is given by
$\alpha_{\Lambda_\Gamma\mathcal{G}}\co\phi\mapsto \phi_0$,
the source of each element of the image of $\phi_1$, and the action of $g\in\mathcal{G}_1$
on $\phi\in\HOM(\Gamma,\mathcal{G})$ such that $s(g) = \phi_0$ is given by conjugation
in $\mathcal{G}_1$, i.e., $(g\ast\phi)_1 = g\phi_1 g^{-1}$. When there is no ambiguity, we will
use the shorthand $\pi_\Lambda$ to denote the orbit map
$\pi_{\Lambda_\Gamma\mathcal{G}}\co\Lambda_\Gamma\mathcal{G}\to|\Lambda_\Gamma\mathcal{G}|$
of the groupoid $\Lambda_\Gamma\mathcal{G}$. Note that $g\ast\phi$ is a homomorphism $\Gamma\to\mathcal{G}$
with $(g\ast\phi)_0 = t(g)$, i.e., a group homomorphism from $\Gamma$ to $\mathcal{G}_{t(g)}^{t(g)}$.
\end{definition}

When $\Gamma = \Z$, the $\Z$-inertia groupoid and space are referred to simply as the \emph{inertia
groupoid} and \emph{inertia space}, respectively, and the $\Z$-loop space is the \emph{loop space}.
Note that the loop space $\HOM(\Z,\mathcal{G})$ can be identified with the subspace of $\mathcal{G}_1$
consisting of $g$ such that $s(g) = t(g)$ by identifying a homomorphism $\phi$ with the image of a
generator of $\Z$. Then $\HOM(\Z,\mathcal{G})$ has the structure of a subgroupoid of
$\mathcal{G}$, and some authors refer to $\HOM(\Z,\mathcal{G})$ as the \emph{isotropy subgroupoid of
$\mathcal{G}$}. However, for our purposes, $\HOM(\Z,\mathcal{G})$ plays the role of a $\mathcal{G}$-space,
with elements treated as objects instead of arrows, and so the groupoid structure is not relevant.
Hence, we will not use the term isotropy subgroupoid.

\begin{remark}
\label{rem:InertiaTranslation}
If $\mathcal{G} = G\ltimes X$ is a translation groupoid where $G$ is a topological group,
the $\Gamma$-loop space $\HOM(\Gamma,G\ltimes X)$ is the topological space
$\{ (\phi,x) \in \HOM(\Gamma,G)\times X : \phi(\gamma)\in G_x\;\forall\gamma\in\Gamma \}$
where $G_x$ denotes the isotropy group of $x\in X$. The $(G\ltimes X)$-action on $\HOM(\Gamma,G\ltimes X)$,
whose anchor is the source map, coincides with restriction of the diagonal $G$-action
on $\HOM(\Gamma,G)\times X$ to the $G$-invariant subset $\HOM(\Gamma,G\ltimes X)$, where the action on
$\HOM(\Gamma,G)$ is by pointwise conjugation.
Choosing a finite presentation $\Gamma = \langle\gamma_1,\ldots,\gamma_\ell\mid R_1,\ldots,R_k\rangle$,
we identify this space via the map $\phi\mapsto (\phi(\gamma_1),\ldots,\phi(\gamma_r))$
with the set of points $(g_1,\ldots,g_\ell,x)\in G^\ell\times X$ such that the $g_i$ satisfy the
relations $R_j$ and $g_i(x) = x$ for each $i$; see Notation~\ref{not:FinitePres}. The $(G\ltimes X)$-action
coincides with the diagonal $G$-action on $G^\ell\times X$, acting by simultaneous conjugation on $G^\ell$-factor.
\end{remark}

We have the following, which was proven for orbifold groupoids in \cite[Lemma~2.5]{FarsiSeaVectorFld}.

\begin{lemma}
\label{lem:InertHomomorph}
If $\mathcal{G}$ and $\mathcal{H}$ are topological groupoids and $\Gamma$ is a finitely presented discrete group,
then a homomorphism $\Phi\co\mathcal{G}\to\mathcal{H}$ induces a homomorphism
$\Lambda_\Gamma\Phi\co\Lambda_\Gamma\mathcal{G}\to\Lambda_\Gamma\mathcal{H}$. If $\Phi$ is
an essential equivalence, then $\Lambda_\Gamma\Phi$ is as well.
\end{lemma}
\begin{proof}
Define the map on objects $(\Lambda_\Gamma\Phi)_0\co\HOM(\Gamma,\mathcal{G})\to\HOM(\Gamma,\mathcal{H})$
by $\phi\mapsto\Phi_1\circ\phi$, and then for each
$\phi\in\HOM(\Gamma,\mathcal{G})$ and $g\in\mathcal{G}_1$ such that $s_\mathcal{G}(g) = \phi_0$, we have
$\Phi_1(g)\ast\big((\Lambda_\Gamma\Phi)_0(\phi)\big) = (\Lambda_\Gamma\Phi)_0(g\ast\phi)$. It follows
that $\Phi_1$ induces a homomorphism of the translation groupoids
$\Lambda_\Gamma\Phi\co\mathcal{G}\ltimes\HOM(\Gamma,\mathcal{G})=\Lambda_\Gamma\mathcal{G}\to
\Lambda_\Gamma\mathcal{H}=\mathcal{H}\ltimes\HOM(\Gamma,\mathcal{H})$.

Now suppose $\Phi$ is an essential equivalence. Then $\mathcal{G}_1$ is the fibred product
$\mathcal{H}_1 \ftimes{(s_{\mathcal{H}},t_{\mathcal{H}})}{\Phi_0\times\Phi_0} (\mathcal{G}_0\times\mathcal{G}_0)$,
implying that for each $x\in\mathcal{G}_0$, $\Phi_1$ restricts to an isomorphism
$\mathcal{G}_x^x\to\mathcal{H}_{\Phi_0(x)}^{\Phi_0(x)}$, and hence induces an isomorphism
$\Phi_x\co\HOM(\Gamma,\mathcal{G}_x^x)\to\HOM\big(\Gamma,\mathcal{H}_{\Phi_0(x)}^{\Phi_0(x)}\big)$. The map
\[
    t_{\Lambda_\Gamma\mathcal{H}}\circ\pr_1\co
        \big(\mathcal{H}_1\ftimes{s_{\mathcal{H}}}{\alpha_{\Lambda_\Gamma\mathcal{H}} } \HOM(\Gamma,\mathcal{H})\big)
            \ftimes{s_{\Lambda_\Gamma\mathcal{H}}}{(\Lambda_\Gamma\Phi)_0} \HOM(\Gamma,\mathcal{G})
            \to\HOM(\Gamma,\mathcal{H}),
\]
where $\big(\mathcal{H}_1\ftimes{s_{\mathcal{H}}}{\alpha_{\Lambda_\Gamma\mathcal{H}} }\HOM(\Gamma,\mathcal{H})\big)
\ftimes{s_{\Lambda_\Gamma\mathcal{H}}}{(\Lambda_\Gamma\Phi)_0} \HOM(\Gamma,\mathcal{G})$ is the set of
$\big( (h,\Phi_1\circ\phi),\phi\big)$ with
$\phi\in\HOM(\Gamma,\mathcal{H})$ and $h\in\mathcal{H}_1$ such that $s_{\mathcal{H}}(h) = \Phi_0(\phi_0)$,
is given by $\big((h,\Phi_1\circ\phi),\phi\big)\mapsto h(\Phi_1\circ\phi) h^{-1}$.
As the map $t_\mathcal{H}\circ\pr_1\co\mathcal{H}_1\ftimes{s_{\mathcal{H}}}{\Phi_0}\mathcal{G}_0\to\mathcal{H}_0$
is surjective, each $y\in\mathcal{H}_0$ is in the orbit of some $\Phi_0(x)$, and hence
each $\psi\in\HOM(\Gamma,\mathcal{H})$ such that $\psi_0 = y$ is conjugate via $\mathcal{H}_1$
to an element of $\HOM(\Gamma,\Phi_1(\mathcal{G}_x^x))$. It follows that
$t_{\Lambda_\Gamma\mathcal{H}}\circ\pr_1$ is surjective. To see that $t_{\Lambda_\Gamma\mathcal{H}}\circ\pr_1$
has local sections, let
$\big( (h,\Phi_1\circ\phi),\phi\big)\in\big(\mathcal{H}_1\ftimes{s_{\mathcal{H}}}{\alpha_{\Lambda_\Gamma\mathcal{H}} }\HOM(\Gamma,\mathcal{H})\big)
\ftimes{s_{\Lambda_\Gamma\mathcal{H}}}{(\Lambda_\Gamma\Phi)_0} \HOM(\Gamma,\mathcal{G})$ with
$t_{\Lambda_\Gamma\mathcal{H}}\big( (h,\Phi_1\circ\phi) \big) = h(\Phi_1\circ\phi)h^{-1}=\psi\in\HOM(\Gamma,\mathcal{H})$
so that $t_\mathcal{H}(h) = \psi_0$. As the map $t_\mathcal{H}\circ\pr_1$ admits local sections,
there is an open neighborhood $U$ of $\psi_0$ in $\mathcal{H}_0$ and a continuous section
$\omega\co U\to\mathcal{H}_1\ftimes{s_{\mathcal{H}}}{\Phi_0}\mathcal{G}_0$ of $t_\mathcal{H}\circ\pr_1$ such that $\omega(\psi_0) = (h, \phi_0)$.
Let $\omega_i = \pr_i\circ\omega$ for $i = 1,2$, and then the map
$\alpha_{\Lambda_\Gamma\mathcal{H}}^{-1}(U)\ni\psi^\prime\mapsto\big((\omega_1(\psi_0^\prime), \omega_1(\psi_0^\prime)^{-1}\psi^\prime\omega_1(\psi_0^\prime)), \Phi_{\omega_2(\psi^\prime)}^{-1}\big(\omega_1(\psi_0^\prime)^{-1}\psi^\prime\omega_1(\psi_0^\prime)\big)\big)$ is a local section
of $t_{\Lambda_\Gamma\mathcal{H}}\circ\pr_1$ such that $\psi\mapsto\big( (h,\Phi_1\circ\phi),\phi\big)$.
Similarly, as $\Phi_1$ restricts to an isomorphism $\mathcal{G}_x^x\to\mathcal{H}_{\Phi_0(x)}^{\Phi_0(x)}$
at each $x\in\mathcal{G}_0$, the fact that $\mathcal{G}_1$ is a fibred product and
$\Phi_1$ induces an isomorphism $\Phi_x\co\HOM(\Gamma,\mathcal{G}_x^x)\to\HOM\big(\Gamma,\mathcal{H}_{\Phi_0(x)}^{\Phi_0(x)}\big)$
for each $x\in\mathcal{G}_0$ as above implies that
\[
    \xymatrix@C+1.5cm{
        \mathcal{G}_1\ftimes{s_{\mathcal{G}}}{\alpha_{\Lambda_\Gamma\mathcal{G}}}\HOM(\Gamma,\mathcal{G})
            \ar[d]_{(s_{\Lambda_\Gamma\mathcal{G}},t_{\Lambda_\Gamma\mathcal{G}})}
            \ar@{>}[r]^{(\Lambda_\Gamma\Phi)_1}
        &\mathcal{H}_1\ftimes{s_{\mathcal{H}}}{\alpha_{\Lambda_\Gamma\mathcal{H}}}\HOM(\Gamma,\mathcal{H})
            \ar[d]^{(s_{\Lambda_\Gamma\mathcal{H}},t_{\Lambda_\Gamma\mathcal{H}})}
        \\
        \HOM(\Gamma,\mathcal{G})\times\HOM(\Gamma,\mathcal{G})
            \ar[r]^{(\Lambda_\Gamma\Phi)_0\times(\Lambda_\Gamma\Phi)_0}
        &
        \HOM(\Gamma,\mathcal{H})\times\HOM(\Gamma,\mathcal{H})
    }
\]
is as well a fibred product, completing the proof.
\end{proof}

We recall the following, which was proven for arbitrary groupoids (without a topology)
in \cite[Theorem~3.1]{FarsiSeaGenOrbEuler}, and in \cite[Proposition 2-1(2)]{TamanoiCovering}
for the case of global quotient orbifolds. We summarize the proof for topological groupoids
here, referring the reader to \cite[Theorem~3.1]{FarsiSeaGenOrbEuler} for more details.

\begin{lemma}[(The $\Gamma$-inertia groupoid is iterative)]
\label{lem:IterateInertia}
Let $\Gamma_1$ and $\Gamma_2$ be finitely presented discrete groups. Then for a topological groupoid $\mathcal{G}$,
$\Lambda_{\Gamma_1\times\Gamma_2}\mathcal{G}$ is isomorphic to $\Lambda_{\Gamma_2}(\Lambda_{\Gamma_1}\mathcal{G})$.
\end{lemma}
\begin{proof}
An element of
$\HOM(\Gamma_2,\Lambda_{\Gamma_1}\mathcal{G}) = \HOM\big(\Gamma_2,\mathcal{G}\ltimes\HOM(\Gamma_1,\mathcal{G})\big)$
is given by a choice of $\phi\in\HOM(\Gamma_1,\mathcal{G})$ and
$\psi\in\HOM(\Gamma_2,(\Lambda_{\Gamma_1}\mathcal{G})_\phi^\phi) =
\HOM(\Gamma_2,C_{\mathcal{G}_{\phi_0}^{\phi_0}}(\phi_1))$.
Hence, the map sending $(\psi,\phi)$ to the pointwise product
$\phi_1\psi_1\in\HOM\big(\Gamma_1\times\Gamma_2,\mathcal{G}_{\phi_0}^{\phi_0}\big)$ induces a bijection
$e_0\co\HOM(\Gamma_2,\Lambda_{\Gamma_1}\mathcal{G})\to\HOM(\Gamma_1\times\Gamma_2,\mathcal{G})$.
The map $e_0$ is continuous because multiplication in $\mathcal{G}_1$ is continuous, and $e_0^{-1}$
is given by composition with the projections $\Gamma_1\times\Gamma_2\to\Gamma_i$ so is continuous
as well. Therefore, $e_0$ it is a homeomorphism. If $h\in\mathcal{G}_1$ such that $s_{\mathcal{G}}(h) = \phi_0$,
then $h\ast e_0(\phi,\psi) = e_0(h\ast\phi, h\ast\psi)$ so that $e_0$ is $\mathcal{G}$-equivariant.
It follows that $e_0$ induces an isomorphism
$e\co\Lambda_{\Gamma_2}(\Lambda_{\Gamma_1}\mathcal{G})\to\Lambda_{\Gamma_1\times\Gamma_2}\mathcal{G}$
of topological groupoids.
\end{proof}

The construction of the $\Gamma$-inertia groupoid also commutes with restriction to subsets
in the following sense.

\begin{lemma}[(The $\Gamma$-inertia groupoid is local in $\mathcal{G}_0$)]
\label{lem:InertRestriction}
Let $\mathcal{G}$ be a topological groupoid, let $\Gamma$ be a finitely presented discrete group, and
let $U\subseteq\mathcal{G}_0$. Recall that the anchor map
$\alpha_{\Lambda_\Gamma\mathcal{G}}\co(\Lambda_\Gamma\mathcal{G})_0 = \HOM(\Gamma,\mathcal{G})\to\mathcal{G}_0$
is defined by $\alpha_{\Lambda_\Gamma\mathcal{G}}(\phi) = \phi_0$; see Notation~\ref{not:HomGammaGpoid}.
Then $\Lambda_\Gamma(\mathcal{G}_{|U}) = (\Lambda_\Gamma\mathcal{G})_{| \alpha_{\Lambda_\Gamma\mathcal{G}}^{-1}(U)}$.
\end{lemma}
\begin{proof}
We have
\begin{align*}
    \HOM(\Gamma,\mathcal{G}_{|U})
        &=      \{\phi\in\HOM(\Gamma,\mathcal{G}) : \phi_0\in U\}
        \\&=    \alpha_{\Lambda_\Gamma\mathcal{G}}^{-1}(U)
\end{align*}
so that the object spaces of $\Lambda_\Gamma(\mathcal{G}_{|U})$ and
$(\Lambda_\Gamma\mathcal{G})_{|\alpha_{\Lambda_\Gamma\mathcal{G}}^{-1}(U)}$ coincide. Then the result follows from the fact that
each groupoid is given by the restriction of the $\mathcal{G}$-action to this set.
\end{proof}

We recall the following.

\begin{definition}[(Saturation)]
\label{def:Saturation}
Let $\mathcal{G}$ be a topological groupoid.
For a subset $U\subseteq\mathcal{G}_0$, the \emph{saturation of $U$}, denoted $\Sat(U)$,
is defined to be $\pi^{-1}(\pi(U))$, where we recall that $\pi\co\mathcal{G}_0\to|\mathcal{G}|$
denotes the orbit map. A set $U$ is \emph{saturated} if $U = \Sat(U)$.
\end{definition}

Let $|\alpha_{\Lambda_\Gamma\mathcal{G}}|\co|\Lambda_\Gamma\mathcal{G}| \to |\mathcal{G}|$ denote the map on orbit spaces induced
by $\alpha_{\Lambda_\Gamma\mathcal{G}}$, i.e., $|\alpha_{\Lambda_\Gamma\mathcal{G}}|\co\mathcal{G}\phi\mapsto\mathcal{G}\phi_0$.
Note that $|\alpha_{\Lambda_\Gamma\mathcal{G}}|$
is obviously surjective; for each $\mathcal{G}x\in|\mathcal{G}|$, there is a
$\phi\in\HOM(\Gamma,\mathcal{G}) = (\Lambda_\Gamma\mathcal{G})_0$ such that $\phi_0 = x$
and $\phi(\gamma) = u(x)$ for each $\gamma\in\Gamma$, and then
$|\alpha_{\Lambda_\Gamma\mathcal{G}}|(\mathcal{G}\phi) = \mathcal{G}x$. Note further that
$\pi_{\mathcal{G}}\circ \alpha_{\Lambda_\Gamma\mathcal{G}} = |\alpha_{\Lambda_\Gamma\mathcal{G}}|\circ\pi_\Lambda$, where we recall that $\pi_\Lambda$
denotes the orbit map for the groupoid $\Lambda_\Gamma\mathcal{G}$.

If $U\subseteq\mathcal{G}_0$ with the subspace topology, then the inclusion $\nu_0\co U\to\mathcal{G}_0$
induces a groupoid homomorphism $\nu\co\mathcal{G}_{|U}\to\mathcal{G}$, and it is easy to see that the map
$\nu_1$ on arrows is a homeomorphism onto its image. For the groupoid homomorphism
$\Lambda_\Gamma\nu\co\Lambda_\Gamma(\mathcal{G}_{|U})\to\Lambda_\Gamma\mathcal{G}$
described by Lemma~\ref{lem:InertHomomorph}, we have
$(\Lambda_\Gamma\nu)_0(\phi) = \nu_1\circ\phi$, so that $(\Lambda_\Gamma\nu)_0$
is simply the natural inclusion of $\HOM(\Gamma,\mathcal{G}_{|U})$ into $\HOM(\Gamma,\mathcal{G})$
as a subspace. Combining this observation with Lemmas~\ref{lem:InertHomomorph} and \ref{lem:InertRestriction},
we have the following.

\begin{corollary}
\label{cor:InertRestrictionEssEquiv}
Let $\mathcal{G}$ be a topological groupoid and let $\Gamma$ be a finitely presented discrete group.
Suppose $U\subseteq\mathcal{G}_0$ such that the inclusion $\nu\co\mathcal{G}_{|U}\to\mathcal{G}_{|\Sat(U)}$
is an essential equivalence. Then the induced map
$\Lambda_\Gamma\nu\co\Lambda_\Gamma(\mathcal{G}_{|U})\to(\Lambda_\Gamma\mathcal{G})_{|\alpha_{\Lambda_\Gamma\mathcal{G}}^{-1}(\Sat(U))}$
defined in Lemma~\ref{lem:InertHomomorph}
is an essential equivalence in particular inducing a homeomorphism of $|\Lambda_\Gamma(\mathcal{G}_{|U})|$
with $|\alpha_{\Lambda_\Gamma\mathcal{G}}|^{-1}(\pi(U))\subseteq|\Lambda_\Gamma(\mathcal{G})|$.
\end{corollary}


\subsection{The Euler characteristic of the $\Gamma$-inertia space}
\label{subsec:GpoidECSpace}

Let $\mathcal{G}$ be an orbit space definable groupoid and $\Gamma$ a finitely presented discrete group.
We introduce the notion of $\mathcal{G}$ being $\Gamma$-inertia definable, which ensures that
the inertia space $|\Lambda_\Gamma\mathcal{G}|$ has a well-defined Euler characteristic $\Lambda\chi_\Gamma(\mathcal{G})$ and that
this Euler characteristic coincides with $\chi_\Gamma(\mathcal{G})$; see Theorem~\ref{thrm:GamECInertia}.

\begin{definition}[($\Gamma$-inertia definable groupoid and $\Gamma$-inertia Euler characteristic)]
\label{def:GamInertiaDefinable}
Let $\Gamma$ be a finitely presented discrete group.
\begin{enumerate}
\item[(i)]
A \emph{$\Gamma$-inertia definable groupoid} $\mathcal{G}$ is an orbit space definable groupoid along with
an embedding $\iota_{|\Lambda_\Gamma\mathcal{G}|}$ of $|\Lambda_\Gamma\mathcal{G}|$ into
Euclidean space with respect to which $|\Lambda_\Gamma\mathcal{G}|$ is an affine definable space and
$|\alpha_{\Lambda_\Gamma\mathcal{G}}|$ is a morphism of affine definable spaces.

\item[(ii)]
If $\mathcal{G}$ is $\Z$-inertia definable, we define the \emph{inertia Euler characteristic} $\Lambda\chi(\mathcal{G})$ to be
\[
    \Lambda\chi(\mathcal{G})  =  \chi\big(|\Lambda_\Z\mathcal{G}|\big),
\]
the Euler characteristic of the affine definable space $|\Lambda_\Z\mathcal{G}|$.
If $\mathcal{G}$ is $\Gamma$-inertia definable, we define
the \emph{$\Gamma$-inertia Euler characteristic} $\Lambda\chi_\Gamma(\mathcal{G})$ to be
\[
    \Lambda\chi_\Gamma(\mathcal{G})  =  \chi\big(|\Lambda_\Gamma\mathcal{G}|\big),
\]
the Euler characteristic of the affine definable space $|\Lambda_\Gamma\mathcal{G}|$.
\end{enumerate}
As in the case of orbit definable groupoids, we will often say simply that $\mathcal{G}$ is $\Gamma$-inertia definable
and assume that $\iota_{|\Lambda_\Gamma\mathcal{G}|}$ denotes the chosen embedding.
\end{definition}

\begin{remark}
\label{rem:GamInertiaDefinableGOrbitSpaceDef}
If $\mathcal{G}$ is $\Gamma$-inertia definable, then both $|\mathcal{G}|$ and $|\Lambda_\Gamma\mathcal{G}|$
are affine definable spaces so that $\chi_\Gamma(\mathcal{G})$ and $\Lambda\chi_\Gamma(\mathcal{G})$ are defined.
The definition of $\Lambda\chi_\Gamma(\mathcal{G})$ makes sense if $|\Lambda_\Gamma\mathcal{G}|$ is an affine
definable space, even if $\mathcal{G}$ is not orbit space definable. However, our interest is in the relationship
between these Euler characteristics. Moreover, we will make frequent use of
the compatibility of the affine definable structures of $|\Lambda_\Gamma\mathcal{G}|$ and $|\mathcal{G}|$ via
$|\alpha_{\Lambda_\Gamma\mathcal{G}}|$, and we do not know of an example where $|\Lambda_\Gamma\mathcal{G}|$ admits an affine
definable structure while $|\mathcal{G}|$ does not.
\end{remark}

\begin{remark}
\label{rem:GamInertiaDefinableWellDef}
Once again, by Theorem~\ref{thrm:ECHomeoInvar}, $\Lambda\chi_\Gamma(\mathcal{G})$ depends only on the structure
of $\mathcal{G}$ as a topological groupoid and not on the o-minimal structure or choice of embeddings.
\end{remark}

If $G$ is a compact Lie group and $A$ is a semialgebraic $G$-set, then identifying $G$ with a compact linear algebraic group,
$G\ltimes A$ is $\Gamma$-inertia definable for any finitely presented discrete group $\Gamma$ by the description in
Remark~\ref{rem:InertiaTranslation} and Corollary~\ref{cor:DefinGpdOrbitDefin}. We now show that cocompact proper Lie
groupoids admit this structure as well.

\begin{lemma}[(Cocompact proper Lie groupoids are $\Gamma$-inertia definable)]
\label{lem:LieGpdGamInertDefin}
Let $\mathcal{G}$ be a cocompact proper Lie groupoid. Then for any finitely presented discrete group
$\Gamma$, $|\Lambda_\Gamma\mathcal{G}|$ is compact. Moreover, $|\mathcal{G}|$ and $|\Lambda_\Gamma\mathcal{G}|$
admit embeddings $\iota_{|\mathcal{G}|}$ and $\iota_{|\Lambda_\Gamma\mathcal{G}|}$, respectively,
with respect to which $\mathcal{G}$ is $\Gamma$-inertia definable.
\end{lemma}
\begin{proof}
Let $x\in\mathcal{G}_0$. By the slice theorem \cite[Corollary~3.11]{PPTOrbitSpace}, there is a neighborhood
$U_x$ of $x$ in $\mathcal{G}_0$ diffeomorphic to $O_x\times V_x$ where $V_x$ is an invariant open neighborhood
of the origin in a linear representation of $\mathcal{G}_x^x$ and $O_x$ is an open neighborhood of $x$ in its
orbit, such that $\mathcal{G}_{|U_x}$ is isomorphic to the product of $\mathcal{G}_x^x\ltimes V_x$ and the pair
groupoid on $O_x$. Choosing a presentation $\Gamma = \langle\gamma_1,\ldots,\gamma_\ell\mid R_1,\ldots,R_k\rangle$
of $\Gamma$, it follows as in Remark~\ref{rem:InertiaTranslation} that the open subset
$\HOM(\Gamma,\mathcal{G}_{|U_x})$ of $\HOM(\Gamma,\mathcal{G})$ given by
those $\phi$ such that $\phi_0\in U_x$ can be identified via $\phi\mapsto  (\phi(\gamma_1),\ldots,\phi(\gamma_r))$
with the set of points $(g_1,\ldots,g_\ell,v,o)\in (\mathcal{G}_x^x)^\ell\times V_x\times O_x$ such that the
$g_i$ satisfy the algebraic relations $R_j$ and $g_i(v) = v$ for each $i$. Recalling that
$\pi_\Lambda\co\HOM(\Gamma,\mathcal{G})\to|\Lambda_\Gamma\mathcal{G}|$ denotes the orbit map,
it follows that $\pi_\Lambda(\HOM(\Gamma,\mathcal{G}_{|U_x}))$ is homeomorphic to the quotient of
an algebraic subset of $(\mathcal{G}_x^x)^\ell\times V_x$ by the diagonal of the linear
$\mathcal{G}_x^x$-action on $V_x$ and component-wise conjugation on $(\mathcal{G}_x^x)^\ell$.
Then by \cite[Corollary~1.6]{BrumfielQuot}, $\pi_\Lambda(\HOM(\Gamma,\mathcal{G}_{|U_x}))$ is a
semialgebraic set and hence definable.

For each orbit in $|\mathcal{G}|$, choose a representative $x$ from the orbit and a set $U_x$ as above.
Via the identification of $U_x$ with $V_x\times O_x$, let $U_x^\prime$ be a $\mathcal{G}_x^x$-invariant
open neighborhood of $x$ such that $\overline{U_x^\prime}\subset U_x$. Then as $|\mathcal{G}|$ is compact,
it can be covered by finitely many $\pi(U_x^\prime)$ so that $|\Lambda_\Gamma\mathcal{G}|$ is covered by
finitely many $\pi_\Lambda(\HOM(\Gamma,\mathcal{G}_{|U_x^\prime}))$. As each
$\pi_\Lambda(\HOM(\Gamma,\mathcal{G}_{|\overline{U_x^\prime}}))$ is compact,
$|\Lambda_\Gamma\mathcal{G}|$ is compact as well. Then $|\Lambda_\Gamma\mathcal{G}|$
is a compact Hausdorff (hence regular) semialgebraic space in the sense of \cite[Section~7, Definition~3]{DelfsKnebuschSTII},
hence a locally semialgebraic space in the sense of \cite[Section~1, Definition~3]{DelfsKnebuschIntro}.
In each $\pi_\Lambda(\HOM(\Gamma,\mathcal{G}_{|U_x^\prime}))$, $|\alpha_{\Lambda_\Gamma\mathcal{G}}|$ corresponds
to the restriction of the semialgebraic function
$\mathcal{G}_x^x\backslash\big((\mathcal{G}_x^x)^\ell\times V_x\big)\to \mathcal{G}_x^x\backslash V_x$
induced by the projection $(\mathcal{G}_x^x)^\ell\times V_x \to V_x$ so that the induced map
$|\alpha_{\Lambda_\Gamma\mathcal{G}}|_{|\pi_\Lambda(\HOM(\Gamma,\mathcal{G}_{|U_x^\prime}))}$ is semialgebraic by Lemma~\ref{lem:SemialgEquivarFunc}.
Identifying $V_x$ with $\{e\}\times V_x\subset(\mathcal{G}_x^x)^\ell\times V_x$ and hence each weak orbit type in
$\mathcal{G}_x^x\backslash V_x$ with a semialgebraic subspace of $\mathcal{G}_x^x\backslash\big((\mathcal{G}_x^x)^\ell\times V_x\big)$,
by \cite[Theorem~3.2]{DelfsKnebuschIntro}, $|\Lambda_\Gamma\mathcal{G}|$ admits a locally finite,
hence finite, triangulation such that each weak orbit type in $|\mathcal{G}|$ is a subcomplex. Using these triangulations to define respective
embeddings $\iota_{|\Lambda_\Gamma\mathcal{G}|}$ and $\iota_{|\mathcal{G}|}$ of $|\Lambda_\Gamma\mathcal{G}|$ and $|\mathcal{G}|$
as semialgebraic subspaces of Euclidean space as in the proof of Proposition~\ref{prop:LieGpdOrbitDefin}, it follows that
$|\alpha_{\Lambda_\Gamma\mathcal{G}}|$ is a morphism of the resulting affine definable spaces.
\end{proof}

With this, the application of Fubini's theorem to $|\alpha_{\Lambda_\Gamma\mathcal{G}}|$ yields the following.

\begin{theorem}[($\chi_\Gamma$ and $\Lambda\chi_\Gamma$ coincide)]
\label{thrm:GamECInertia}
If $\Gamma$ is a finitely presented discrete group and $\mathcal{G}$ is a $\Gamma$-inertia definable groupoid, then
\[
    \chi_\Gamma(\mathcal{G})
    = \Lambda\chi_\Gamma(\mathcal{G}).
\]
\end{theorem}
\begin{proof}
We apply Fubini's theorem (Theorem~\ref{thrm:Fubini}) to the morphism of affine definable spaces
$|\alpha_{\Lambda_\Gamma\mathcal{G}}|\co|\Lambda_\Gamma\mathcal{G}| \to |\mathcal{G}|$ to yield
\begin{align*}
    \chi(|\Lambda_\Gamma\mathcal{G}|)
    &=          \int_{|\Lambda_\Gamma\mathcal{G}|} 1 \,d\chi(\mathcal{G}\phi)
    \\&=        \int_{|\mathcal{G}|}
                    \left(  \int_{|\alpha_{\Lambda_\Gamma\mathcal{G}}|^{-1}(\mathcal{G}x)} 1\,d\chi(\mathcal{G}\phi)  \right) \,d\chi(\mathcal{G}x)
    \\&=        \int_{|\mathcal{G}|}
                    \chi\big( \mathcal{G}_x^x\backslash\!\HOM(\Gamma,\mathcal{G}_x^x)\big)
                    \,d\chi(\mathcal{G}x)
    \\&=        \chi_\Gamma(\mathcal{G}).
    \qedhere
\end{align*}
\end{proof}

It follows that when $\mathcal{G}$ is an orbifold groupoid presenting the orbifold $Q$, the
$\chi_\Gamma(\mathcal{G})$ coincide with the $\chi_\Gamma(Q)$ defined in \cite{FarsiSeaGenOrbEuler} and
recalled in Section~\ref{subsec:BackOrbEuler}.

If $A$ is a locally closed semialgebraic set and $\Bbbk$ is a field, the Borel--Moore homology groups $H_i^{\operatorname{BM}}(A;\Bbbk)$
can be defined in terms of the simplicial homology groups $H_i(A;\Bbbk)$. If $A$ is compact, then
$H_i^{\operatorname{BM}}(A;\Bbbk) = H_i(A;\Bbbk)$, and if $A$ is not compact, then $H_i^{\operatorname{BM}}(A;\Bbbk)$
is defined using relative homology and a compactification of $A$; see \cite[Definition~11.7.13]{BochnakCosteRoyBook} and \cite{BorelMoore}
for the definition in a more general setting. If $A$ is a locally compact semialgebraic set, then it is demonstrated in
\cite[Section~1.8]{CosteRealAlgSets} that the Euler characteristic $\chi(A)$ is equal to the Euler characteristic of the
Borel--Moore homology of $A$ with coefficients in $\Z/2\Z$, i.e., the alternating sum of the Betti numbers.
Applying this fact and Theorem~\ref{thrm:GamECInertia}, we have the following.

\begin{corollary}
\label{cor:BettiEuler}
If $\Gamma$ is a finitely presented discrete group and $\mathcal{G}$ is a $\Gamma$-inertia definable groupoid
such that $|\Lambda_\Gamma\mathcal{G}|$ is a locally compact semialgebraic set, then $\Lambda\chi_\Gamma(\mathcal{G})$
and $\chi_\Gamma(\mathcal{G})$ are equal to the Euler characteristic of the Borel--Moore homology of $|\Lambda_\Gamma\mathcal{G}|$
with coefficients in $\Z/2\Z$. If $|\Lambda_\Gamma\mathcal{G}|$ is in addition compact,
then $\Lambda\chi_\Gamma(\mathcal{G})$ and $\chi_\Gamma(\mathcal{G})$ are equal to the Euler characteristic of the usual homology of
$|\Lambda_\Gamma\mathcal{G}|$.
\end{corollary}

The Morita equivalence class of $\Lambda_\Gamma\mathcal{G}$ as a topological groupoid depends only on the Morita equivalence
class of $\mathcal{G}$ by Lemma~\ref{lem:InertHomomorph}. The orbit spaces of Morita equivalent
groupoids are homeomorphic so that the topological space $|\Lambda_\Gamma\mathcal{G}|$ depends only on the (topological) Morita equivalence
class of $\mathcal{G}$. By Theorem~\ref{thrm:ECHomeoInvar}, $\Lambda\chi_\Gamma(\mathcal{G})$ depends only on the Morita equivalence
class of $\mathcal{G}$, and then the same holds for $\chi_\Gamma(\mathcal{G})$ by Theorem~\ref{thrm:GamECInertia}. That is, we have the following.

\begin{corollary}[(Morita invariance of $\chi_\Gamma$)]
\label{cor:ChiMoritaInvar}
If $\Gamma$ is a finitely presented discrete group and $\mathcal{G}$ is a $\Gamma$-inertia definable groupoid, then
$\Lambda\chi_\Gamma(\mathcal{G})$ and $\chi_\Gamma(\mathcal{G})$ depend only on the Morita equivalence class of $\mathcal{G}$
as a topological groupoid.
\end{corollary}

Now, suppose $\Gamma$ is a finitely presented discrete group and $\mathcal{G}$ is a $\Gamma$-inertia
definable groupoid. Recall that $\pi = \pi_{\mathcal{G}}\co\mathcal{G}_0\to|\mathcal{G}|$ denotes the
quotient map. If $U\subseteq\mathcal{G}_0$ is such that $\pi(U)$ is a definable subset of $|\mathcal{G}|$,
then $|\alpha_{\Lambda_\Gamma\mathcal{G}}|^{-1}(\pi(U))$ is a definable subset of $|\Lambda_\Gamma\mathcal{G}|$.
It follows that $\mathcal{G}_{|U}$ is as well $\Gamma$-inertia definable via the restrictions of the embeddings
$\iota_{|\mathcal{G}|}$ and $\iota_{|\Lambda_\Gamma\mathcal{G}|}$. Then using
Corollary~\ref{cor:InertRestrictionEssEquiv} and Theorem~\ref{thrm:GamECInertia}, the additivity
of $\chi$ extends to $\chi_\Gamma$ and $\Lambda\chi_\Gamma$ as follows.

\begin{corollary}[(Additivity of $\chi_\Gamma$ and $\Lambda\chi_\Gamma$)]
\label{cor:GamECAdditive}
Let $\Gamma$ be a finitely presented discrete group and $\mathcal{G}$ a $\Gamma$-inertia definable groupoid.
If $S, T\subseteq|\mathcal{G}|$ are definable
subsets such that $S\cup T = |\mathcal{G}|$ and $U,V,W\subseteq\mathcal{G}_0$ are such that the inclusions
$\mathcal{G}_{|U}\to\mathcal{G}_{|\pi^{-1}(S)}$, $\mathcal{G}_{|V}\to\mathcal{G}_{|\pi^{-1}(T)}$, and
$\mathcal{G}_{|W}\to\mathcal{G}_{|\pi^{-1}(S\cap T)}$ are essential equivalences of topological groupoids,
then
\begin{equation}
\label{eq:GamECAdditive}
    \chi_\Gamma(\mathcal{G})
        =   \chi_\Gamma(\mathcal{G}_{|U}) + \chi_\Gamma(\mathcal{G}_{|V})
            - \chi_\Gamma(\mathcal{G}_{|W})
\end{equation}
and
\begin{equation}
\label{eq:GamECAdditiveLambda}
    \Lambda\chi_\Gamma(\mathcal{G})
        =   \Lambda\chi_\Gamma(\mathcal{G}_{|U}) + \Lambda\chi_\Gamma(\mathcal{G}_{|V})
            - \Lambda\chi_\Gamma(\mathcal{G}_{|W})
\end{equation}
\end{corollary}
\begin{proof}
Note that the hypotheses imply $\pi(U) = S$, $\pi(V) = T$, and $\pi(W) = S\cap T$, and moreover that
$\mathcal{G}_{|U}$, $\mathcal{G}_{|V}$, and $\mathcal{G}_{|W}$ are orbit space definable. Hence via the
homeomorphisms given in Corollary~\ref{cor:InertRestrictionEssEquiv}, we can identify
$|\Lambda_\Gamma(\mathcal{G}_{|U})| = |\alpha_{\Lambda_\Gamma\mathcal{G}}|^{-1}(S)$,
$|\Lambda_\Gamma(\mathcal{G}_{|V})| = |\alpha_{\Lambda_\Gamma\mathcal{G}}|^{-1}(T)$, and
$|\Lambda_\Gamma(\mathcal{G}_{|W})| = |\alpha_{\Lambda_\Gamma\mathcal{G}}|^{-1}(S\cap T)$.
In particular, as $S$, $T$, and hence $S\cap T$ are definable subsets of $|\mathcal{G}|$, their preimages under the morphism
of affine definable spaces $|\alpha_{\Lambda_\Gamma\mathcal{G}}|$ are definable subsets of $|\Lambda_\Gamma(\mathcal{G})|$, so each of the
corresponding restrictions of $|\alpha_{\Lambda_\Gamma\mathcal{G}}|$ are as well morphisms of affine definable spaces. That is,
$\mathcal{G}_{|U}$, $\mathcal{G}_{|V}$, and $\mathcal{G}_{|W}$ are $\Gamma$-inertia
definable, and the homeomorphisms given by Corollary~\ref{cor:InertRestrictionEssEquiv} are morphisms of affine definable spaces. Then
$|\Lambda_\Gamma\mathcal{G}| = |\alpha_{\Lambda_\Gamma\mathcal{G}}|^{-1}(S) \cup |\alpha_{\Lambda_\Gamma\mathcal{G}}|^{-1}(T)$
and
$|\alpha_{\Lambda_\Gamma\mathcal{G}}|^{-1}(S)\cap|\alpha_{\Lambda_\Gamma\mathcal{G}}|^{-1}(T) = |\alpha_{\Lambda_\Gamma\mathcal{G}}|^{-1}(S\cap T)$
so that by the additivity of $\chi$, Equation~\eqref{eq:GamECAdditiveLambda} follows. Applying Theorem~\ref{thrm:GamECInertia} yields
Equation~\eqref{eq:GamECAdditive}.
\end{proof}

Note that in Corollary~\ref{cor:GamECAdditive}, sets $U$, $V$, and $W$ of course always exist;
we can take $U = \pi^{-1}(S)$, $V = \pi^{-1}(T)$, and $W = \pi^{-1}(S\cap T)$ so that the
essential equivalences are isomorphisms.

Finally, we have that $\chi_\Gamma$ and $\Lambda\chi_\Gamma$ are also multiplicative; see also
\cite[Lemma~3.1]{GZEMH-HigherOrbEulerCompactGroup}.

\begin{lemma}[(Multiplicativity of $\chi_\Gamma$ and $\Lambda\chi_\Gamma$)]
\label{lem:GamECMultiplic}
Let $\Gamma$ be a finitely presented discrete group and let $\mathcal{G}$ and
$\mathcal{H}$ be $\Gamma$-inertia definable groupoids. Then $\mathcal{G}\times\mathcal{H}$
is $\Gamma$-inertia definable,
\[
    \chi_\Gamma(\mathcal{G}\times\mathcal{H})
        =   \chi_\Gamma(\mathcal{G})\chi_\Gamma(\mathcal{H}).
\]
and
\[
    \Lambda\chi_\Gamma(\mathcal{G}\times\mathcal{H})
        =   \Lambda\chi_\Gamma(\mathcal{G})\Lambda\chi_\Gamma(\mathcal{H}).
\]
\end{lemma}
\begin{proof}
In \cite[Proposition~3.2]{FarsiSeaGenOrbEuler}, it is demonstrated for arbitrary (not necessarily topological) groupoids
$\mathcal{G}$ and $\mathcal{H}$ that $\Lambda_\Gamma(\mathcal{G})\times\Lambda_\Gamma(\mathcal{H})$
is groupoid isomorphic to $\Lambda_\Gamma(\mathcal{G}\times\mathcal{H})$, where the map on objects is
given by the obvious map sending
$(\phi,\psi)\in\HOM(\Gamma,\mathcal{G})\times\HOM(\Gamma,\mathcal{H})$
to $\phi\times\psi\to\HOM(\Gamma,\mathcal{G}\times\mathcal{H})$, i.e.,
$(\phi\times\psi)(\gamma) = (\phi(\gamma),\psi(\gamma))$. If $\mathcal{G}$ and $\mathcal{H}$
are topological groupoids, then it is a straightforward verification that the function
$(\phi,\psi)\mapsto\phi\times\psi$ is a continuous map
$\HOM(\Gamma,\mathcal{G})\times\HOM(\Gamma,\mathcal{H})\to\HOM(\Gamma,\mathcal{G}\times\mathcal{H})$,
as the open sets of $\mathcal{G}_1\times\mathcal{H}_1$ are generated by products of open sets.
Similarly, the inverse map is the product of the projections and hence continuous.
This map is then a $\mathcal{G}\times\mathcal{H}$-equivariant homeomorphism and hence induces an
isomorphism of topological groupoids and a homeomorphism between $|\Lambda_\Gamma\mathcal{G}|\times|\Lambda_\Gamma\mathcal{H}|$
and $|\Lambda_\Gamma(\mathcal{G}\times\mathcal{H})|$. Then by Theorem~\ref{thrm:ECHomeoInvar} and the multiplicativity of $\chi$,
$\Lambda\chi_\Gamma(\mathcal{G}\times\mathcal{H}) = \Lambda\chi_\Gamma(\mathcal{G})\Lambda\chi_\Gamma(\mathcal{H})$
whenever $\Lambda\chi_\Gamma(\mathcal{G}\times\mathcal{H})$ is defined.

Recalling that $|\alpha_{\Lambda_\Gamma\mathcal{G}}|$ denotes the orbit map of the anchor of the $\Gamma$-inertia groupoid
of $\mathcal{G}$, we have that $|\alpha_{\Lambda_\Gamma(\mathcal{G}\times\mathcal{H})}|$ sends the $\mathcal{G}\times\mathcal{H}$-orbit
of $\phi\times\psi\in\HOM(\Gamma,\mathcal{G}\times\mathcal{H})$ to the $\mathcal{G}\times\mathcal{H}$-orbit of
$(\phi_0,\psi_0)\in\mathcal{G}_0\times\mathcal{H}_0$ and hence coincides with
$|\alpha_{\Lambda_\Gamma\mathcal{G}}|\times|\alpha_{\Lambda_\Gamma\mathcal{H}}|$
up to the homeomorphisms
$|\Lambda_\Gamma\mathcal{G}|\times|\Lambda_\Gamma\mathcal{H}|\simeq|\Lambda_\Gamma(\mathcal{G}\times\mathcal{H})|$
and $|\mathcal{G}\times\mathcal{H}|\simeq|\mathcal{G}\times\mathcal{H}|$.
Therefore, $\mathcal{G}\times\mathcal{H}$ is $\Gamma$-inertia definable via the product embeddings
$\iota_{|\mathcal{G}\times\mathcal{H}|} = \iota_{|\mathcal{G}|}\times\iota_{|\mathcal{H}|}$ and
$\iota_{|\Lambda_\Gamma(\mathcal{G}\times\mathcal{H})|} = \iota_{|\Lambda_\Gamma\mathcal{G}|}\times\iota_{|\Lambda_\Gamma\mathcal{H}|}$.
The multiplicativity of $\chi_\Gamma$ then follows from Theorem~\ref{thrm:GamECInertia}.
\end{proof}


\subsection{Examples}
\label{subsec:Examples}

Here, we give a few concrete examples of computations of $\chi_\Z(\mathcal{G})$ and $\chi_\Gamma(\mathcal{G})$.
We consider translation groupoids of actions of compact Lie groups on semialgebraic subsets of Euclidean space for simplicity.
See Section~\ref{subsec:Extensions} and \cite{TrentinagliaRoleReps} for examples of proper Lie groupoids that are not Morita
equivalent to translation groupoids (as Lie groupoids).

It is well known that the Euler characteristic of a compact Lie group of positive dimension is zero. Hence, if
$\mathcal{G}_x^x$ is abelian, then $\chi(\Ad_{\mathcal{G}_x^x}\!\backslash\mathcal{G}_x^x) = \chi(\mathcal{G}_x^x) = 0$.
However, the partition of $|\mathcal{G}|$ into the level sets of
$\chi(\Ad_{\mathcal{G}_x^x}\!\backslash\mathcal{G}_x^x)$ can still have
an interesting effect on $\chi_\Gamma(\mathcal{G})$, even when all isotropy groups are abelian, as we illustrate
with the following.

\begin{example}
\label{ex:SO2onS2}
Let $A = \Sp^2$ denote the unit sphere in $\R^3$ and let $G = \SO(2,\R)$ act by rotations about the $z$-axis.
Then the orbit space $|G\ltimes A|$ of the translation groupoid $G\ltimes A$ is given by an interval
$[-1, 1]$ parameterized by the $z$-coordinate of points in the orbit. The isotropy group of orbits corresponding
to $z = \pm 1$ is $\SO(2,\R)$, while the isotropy group of other points is trivial. Hence, for a finitely presented
discrete group $\Gamma$, the integral in Equation~\eqref{eq:GammaECGroupoid} can be expressed as
\begin{align*}
    \chi_\Gamma(G\ltimes A)
        &=      \int_{\{-1,1\}} \chi\big(\HOM(\Gamma,\SO(2,\R))\big)\,d\chi(z)
                    + \int_{(-1,1)} \,d\chi(z)
        \\&=    2\chi\big(\HOM(\Gamma,\SO(2,\R))\big) - 1,
\end{align*}
where we note that $\SO(2,\R)$ is abelian so that the conjugation action is trivial.
Hence, if $\Gamma = \Z^\ell$ for some $\ell\geq 1$, then
$\chi\big(\HOM(\Gamma,\SO(2,\R))\big) = \chi\big(\SO(2,\R)^\ell\big) = 0$, and
$\chi_{\Z^\ell}(G\ltimes A) = -1$. If $\Gamma = \Z/k\Z$ for some positive integer $k$,
then a choice of generator for $\Z/k\Z$ identifies $\HOM(\Gamma,\SO(2,\R))$ with the set of
$k$th roots of unity so that $\chi\big(\HOM(\Gamma,\SO(2,\R))\big) = k$ and
$\chi_{\Z/k\Z}(G\ltimes A) = 2k - 1$.

In addition, the inertia space of $G\ltimes A$ is given by two circles attached by a line segment,
which has Euler characteristic $-1$. For arbitrary $\Gamma$, $|\Lambda_\Gamma(G\ltimes A)|$
is given by two copies of $\HOM(\Gamma,\SO(2,\R))$ attached by a line segment.
\end{example}

The following presents an example of the computation of $\chi_\Gamma(\mathcal{G})$ where the groupoid $\mathcal{G}$
is not a Lie groupoid.

\begin{example}
\label{ex:SO2onX}
In Example~\ref{ex:SO2onS2}, we can replace $A$ with any $\SO(2,\R)$-invariant semialgebraic subset of $\R^3$,
and the resulting translation groupoid is orbit space definable by Corollary~\ref{cor:DefinGpdOrbitDefin}.
We will demonstrate below that it is as well $\Gamma$-inertia definable for any finitely presented group $\Gamma$;
see Lemma~\ref{lem:SemialgGSetInertiaDefin}. For example,
let $A$ be the union of the positive $z$-axis $Z = \{(0,0,z):z\geq 0\}$ and the closed unit disk in the
$xy$-plane $D = \{(x,y,0):x^2+y^2\leq 1\}$. Each point in $Z$ has isotropy $\SO(2,\R)$, and the remaining
points have trivial isotropy. Then $|G\ltimes A|$ is given by $Z$ with a closed interval attached by identifying
an endpoint of the interval to the origin in $Z$. Similarly, $|\Lambda_\Gamma (G\ltimes A)|$ is the product
$Z\times\HOM(\Gamma,\SO(2,\R))$ with an endpoint of a closed interval attached to the point
$(0,\mathbf{1})$ where $\mathbf{1}$ denotes the identity homomorphism. As $Z$ is a half-open interval,
$\chi(Z) = 0$ so that $\chi\big(Z\times\HOM(\Gamma,\SO(2,\R))\big) = \chi(Z)\big(\HOM(\Gamma,\SO(2,\R))\big) = 0$.
Hence, for any finitely presented $\Gamma$, $\chi_\Gamma(G\ltimes A)$ is the Euler characteristic of a half-open
interval, i.e., $\chi_\Gamma(G\ltimes A) = 0$.

Similarly, let $Z^\prime = \{(0,0,z):z\in\R\}$ denote the entire $z$-axis, and let $A^\prime = Z^\prime\cup D$.
Then $|G\ltimes A^\prime|$ is $Z^\prime$ with an endpoint of a closed interval attached to the origin, and
$|\Lambda_\Gamma (G\ltimes A^\prime)|$ is the product $Z^\prime\times\HOM(\Gamma,\SO(2,\R))$ with an endpoint of
a closed interval attached to $(0,\mathbf{1})$. Here, $\chi(Z^\prime) = -1$, so as the Euler characteristic
of the points with trivial isotropy vanishes as above,
$\chi_\Gamma(G\ltimes A^\prime)= \chi\big(Z^\prime\times\HOM(\Gamma,\SO(2,\R))\big)= -\chi\big(\HOM(\Gamma,\SO(2,\R))\big)$ for any finitely generated $\Gamma$.
\end{example}

\begin{example}
\label{ex:SO3onR3}
Let $A = \R^3$ and let $G = \SO(3,\R)$ act by its defining representation.
Then the orbit space $|G\ltimes A|$ of the translation groupoid $G\ltimes A$ is given by a ray
$[0,\infty)$ parameterized by the distance $r$ of points in the orbit to the origin. The isotropy
group of orbits corresponding to $r > 0$ is $\SO(2,\R)$, while the isotropy group of the point
$r = 0$ is $\SO(3,\R)$. Hence, for a finitely presented
group $\Gamma$, the integral in Equation~\eqref{eq:GammaECGroupoid} can be expressed as
\begin{align*}
    \chi_\Gamma(G\ltimes A)
        &=      \int_{\{0\}} \chi\big(\SO(3,\R)\backslash\!\HOM(\Gamma,\SO(3,\R))\big)\,d\chi(r)
                    + \int_{(0,\infty)} \chi\big(\HOM(\Gamma,\SO(2,\R))\big) \,d\chi(r)
        \\&=    \chi\big(\{0\}\big) \chi\big(\SO(3,\R)\backslash\!\HOM(\Gamma,\SO(3,\R))\big)
                    + \chi\big((0,\infty)\big) \chi\big(\HOM(\Gamma,\SO(2,\R))\big)
        \\&=    \chi\big(\SO(3,\R)\backslash\!\HOM(\Gamma,\SO(3,\R))\big) - \chi\big(\HOM(\Gamma,\SO(2,\R))\big).
\end{align*}
If $\Gamma = \Z$, then $\chi_\Z(G\ltimes A) = 1$, as $\Ad_{\SO(3,\R)}\!\backslash\!\SO(3,\R)$ is a closed interval
while $\Ad_{\SO(2,\R)}\!\backslash\!\SO(2,\R) = \SO(2,\R)$ is a circle.
See \cite[Section~4.2.6]{FarsiPflaumSeaton1} for a description of the inertia space
$|\Lambda_\Z (G\ltimes A)|$ in this case.
\end{example}


\section{$\Gamma$-Euler characteristics for translation and proper cocompact Lie groupoids}
\label{sec:TranslationProperLie}


\subsection{The $\Gamma$-Euler characteristic for translation groupoids}
\label{subsec:GpoidECTranslation}

In this section, we consider the case that $\mathcal{G} = G\ltimes A$ is a translation groupoid where $G$
is a compact Lie group and $A$ is a semialgebraic $G$-set so that the $\chi^{(\ell)}(A, G)$
are defined; see \cite{GZEMH-HigherOrbEulerCompactGroup} and Definition~\ref{def:GZEC}.
We assume without loss of generality that $G$ is a compact linear algebraic group and hence an algebraic subset of $\R^{m^2}$;
see Section~\ref{subsec:BackDefin}.
We will show that $\chi^{(\ell)}(A, G) = \chi_{\Z^\ell}(G\ltimes A)$, and hence that the
$\chi_\Gamma$ and $\Lambda\chi_\Gamma$ both generalize the $\chi^{(\ell)}$. Note that the $\chi^{(\ell)}(A, G)$ are defined
iteratively by integrating over the space of conjugacy classes in $G$, while the extension of
$\chi_\Gamma$ to groupoids in Equation~\eqref{eq:GammaECGroupoid} required changing the order of
integration to integrate over the orbit space of the groupoid. Here, we give a formulation of
$\chi_\Gamma$ for an arbitrary finitely presented discrete group $\Gamma$ where the integral is over the
space of conjugacy classes in $\HOM(\Gamma,G)$, and hence is closer to the spirit of
\cite{GZEMH-HigherOrbEulerCompactGroup}. This yields a definition of $\chi^{(\ell)}$ that is not
iterative.

\begin{lemma}[(Semialgebraic $G$-sets are $\Gamma$-inertia definable)]
\label{lem:SemialgGSetInertiaDefin}
Let $G$ be a compact linear algebraic group and $A$ a semialgebraic $G$-set. Then for any finitely presented discrete
group $\Gamma$, there are embeddings of $|G\ltimes A|$ and $|\Lambda_\Gamma(G\ltimes A)|$ into Euclidean space with respect
to which $\Lambda_\Gamma(G\ltimes A)$ is orbit space definable, and $G\ltimes A$ is $\Gamma$-inertia definable.
\end{lemma}
\begin{proof}
Using the description of $\HOM(\Gamma,G\ltimes A)$ given in Remark~\ref{rem:InertiaTranslation}, $\HOM(\Gamma,G\ltimes A)$ is a semialgebraic
$G$-set so that by Corollary~\ref{cor:DefinGpdOrbitDefin}, there is an embedding $\iota_{|\Lambda_\Gamma(G\ltimes A)|}$ of
$|\Lambda (G\ltimes A)| = G\backslash\!\HOM(\Gamma,G\ltimes A)$ into $\R^k$ with respect to which $\Lambda_\Gamma(G\ltimes A)$
is orbit space definable. The map $A\to\HOM(\Gamma,G\ltimes A)$ sending each $x\in A$ to $(x,\phi)$ where $\phi$ is the trivial
homomorphism is a topological $G$-embedding of $A$ into $\HOM(\Gamma,G\ltimes A)$.
Hence $G\backslash\! A$ is embedded into $G\backslash\!\HOM(\Gamma,G\ltimes A)$ as the orbits of the trivial homomorphisms,
and composing with $\iota_{|\Lambda_\Gamma(G\ltimes A)|}$ yields an embedding $\iota_{|G\ltimes A|}$ of $G\backslash\! A$ into
$\R^k$. As $A$ has finitely many orbit types, each semialgebraic, by
\cite[Theorem~2.6 and p.~635]{ChoiParkSuhSemialgSlice}, $G\ltimes A$ is orbit space definable using the embedding $\iota_{|G\ltimes A|}$.
Then $\alpha_{\Lambda_\Gamma(G\ltimes A)}$ corresponds to the restriction of the projection $G^\ell\times A\to A$ which
is $G$-equivariant and semialgebraic so that $|\alpha_{\Lambda_\Gamma(G\ltimes A)}|$ is a morphism of affine definable spaces by
Lemma~\ref{lem:SemialgEquivarFunc}.
\end{proof}

\begin{notation}
\label{not:ConjClassCent}
For an element $\phi\in \HOM(\Gamma,G)$, let $[\phi]_G$ denote the $G$-conjugacy class of $\phi$,
let $C_G(\phi)$ denote the centralizer of $\phi$ in $G$, and let $A^{\langle\phi\rangle}$ denote the
set of points in $A$ fixed by the image of $\phi$. The centralizer $C_G(\phi)$ is a closed subgroup of
$G$, and $A^{\langle\phi\rangle}$ is a $C_G(\phi)$-invariant subspace of $A$. The set $A^{\langle\phi\rangle}$
can be defined by a finite number of algebraic conditions using the description of $\phi$ in
Notation~\ref{not:FinitePres} and hence is a definable set. The set
$\bigcup_{\phi^\prime\in[\phi]_G} \{\phi^\prime\} \times A^{\langle\phi^\prime\rangle}$ is a $G$-invariant
subspace of $\HOM(\Gamma,G\ltimes A)\subseteq \HOM(\Gamma,G)\times A$ as
$A^{\langle g\phi g^{-1}\rangle} = g A^{\langle\phi\rangle}$ for any $g\in G$; see
Remark~\ref{rem:InertiaTranslation}.

For the special case $\Gamma = \Z$, we can identify
$\HOM(\Z,G)$ with $G$ by fixing a generator of $\Z$, and then $G\backslash\!\HOM(\Z,G) = \Ad_G\!\backslash G$
is the set of conjugacy
classes in $G$. For $g\in G$, we let $[g]_G$ denote the $G$-conjugacy class of $g$, $C_G(g)$ its centralizer,
and $A^{\langle g\rangle}$ the set of its fixed points in $A$.
\end{notation}

We have the following.

\begin{lemma}
\label{lem:InjectFixSet}
Let $G$ be a compact linear algebraic group, $A$ a semialgebraic $G$-set, and $\Gamma$ a finitely presented discrete group.
Fix an embedding $\iota_{|\Lambda_\Gamma(G\ltimes A)|}$ of $|\Lambda_\Gamma(G\ltimes A)|$ as in Lemma~\ref{lem:SemialgGSetInertiaDefin}.
For each $\phi\in \HOM(\Gamma,G)$, the groupoids $\mathcal{G} = C_G(\phi)\ltimes A^{\langle\phi\rangle}$ and
$\mathcal{H} = G\ltimes \big(\bigcup_{\phi^\prime\in[\phi]_G} \{\phi^\prime\} \times A^{\langle\phi^\prime\rangle}\big)$
are Morita equivalent as topological groupoids, hence the orbit spaces $|\mathcal{G}|$ and $|\mathcal{H}|$
are homeomorphic. Moreover, the groupoid $\mathcal{G}$ admits the structure of an orbit space definable groupoid with respect to which
the homeomorphism $|\mathcal{G}|\to |\mathcal{H}|\subseteq |\Lambda_\Gamma (G\ltimes A)|$ is a morphism of affine definable spaces
$|\mathcal{G}|\to |\Lambda_\Gamma (G\ltimes A)|$.
\end{lemma}
\begin{proof}
Let $\rho\co \mathcal{G}\to\mathcal{H}$
be the groupoid homomorphism with $\rho_0(x) = (\phi, x)$ and $\rho_1(g, x) = (g\phi g^{-1}, gx)$
for $g\in C_G(\phi)$. Then the map
$t\circ\pr_1\co \mathcal{H}_1\ftimes{s_{\mathcal{H}}}{\rho_0}\mathcal{G}_0 \to\mathcal{H}_0$
is given by $G\times A^{\langle\phi\rangle}\ni\big(g,(\phi,x)\big)\mapsto(g\phi g^{-1}, gx)$.
For any $\phi^\prime = g\phi g^{-1}\in [\phi]_G$ and
$x\in A^{\langle\phi^\prime\rangle}$, $(\phi^\prime, x)$ is the image under this map of
$g^{-1}(\phi^\prime, x) = (\phi,g^{-1}x)$ where $g^{-1}x\in A^{\langle\phi\rangle}$,
so $t\circ\pr_1$ is surjective. As well, $t\circ\pr_1$ admits local sections by choosing a
constant $G$-element, i.e., a section of the form
$(\phi^\prime, y)\mapsto \big(g, (g^{-1}\phi^\prime g, g^{-1}y)\big)$ where $g^{-1}\phi^\prime g = \phi$.
The restriction $\mathcal{H}_{|\rho_0(\mathcal{G}_0)}$ is obviously
isomorphic to $\mathcal{G}$ so that
$\mathcal{G}_1 = (\mathcal{G}_0\times\mathcal{G}_0)\ftimes{\rho_0\times\rho_0}{(s_{\mathcal{H}},t_{\mathcal{H}})}\mathcal{H}_1$
via $\rho$ and the corresponding maps $(s,t)$. Hence $\rho$ is an essential equivalence.
Finally, note that
$A^{\langle\phi\rangle} = A^{\langle\phi(\gamma_1)\rangle}\cap\cdots\cap A^{\langle\phi(\gamma_\ell)\rangle}$
for generators $\gamma_1,\ldots,\gamma_\ell$ of $\Gamma$ and hence is semialgebraic. Hence $\{\phi\}\times A^{\langle\phi\rangle}$
is a definable subset of $\HOM(\Gamma,G\ltimes A)$ and its image in $|\Lambda_\Gamma (G\ltimes A)|$ via the quotient map
is a definable subset. Then as $C_G(\phi)$
is compact, $A^{\langle\phi\rangle}$ is a semialgebraic $C_G(\phi)$-set so that it has compact isotropy groups
and finitely many orbit types, each a semialgebraic set, by \cite[Theorem~2.6 and p.~635]{ChoiParkSuhSemialgSlice}.
Then composing the homeomorphism $|\mathcal{G}|\to|\mathcal{H}|\subset|\Lambda_\Gamma (G\ltimes A)|$ with
the embedding $\iota_{|\Lambda_\Gamma(G\ltimes A)|}$ of $|\Lambda_\Gamma(G\ltimes A)|$ yields an embedding
of $|\mathcal{G}|$ into Euclidean space with respect to which $\mathcal{G}$ is orbit space definable and
$|\mathcal{G}|\to|\Lambda_\Gamma (G\ltimes A)|$ is a morphism of affine definable spaces by construction.
\end{proof}

We now have the following, demonstrating that the higher-order orbifold Euler characteristics of
Equation~\eqref{eq:HigherEulerCharGZDef} coincide with the $\chi_{\Z^\ell}$
of Equation~\eqref{eq:GammaECGroupoid} and the $\Lambda\chi_{\Z^\ell}$ of Definition~\ref{def:GamInertiaDefinable}
when they are all defined, i.e., in the case of a translation
groupoid for a semialgebraic $G$-set where $G$ is a compact Lie group. The idea behind the proof is that, by Lemma~\ref{lem:IterateInertia},
the $\Lambda_{\Z^\ell}(G\ltimes A)$ can be iterated, i.e.,
$\Lambda_{\Z^\ell}(G\ltimes A) = \Lambda_\Z(\Lambda_{\Z^{\ell-1}}(G\ltimes A))$, allowing
us to interpret each iteration of the recursive definition in Equation~\eqref{eq:HigherEulerCharGZDef}
in terms of applying $\Lambda_\Z$.

\begin{theorem}[($\chi_{\Z^\ell}$ generalizes $\chi^{(\ell)}$)]
\label{thrm:TranslationECequal}
Let $G$ be a compact linear algebraic group and $A$ a semialgebraic $G$-set. Then for each $\ell \geq 0$,
$\chi^{(\ell)}(A, G) = \chi_{\Z^\ell}(G\ltimes A) = \Lambda\chi_{\Z^\ell}(G\ltimes A)$. In particular,
$\chi^{(1)}(A, G) = \chi_\Z(G\ltimes A) = \Lambda\chi_\Z(G\ltimes A)$.
\end{theorem}
\begin{proof}
We prove by induction on $\ell$ that $\chi^{(\ell)}(A, G) = \Lambda\chi_{\Z^\ell}(G\ltimes A)$,
and then the claim follows by Theorem~\ref{thrm:GamECInertia}. If $\ell = 0$, then
$\chi^{(0)}(A, G) = \chi(G\backslash A)$ and $\Lambda_0(G\ltimes A) = G\ltimes A$ so that
$\chi^{(0)}(A, G) = \Lambda\chi_0(G\ltimes A)$.
Consider the case $\ell = 1$, and fix an embedding $\iota_{|\Lambda_{\Z}(G\ltimes A)|}$ as in Lemma~\ref{lem:SemialgGSetInertiaDefin}.
Fixing a generator of $\Z$, we identify $\HOM(\Z,G\ltimes A)$ with the subspace
$\{(g,x)\in G\times A : gx = x \}$ of $G\times A$, i.e., a homomorphism $\phi\co\Z\to G\ltimes A$
corresponds to the point $(\phi_1(1), \phi_0)\in G\times A$. Then the projection $\pr_1\co G\times A \to G$
restricts to a $G$-equivariant map $\pr_1\co\HOM(\Z,G\ltimes A)\to G$ and hence induces a map
$|\pr_1|\co|\Lambda_\Z(G\ltimes A)| \to \Ad_G\!\backslash G$. As $\pr_1$ is equivariant and semialgebraic,
$|\pr_1|$ is a morphism of affine definable spaces by Lemma~\ref{lem:SemialgEquivarFunc}, where
$\Ad_G\!\backslash G$ is given the structure of an affine definable space as in Remark~\ref{rem:QuotUnique}.
For $[g]_G\in \Ad_G\!\backslash G$, we have
\begin{align*}
    |\pr_1|^{-1}([g]_G)
        &=      G\backslash\{(h,x)\in G\times A : hx = x \text{ and } h\in [g]_G\}
        \\&=    G\backslash\big(\bigcup_{h\in[g]_G} \{h\} \times A^{\langle h\rangle}\big),
\end{align*}
which is a definable subset of $|\Lambda_\Z(G\ltimes A)|$
homeomorphic to $C_G(g)\backslash A^{\langle g\rangle}$ by Lemma~\ref{lem:InjectFixSet}.
By Theorem~\ref{thrm:ECHomeoInvar}, $\chi\big(C_G(g)\backslash A^{\langle g\rangle}\big)$ depends only on the
homeomorphism type of $C_G(g)\backslash A^{\langle g\rangle}$ so that we can express
\begin{align*}
    \chi^{(1)}(A,G)
        &=      \int_{\Ad_G\!\backslash G} \chi\big(C_G(g)\backslash A^{\langle g\rangle}\big) \, d\chi([g]_G)
        \\&=    \int_{\Ad_G\!\backslash G} \chi\big(|\pr_1|^{-1}([g]_G)\big) \, d\chi([g]_G)
        \\&=    \int_{\Ad_G\!\backslash G} \left(\int_{|\pr_1|^{-1}([g]_G)} 1\,d\chi\right)d\chi([g]_G).
\end{align*}
Applying Fubini's theorm to the map $|\pr_1|$, this is equal to
\[
    \int_{|\Lambda_\Z(G\ltimes A)|} 1\, d\chi([\phi]_G) = \Lambda\chi_\Z(G\ltimes A).
\]
Hence $\chi^{(1)}(A,G) = \Lambda\chi_\Z(G\ltimes A)$, and by Theorem~\ref{thrm:GamECInertia},
$\chi^{(1)}(A,G) = \chi_\Z(G\ltimes A)$.

Now assume $\chi^{(\ell-1)}(A,G) = \Lambda\chi_{\Z^{\ell-1}}(G\ltimes A)$ for some $\ell \geq 2$.
Fixing generators of $\Z^\ell$ and identifying $\phi\in\HOM(\Z^\ell,G)$ with the image of the generators
$(g_\ell,\ldots,g_1)\in G^\ell$,
$\HOM(\Z^\ell,G\ltimes A) = \{(g_\ell,\ldots,g_1,x)\in G^\ell\times A : g_i x = x\; \forall i \leq\ell \}$.
Choose an embedding $\iota_{|\Lambda_{\Z^{\ell}}(G\ltimes A)|}$ as in Lemma~\ref{lem:SemialgGSetInertiaDefin},
and let $\iota_{|\Lambda_{\Z^{\ell-1}}(G\ltimes A)|}$ be the restricted embedding constructed by identifying
$|\Lambda_{\Z^{\ell-1}}(G\ltimes A)|$ with the orbits in $|\Lambda_{\Z^{\ell}}(G\ltimes A)|$ of
homomorphisms such that $g_\ell$ is the identity.
The projection $\pr_1\co G^\ell\times A \to G$ mapping $(g_\ell,\ldots,g_1,x)\mapsto g_\ell$
is $G$-equivariant and induces as in the previous case a morphism of affine definable spaces
$|\pr_1^\ell|\co|\Lambda_{\Z^\ell}(G\ltimes A)| \to \Ad_G\!\backslash G$. For $g \in G$, the preimage
$|\pr_1^\ell|^{-1}([g]_G)$ is given by the $G$-quotient of the set of $(g_\ell,\ldots,g_1,x)\in G^\ell\times A$
such that the $g_i$ pairwise commute, $g_i x = x$ for each $i$, and $g_\ell \in [g]_G$. We may rewrite this as
\begin{align*}
    |\pr_1^\ell|^{-1}([g]_G)
        &=      G\backslash\Big(\bigcup\limits_{g_\ell\in [g]_G} \{g_\ell\}\times
                    \{(g_{\ell-1},\ldots,g_1,x)\in C_G(g_\ell)^{\ell-1}\times A^{\langle g_\ell\rangle}
                        : g_i x = x, \; g_i g_j = g_j g_i \}\Big)
        \\&=    G\backslash\Big(\bigcup\limits_{g_\ell\in [g]_G} \{g_\ell\}\times
                    \HOM(\Z^{\ell-1}, C_G(g_\ell)\ltimes A^{\langle g_\ell\rangle}) \Big)
        \\&=    G\backslash\Big(\bigcup\limits_{g_\ell\in [g]_G} \{g_\ell\}\times
                    \HOM(\Z^{\ell-1}, G\ltimes A)^{\langle g_\ell\rangle} \Big).
\end{align*}
Here, the $G$-action on $\HOM(\Z^{\ell-1}, G\ltimes A)$ corresponds to
the diagonal action on $(g_{\ell-1},\ldots,g_1,x)$. By Lemma~\ref{lem:InjectFixSet}, this is a definable subset of
$|\Lambda_{\Z^\ell}(G\ltimes A)|$ homeomorphic to $C_G(g)\backslash\!\HOM(\Z^{\ell-1}, G\ltimes A)^{\langle g\rangle}$, which by the
definition of the $G$-action is equal to $C_G(g)\backslash\!\HOM(\Z^{\ell-1}, C_G(g)\ltimes A^{\langle g\rangle})$.
Note that $A^{\langle g\rangle}$ is a semialgebraic $C_G(g)$-set so that
$C_G(g)\ltimes\HOM(\Z^{\ell-1}, C_G(g)\ltimes A^{\langle g\rangle}) =
\Lambda_{\Z^{\ell-1}}(C_G(g)\ltimes A^{\langle g\rangle})$ admits the structure of an orbit space definable groupoid by
Lemma~\ref{lem:SemialgGSetInertiaDefin}. Then by the inductive hypothesis, we have
\begin{align*}
    \chi^{(\ell)}(A,G)
        &=      \int_{Ad_G\!\backslash G} \chi^{(\ell-1)}\big(C_G(g)\backslash A^{\langle g\rangle}\big) \, d\chi([g]_G)
        \\&=    \int_{Ad_G\!\backslash G} \chi\big(|\Lambda_{\Z^{\ell-1}}(C_G(g)\ltimes A^{\langle g\rangle})|\big)
                    \, d\chi([g]_G),
\end{align*}
which by the previous discussion and the homeomorphism-invariance of $\chi$ given by Theorem~\ref{thrm:ECHomeoInvar} is equal to
\begin{align*}
        &=      \int_{Ad_G\!\backslash G} \chi\big(C_G(g)\backslash\!\HOM(\Z^{\ell-1}, C_G(g)\ltimes A^{\langle g\rangle})\big)
                    \, d\chi([g]_G)
        \\&=    \int_{Ad_G\!\backslash G} \chi\big(|\pr_1^\ell|^{-1}([g]_G)\big) \, d\chi([g]_G)
        \\&=    \int_{Ad_G\!\backslash G} \left(\int_{|\pr_1^\ell|^{-1}([g]_G)} 1\, d\chi([\phi]_G)  \right)\,
                    d\chi([g]_G).
\end{align*}
Applying Fubini's theorem, we continue
\begin{align*}
    &=      \int_{Ad_G\!\backslash G} \left(\int_{|\pr_1^\ell|^{-1}([g]_G)} 1\, d\chi([\phi]_G)  \right)\, d\chi([g]_G)
    \\&=    \int_{|\Lambda_{\Z^\ell}(G\ltimes A)|} 1\, d\chi([\phi]_G)
    \\&=    \Lambda\chi_{\Z^\ell}(G\ltimes A),
\end{align*}
completing the proof.
\end{proof}

In particular, Theorem~\ref{thrm:TranslationECequal} and Corollary~\ref{cor:ChiMoritaInvar}
imply the following.

\begin{corollary}[(Morita invariance of $\chi^{(\ell)}$)]
\label{cor:ChiGZEMHMoritaInvar}
Let $G$ be a compact linear algebraic group, $A$ a semialgebraic $G$-set, and $\Gamma$ a finitely presented
discrete group. Then the $\chi^{(\ell)}(G\ltimes A)$ depend only on the Morita
equivalence class of the translation groupoid $G\ltimes A$ as a topological groupoid.
\end{corollary}

Finally, using the same idea as the proof of Theorem~\ref{thrm:TranslationECequal}, we can give a description
of $\chi_\Gamma(G\ltimes A)$ for an arbitrary finitely presented discrete group $\Gamma$ that is very
closely related to the original definitions of $\chi^{(\ell)}$ in \cite{GZEMH-HigherOrbEulerCompactGroup}.

\begin{theorem}[($\chi_\Gamma$ for translation groupoids)]
\label{thrm:TranslationECNonIterat}
Let $G$ be a compact linear algebraic group, $A$ and semialgebraic $G$-set, and $\Gamma$ a finitely presented
discrete group. Then
\[
    \chi_\Gamma(G\ltimes A)
        =   \int_{G\backslash\!\HOM(\Gamma,G)} \chi\big(C_G(\phi)\backslash A^{\langle\phi\rangle}\big)
                \, d\chi([\phi]_G).
\]
\end{theorem}
\begin{proof}
The proof is similar to the case $\ell=1$ of Theorem~\ref{thrm:TranslationECequal}. The composition of
$\phi\in\HOM(\Gamma,G\ltimes A)$ with the $G$-equivariant projection $\pr_1\co G\times A\to G$ restricts
to a $G$-equivariant map $\pr_1^\Gamma\co\HOM(\Gamma,G\ltimes A)\to\HOM(\Gamma,G)$ and hence, choosing embeddings
as in Lemma~\ref{lem:SemialgGSetInertiaDefin} and Remark~\ref{rem:QuotUnique}, induces a morphism
$|\pr_1^\Gamma|\co|\Lambda_\Gamma(G\ltimes A)|\to G\backslash\!\HOM(\Gamma,G)$ of affine definable spaces. Using the identification
of $\HOM(\Gamma,G\ltimes A)$ with a subset of $\HOM(\Gamma,G)\times A$ in Remark~\ref{rem:InertiaTranslation},
we have for $[\phi]_G\in G\backslash\!\HOM(\Gamma,G)$ that
\begin{align*}
    |\pr_1^\Gamma|^{-1}([\phi]_G)
        &=      G\backslash\{(\phi^\prime,x) : \phi^\prime(\gamma)\in G_x\;
                    \forall\gamma\in\Gamma \text{ and } \phi^\prime\in [\phi]_G\}
        \\&=    G\backslash\big(\bigcup_{\phi^\prime\in[\phi]_G} \{\phi^\prime\}
                    \times A^{\langle\phi^\prime\rangle}\big),
\end{align*}
which is a definable subset of $|\Lambda_\Gamma(G\ltimes A)|$ homeomorphic to $C_G(\phi)\backslash A^{\langle\phi\rangle}$ by
Lemma~\ref{lem:InjectFixSet}. Then by Theorem~\ref{thrm:ECHomeoInvar},
\begin{align*}
    \int_{G\backslash\!\HOM(\Gamma,G)} \chi\big( C_G(\phi)\backslash A^{\langle\phi\rangle}\big) \, d\chi([\phi]_G)
        &=      \int_{G\backslash\!\HOM(\Gamma,G)} \chi\big( |\pr_1^\Gamma|^{-1}([\phi]_G)\big) \, d\chi([\phi]_G)
        \\&=    \int_{G\backslash\!\HOM(\Gamma,G)} \left(\int_{|\pr_1^\Gamma|^{-1}
                    ([\phi]_G)} 1\,d\chi([(\phi^\prime,x)]_G)\right) \, d\chi([\phi]_G)
        \\&=    \int_{|\Lambda_\Gamma(G\ltimes A)|} 1\, d\chi([(\phi^\prime,x)]_G),
\end{align*}
where in the last step we apply Fubini's theorem to $|\pr_1^\Gamma|$. This is equal to
$\Lambda\chi_\Gamma(G\ltimes A)$, and hence by Theorem~\ref{thrm:GamECInertia} to
$\chi_\Gamma(G\ltimes A)$, completing the proof.
\end{proof}

In particular, when $\Gamma=\Z^\ell$, we can express
\[
    \chi^{(\ell)}(G,A)
        =   \int_{G\backslash\!\HOM(\Z^\ell,G)} \chi\big( C_G(\phi)\backslash A^{\langle\phi\rangle}\big)
                \, d\chi([\phi]_G).
\]
Note that $\HOM(\Z^\ell,G)$ can be identified with the set of commuting $\ell$-tuples of elements of $G$,
and then the $G$-action is by component-wise conjugation.


\subsection{The $\chi_\Gamma$ for cocompact proper Lie groupoids}
\label{subsec:GpoidECLie}

In this section, for the case of a cocompact proper Lie groupoid, we describe an alternate formulation of
$\chi_\Gamma$ as a sum of Euler characteristics of orbit space definable translation groupoids. In the case
that these can be chosen to be semilagebraic translation groupoids, this expresses $\chi_\Gamma$
an ``integral over the group factor," closer to the spirit of Equations~\eqref{eq:EulerCharGZDef} and
\eqref{eq:HigherEulerCharGZDef} and Theorem~\ref{thrm:TranslationECNonIterat}; see
Equation~\eqref{eq:GamECLieGIntegral2}.

It will simplify matters to consider the language of \emph{orbispace charts}, see \cite[Section~3.2.1--2]{WangThesis},
and $\mathfrak{T}$-\emph{orbispace charts}; see \cite[Section~2.1]{PflaumOrbispace}. If $X$ is a connected Hausdorff
space, an \emph{orbispace chart} for $X$ is a triple $(V,G,\pi)$ where $V$ is a smooth manifold, $G$ is a Lie
group, and $\pi\co V\to X$ is a $G$-invariant map such that $\pi(V)$ is open, and $\pi$ induces
a homeomorphism of $G\backslash V$ onto $\pi(V)$. The orbispace chart $(V,G,\pi)$ is \emph{compact} if $G$ is compact and
\emph{linear} if $V$ is a $G$-invariant neighborhood of the origin in a linear representation of $G$.
Note that in \cite[Theorem~3.2.31 and Corollary~3.2.32]{WangThesis}, Wang demonstrates that every proper Lie
groupoid is Morita equivalent (as a Lie groupoid) to a groupoid associated to an \emph{orbispace atlas} for $|\mathcal{G}|$, a
collection of orbispace charts satisfying compatibility conditions; see \cite[Definition~3.2.3--4]{WangThesis}.

By the slice theorem \cite[Corollary~3.11]{PPTOrbitSpace}, if $X = |\mathcal{G}|$ for a proper Lie groupoid
$\mathcal{G}$, then every orbit $\mathcal{G}x\in|\mathcal{G}|$ is contained in the image of a compact linear
orbispace chart $(V_x,\mathcal{G}_x^x,\pi_x)$ where $x\in V_x\subseteq\mathcal{G}_0$
corresponds to the origin of the linear representation of $\mathcal{G}_x^x$ on $V_x$, $\mathcal{G}_{|V_x}$
is isomorphic to $\mathcal{G}_x^x\ltimes V_x$, and $\pi_x$ is the restriction to $V_x$ of the orbit map
$\pi\co\mathcal{G}_0\to|\mathcal{G}|$. We say that such an orbispace chart is \emph{centered at $x$}
and will always assume orbispace charts for $|\mathcal{G}|$ are of this form. Note that for such an
orbispace chart, the inclusion $V_x\to\mathcal{G}_0$ induces an essential equivalence
$\mathcal{G}_{|V_x}\to\mathcal{G}_{|\Sat(V_x)}$; see \cite[Proposition~3.7]{PPTOrbitSpace}.

\begin{lemma}
\label{lem:FiniteCharts}
Let $\mathcal{G}$ be a cocompact proper Lie groupoid and $\iota_{|\mathcal{G}|}$ an embedding of $|\mathcal{G}|$
into $\R^n$  with respect to which $\mathcal{G}$ is orbit space definable. Then there is a finite collection of compact
linear orbispace charts $\{V_i, G_i, \pi_i\}$ centered at
$x_i\in\mathcal{G}_0$, $i=1,\ldots,k$, such that each $\pi_i(V_i)\subseteq|\mathcal{G}|$
is a definable subset and $|\mathcal{G}| = \bigcup_{i=1}^k \pi_i(V_i)$.
\end{lemma}
\begin{proof}
By \cite[Proposition~3.11]{PPTOrbitSpace} each orbit $\mathcal{G}x\in |\mathcal{G}|$ is contained
in the image $\pi_x(V_x)\subseteq|\mathcal{G}|$ of a compact linear orbispace chart $(V_x,\mathcal{G}_x^x,\pi_x)$
for $|\mathcal{G}|$ centered at $x$. For each $x$, we can arrange
that $\pi_x(V_x)$ is a definable subset of $|\mathcal{G}|$ as follows. Note that
$\iota_{|\mathcal{G}|}\circ\pi_x(V_x)$ is an open neighborhood of
$\iota_{|\mathcal{G}|}(\mathcal{G}x)$ in $\iota_{|\mathcal{G}|}(|\mathcal{G}|)$ and hence there is an $\epsilon$-ball
$B_x$ about $\iota_{|\mathcal{G}|}(\mathcal{G}x)$ in $\R^n$ such that
$B_x\cap\iota_{|\mathcal{G}|}(|\mathcal{G}|)\subset\iota_{|\mathcal{G}|}\circ\pi_x(V_x)$. As $B_x$ is semialgebraic,
$B_x\cap\iota_{|\mathcal{G}|}(|\mathcal{G}|)$ is a definable set. Let
$Q_x = \iota_{|\mathcal{G}|}^{-1}\big(B_x\cap\iota_{|\mathcal{G}|}(|\mathcal{G}|)\big)$ and then $Q_x$ is a definable
subset of $|\mathcal{G}|$ that is an open neighborhood of $x$ in $\pi_x(V_x)$. Hence $\pi_x^{-1}(Q_x)$ is an open
$\mathcal{G}_x^x$-invariant neighborhood of the origin in $V_x$ so that $\big(\pi_x^{-1}(Q_x),\mathcal{G}_x^x,(\pi_x)_{|\pi_x^{-1}(Q_x)}\big)$
is a linear orbispace chart whose image is a definable subset. Hence, we redefine each $(V_x,\mathcal{G}_x^x,\pi_x)$ to be
$\big(\pi_x^{-1}(Q_x),\mathcal{G}_x^x,(\pi_x)_{|\pi_x^{-1}(Q_x)}\big)$.
By compactness, there is a finite set $(V_i,G_i,\pi_i)$, $i=1,\ldots,k$,
such that the $\pi_i(V_i)$ cover $|\mathcal{G}|$.
\end{proof}

Fix a set $\{V_i, G_i, \pi_i\}$, $i=1,\ldots,k$,
of orbispace charts covering $|\mathcal{G}|$ as in Lemma~\ref{lem:FiniteCharts}.
We now replace the $V_i$ with subsets $W_i\subseteq V_i$ such that
the $\pi_i(W_i)$ are disjoint. Specifically, define $S_1 = \pi_1(V_1)$ and
for $i \geq 1$, recursively define $S_{i+1} = \pi_{i+1}(V_{i+1})\cap(|\mathcal{G}|\smallsetminus \bigcup_{j=1}^i S_j)$.
Then the $S_i$ are disjoint definable subsets of $|\mathcal{G}|$ that cover $|\mathcal{G}|$ by definition.
Set $W_i = V_i\cap\pi_i^{-1}(S_i)$, and then each $W_i$ is a $G_i$-invariant
subset of $V_i$. Note that the $W_i$ may no longer be open and hence are not necessarily manifolds,
but the essential equivalence $\mathcal{G}_{|V_i}\to\mathcal{G}_{|\Sat(V_i)}$ of Lie groupoids obviously restricts
to an essential equivalence $\mathcal{G}_{|W_i}\to\mathcal{G}_{|\Sat(W_i)}$ of topological groupoids.
As the $S_i$ are definable subsets of $|\mathcal{G}|$, the $\mathcal{G}_{|W_i} = G_i\ltimes W_i$ are orbit space definable by
Lemma~\ref{lem:OrbDefinGpoidConstructions}(ii).
Noting that the $S_i$ are pairwise disjoint, an application of Corollary~\ref{cor:GamECAdditive} yields the following.

\begin{theorem}[($\chi_\Gamma$ for a cocompact proper Lie groupoid)]
\label{thrm:GamECLieGIntegral}
Let $\mathcal{G}$ be a cocompact proper Lie groupoid and $\iota_{|\mathcal{G}|}$ an embedding of $|\mathcal{G}|$
into $\R^n$ with respect to which $\mathcal{G}$ is orbit space definable. Let $\{V_i, G_i, \pi_i\}$,
$i=1,\ldots,k$, be a finite set of compact linear orbispace charts covering $|\mathcal{G}|$ whose images $\pi_i(V_i)$ are
definable subsets of $|\mathcal{G}|$, and define $W_i$ and $S_i$ as above. Then for any finitely presented discrete group $\Gamma$,
\begin{equation}
\label{eq:GamECLieGIntegral1}
    \chi_\Gamma(\mathcal{G})
        =   \sum\limits_{i=1}^k \chi_\Gamma(G_i\ltimes W_i).
\end{equation}
\end{theorem}

\begin{remark}
\label{rem:GamECLieGIntegral}
If each $W_i$ can be chosen to be a semialgebraic subset of the representation space containing $V_i$, then
applying Theorem~\ref{thrm:TranslationECNonIterat}, Equation~\eqref{eq:GamECLieGIntegral1} can be written
\begin{equation}
\label{eq:GamECLieGIntegral2}
    \chi_\Gamma(\mathcal{G})
        =   \sum\limits_{i=1}^k \int_{G_i\!\backslash\!\HOM(\Gamma,G_i)}
                \chi\big( C_{G_i}(\phi)\backslash W_i^{\langle\phi\rangle}(\phi)\big) \, d\chi([\phi]_{G_i}).
\end{equation}
This is the case, for instance, if each $V_i$ is a semialgebraic subset of the corresponding representation
space, $\iota_{|\mathcal{G}|}(|\mathcal{G}|)$ is semialgebraic, and the map $\pi_i\co V_i\to|\mathcal{G}|$ is
a morphism of affine definable sets in the o-minimal structure of semialgebraic sets.
\end{remark}


\subsection{Abelian extensions of translation groupoids}
\label{subsec:Extensions}

Let $\mathcal{G}$ be a proper Lie groupoid such that $|\mathcal{G}|$ is connected. Let us briefly recall the following language from
\cite{TrentinagliaGlobStruc,TrentinagliaRoleReps}; while not necessary for the sequel, it will help motivate our
observations in this section. A \emph{representation} of $\mathcal{G}$ is a vector bundle
$E\to\mathcal{G}_0$ and a Lie groupoid homomorphism $\mathcal{G}\to\GL(E)$ where
$\GL(E)$ denotes the groupoid with objects $\mathcal{G}_0$ and arrows given by
linear isomorphisms between fibers $E_x$; see \cite[Section 1.1]{TrentinagliaRoleReps} for more details.
A representation $\mathcal{G}\to\GL(E)$ of $\mathcal{G}$ is \emph{effective} at $x\in\mathcal{G}_0$ if the kernel of
the restriction $\mathcal{G}_x^x\to\GL(E_x)$ is contained in the \emph{ineffective part} of $\mathcal{G}_x^x$, the kernel
of the action of $\mathcal{G}_x^x$ on a slice at $x$; a representation is \emph{globally effective} if it is effective
at every $x\in\mathcal{G}_0$; see \cite[Definition 2]{TrentinagliaGlobStruc}. The groupoid $\mathcal{G}$
is \emph{reflexive} if for each $x\in\mathcal{G}_0$, there is a representation $\mathcal{G}\to\GL(E)$
such that the restriction $\mathcal{G}_x^x\to\GL(E_x)$ is injective. It is
is \emph{parareflexive} if for each $x\in\mathcal{G}_0$, there is a representation $\mathcal{G}\to\GL(E)$ that is effective at $x$.
Note that if $\mathcal{G}$ is Morita equivalent
(as a Lie groupoid) to a translation groupoid, then it is reflexive; see \cite{TrentinagliaGlobStruc}. See
\cite[Example~2.10]{TrentinagliaRoleReps} for an example of a proper Lie groupoid that is not reflexive and hence not Morita
equivalent as a Lie groupoid to a translation groupoid.

Trentinaglia demonstrated in
\cite[Corollary 4]{TrentinagliaGlobStruc} that if a proper Lie groupoid $\mathcal{G}$ with connected orbit space
admits a globally effective
representation, then it is Lie groupoid Morita equivalent to an extension of a translation groupoid by a bundle of compact
groups; see Definitions~\ref{def:BundleGrps} and \ref{def:ExtensionTrans} below.
Note that Trentinaglia states this as a corollary to a theorem with a strong hypothesis on the codimensions of orbits,
but the corollary and its proof only require the existence of a globally effective representation.
Note that in \cite[p.709]{TrentinagliaGlobStruc}, Trentinaglia poses the questions of whether every
proper Lie groupoid is parareflexive, and whether every parareflexive proper Lie groupoid admits
a globally effective representation. As far as we are aware, these questions remain open.

Here, we consider
$\chi_\Gamma(\mathcal{G})$ of such a groupoid. As we will see, if the isotropy groups of $\mathcal{G}$ are
abelian and $\Gamma = \Z^\ell$, then $\chi_\Gamma(\mathcal{G})$ is closely related to the $\Gamma$-Euler
characteristic of the corresponding translation groupoid, though this relationship does not appear to extend
to the nonabelian case.

We begin with the following.

\begin{definition}[({Bundle of Lie groups, \cite[Appendix~A.1]{MackenzieLGLADiffGeom},
\cite[Section~1.3(c)]{MoerdijkClassifRegular}})]
\label{def:BundleGrps}
A \emph{bundle of Lie groups} is a Lie groupoid $\mathcal{G}$ such that $s = t$. A bundle of Lie groups is
\emph{locally trivial} if each $x\in\mathcal{G}_0$ is contained in a neighborhood $U$ such that
$\mathcal{G}_{|U}$ is diffeomorphic to the trivial bundle of groups $\mathcal{G}_x^x\times U$.
By a \emph{bundle of compact Lie groups}, we will mean a locally trivial bundle of Lie groups such that
$s$ is proper.
\end{definition}

If $\mathcal{G}$ is a bundle of compact Lie groups,
then $|\mathcal{G}| = \mathcal{G}_0$ and for each $x,y\in\mathcal{G}_0$, $\mathcal{G}_x^x\simeq\mathcal{G}_y^y$. In
particular, an orbit space definable structure on $\mathcal{G}$ is an embedding $\iota_{|\mathcal{G}|}$ of
$|\mathcal{G}| = \mathcal{G}_0$ with respect to which $\mathcal{G}_0$ is an affine definable space.

\begin{lemma}[($\chi_\Gamma$ for bundles of compact Lie groups)]
\label{lem:BundleGroups}
Suppose $\mathcal{G}$ is a bundle of compact Lie groups such that $\mathcal{G}_0$ is connected and
$\iota_{|\mathcal{G}|}$ is an embedding of $|\mathcal{G}|$ with respect to which $\mathcal{G}$ is orbit
space definable. Then for any finitely presented discrete group $\Gamma$ and any $x\in\mathcal{G}_0$,
\[
    \chi_\Gamma(\mathcal{G})
        =   \chi(\mathcal{G}_0)\chi\big(\Ad_{\mathcal{G}_x^x}\!\backslash\!\HOM(\Gamma,\mathcal{G}_x^x) \big)
        =   \chi(\mathcal{G}_0)\chi_\Gamma(\mathcal{G}_x^x),
\]
where $\chi_\Gamma(\mathcal{G}_x^x)$ is the $\Gamma$-Euler characteristic of
$\mathcal{G}_x^x$, treated as a groupoid with a single object.
\end{lemma}
\begin{proof}
Using Equation~\eqref{eq:GammaECGroupoid} and the fact that $\mathcal{G}_x^x$ does not depend on $x$,
\begin{align*}
    \chi_\Gamma(\mathcal{G})
        &=      \int_{ |\mathcal{G}|}
                    \chi\big(\mathcal{G}_{x}^x\backslash\!\HOM(\Gamma,\mathcal{G}_{x}^x) \big) \, d\chi(\mathcal{G}x)
        \\&=    \int_{\mathcal{G}_0}
                    \chi\big(\mathcal{G}_{x}^x\backslash\!\HOM(\Gamma,\mathcal{G}_x^x) \big) \, d\chi(x)
        \\&=    \chi(\mathcal{G}_0)\chi\big(\mathcal{G}_{x}^x\backslash\!\HOM(\Gamma,\mathcal{G}_x^x) \big).
        \qedhere
\end{align*}
\end{proof}

Following \cite{TrentinagliaGlobStruc}, we consider extensions of the following form.

\begin{definition}[(Extension of a translation groupoid)]
\label{def:ExtensionTrans}
A proper Lie groupoid $\mathcal{G}$ is an \emph{extension of a translation groupoid by a bundle of compact
Lie groups} if there is a short exact sequence
\begin{equation}
\label{eq:ExtensionTrans}
    1\to\mathcal{B}
    \stackrel{\nu}{\rightarrow}
        \mathcal{G}
    \stackrel{\rho}{\rightarrow}
    H\ltimes\mathcal{G}_0 \to 1
\end{equation}
such that $\mathcal{B}$ is a bundle of compact groups with object space $\mathcal{B}_0 = \mathcal{G}_0$,
$\nu_0$ and $\rho_0$ are the identity maps, and $H$ is a compact Lie group acting smoothly
on $\mathcal{G}_0$.
\end{definition}

If $\mathcal{G}$ is an extension as in Definition~\ref{def:ExtensionTrans},
then by the surjectivity of $\rho$, the orbit spaces $|\mathcal{G}|$ and $|H\ltimes\mathcal{G}_0|$
coincide. Specifically, for $x,y\in\mathcal{G}_0$, there is a $g\in\mathcal{G}$ with
$s(g) = x$ and $t(g) = y$ if and only if there is an $h\in H$ such that $hx = y$.

\begin{theorem}[($\chi_{\Z^\ell}$ for abelian extensions of translation groupoids)]
\label{thrm:AbelExtension}
Let $\mathcal{G}$ be an extension of a translation groupoid by a bundle of compact Lie groups as in
Equation~\eqref{eq:ExtensionTrans}, and let $\iota_{|\mathcal{G}|}$ be an embedding of $|\mathcal{G}|$
into Euclidean space with respect to which both $\mathcal{G}$ and $H\ltimes\mathcal{G}_0$ are orbit space
definable. Suppose that for each $x\in\mathcal{G}_0$, the isotropy group $\mathcal{G}_x^x$
is abelian. Then for each $\ell\geq 0$ and each $x \in\mathcal{G}_0$,
\[
    \chi_{\Z^\ell}(\mathcal{G}) = \chi\big(\HOM(\Z^\ell,\mathcal{B}_x^x) \big)
        \chi_{\Z^\ell}(H\ltimes\mathcal{G}_0).
\]
\end{theorem}

Of course, by Corollary~\ref{cor:ChiMoritaInvar}, the theorem applies to a groupoid that is
Lie groupoid Morita equivalent to such an extension. The hypothesis that both
$\mathcal{G}$ and $H\ltimes\mathcal{G}_0$ are orbit space definable is satisfied, for instance,
when $\mathcal{G}$ is orbit space definable and the weak isotropy types of
$\mathcal{G}$ and $H\ltimes\mathcal{G}_0$ coincide.

\begin{proof}
First note that for each $x\in\mathcal{G}_0$, as $\mathcal{G}_x^x$ is abelian, the subgroup $\nu(\mathcal{B}_x^x)$
and quotient group $\ker(\rho_{|\mathcal{G}_x^x})\backslash\mathcal{G}_x^x$ are abelian as well so that the isotropy groups
of $\mathcal{B}$ and $H\ltimes\mathcal{G}_0$ are also abelian. As $|\mathcal{G}| = |H\ltimes\mathcal{G}_0|$,
we have
\begin{align*}
    \chi_{\Z^\ell}(\mathcal{G})
        &=      \int_{|\mathcal{G}|}
                    \chi\big(\mathcal{G}_{x}^x\backslash\!\HOM(\Z^\ell,\mathcal{G}_{x}^x) \big) \, d\chi(\mathcal{G}x)
        \\&=    \int_{|H\ltimes\mathcal{G}_0|}
                    \chi\big(\HOM(\Z^\ell,\mathcal{G}_{x}^x) \big) \, d\chi(Hx).
\end{align*}
By the exactness of $\HOM(\Z^\ell,\cdot)$ on abelian groups, see
\cite[Chapter I, Theorems 2.1 and Proposition 4.4]{HiltonStammbach}, we have for each $x\in\mathcal{G}_0$
an exact sequence
\[
    1\to\HOM(\Z^\ell,\mathcal{B}_x^x)
    \stackrel{\nu_{|\mathcal{B}_x^x}^\ast}{\rightarrow}
        \HOM(\Z^\ell,\mathcal{G}_x^x)
    \stackrel{\rho_{|\mathcal{G}_x^x}^\ast}{\rightarrow}
        \HOM(\Z^\ell,H_x)
    \to 1,
\]
where $\nu_{|\mathcal{B}_x^x}^\ast$ and $\rho_{|\mathcal{G}_x^x}^\ast$ are the pullbacks.
Identifying $\HOM(\Z^\ell,\mathcal{G}_x^x)$ with $(\mathcal{G}_x^x)^\ell$ and $\HOM(\Z^\ell,\mathcal{B}_x^x)$
with $(\mathcal{B}_x^x)^\ell$, the map $\rho_{|\mathcal{G}_x^x}^\ast$ is realized as the quotient map of the Lie
group $(\mathcal{G}_x^x)^\ell$ by the closed subgroup $\nu(\mathcal{B}_x^x)^\ell$. Hence
$\HOM(\Z^\ell,\mathcal{G}_x^x)\to\HOM(\Z^\ell,H_x)$ is a fiber bundle with fiber $\HOM(\Z^\ell,\mathcal{B}_x^x)$.
By the multiplicativity of $\chi$ on fiber bundles,
$\chi\big(\HOM(\Z^\ell,\mathcal{G}_{x}^x) \big) = \chi\big(\HOM(\Z^\ell,\mathcal{B}_x^x)\big)
\chi\big(\HOM(\Z^\ell,H_x)\big)$, and we have
\begin{align*}
    \chi_{\Z^\ell}(\mathcal{G})
        &=      \int_{|H\ltimes\mathcal{G}_0|}
                    \chi\big(\HOM(\Z^\ell,\mathcal{B}_x^x)\big)
                    \chi\big(\HOM(\Z^\ell,H_x)\big) \, d\chi(Hx)
        \\&=    \chi\big(\HOM(\Z^\ell,\mathcal{B}_x^x)\big) \int_{|H\ltimes\mathcal{G}_0|}
                    \chi\big(\HOM(\Z^\ell,H_x)\big) \, d\chi(Hx)
        \\&=    \chi\big(\HOM(\Z^\ell,\mathcal{B}_x^x)\big)\chi_{\Z^\ell}\big(H\ltimes G\big).
        \qedhere
\end{align*}
\end{proof}

Note that if $\mathcal{B}_x^x$ is a compact abelian Lie group such that $\dim\mathcal{B}_x^x > 0$,
then $\chi(\mathcal{B}_x^x) = 0$. Hence $\chi\big(\HOM(\Z^\ell,\mathcal{B}_x^x)\big) = \chi(\mathcal{B}_x^x)^\ell = 0$
as well. Therefore, for a groupoid $\mathcal{G}$ satisfying the hypotheses of Theorem~\ref{thrm:AbelExtension},
$\chi_{\Z^\ell}(\mathcal{G}) = 0$ unless $\mathcal{B}_x^x$ is finite.

Finally, we have the following, illustrating that we cannot expect a generalization of
Theorem~\ref{thrm:AbelExtension} to the non-abelian case.

\begin{example}
\label{ex:NonabelExtension}
Let $\mathcal{G}_0 = \{x\}$ be a single point, let $\mathcal{B} = \operatorname{SO}(2,\R)$, let $H = \Z/2\Z$,
and let $\mathcal{G} = \operatorname{O}(2,\R)$. We define $\nu$ and $\rho$ to be the usual expression of
$\operatorname{O}(2,\R) = \Z/2\Z\ltimes\operatorname{SO}(2,\R)$, i.e., $\nu$ is the embedding of
$\operatorname{SO}(2,\R)$ into $\operatorname{O}(2,\R)$ and $\rho$ is the quotient map to the component group.
Treating these groups as groupoids with a single object, $\operatorname{O}(2,\R)$ is then an extension of a
translation groupoid by a bundle of compact groups. As $\operatorname{SO}(2,\R)$ and $\Z/2\Z$ are abelian
so that $|\Lambda_\Z\big(\operatorname{SO}_2(\R)\big)| = \operatorname{SO}(2,\R)$ and
$|\Lambda_\Z\big(\Z/2\Z\big)| = \Z/2\Z$, we have $\chi\big(\operatorname{SO}(2,\R)\big) = 0$
and $\chi\big(\Z/2\Z\big) = 2$. However, $|\Lambda_\Z\big(\operatorname{O}(2,\R)\big)|$
is the set of conjugacy classes of $\operatorname{O}(2,\R)$, which is homeomorphic to the disjoint union of
an interval and a point, so that $\chi\big(\operatorname{O}(2,\R)\big) = 2$.
\end{example}


\bibliographystyle{amsplainAbbrv}
\bibliography{GE}

\providecommand{\bysame}{\leavevmode\hbox to3em{\hrulefill}\thinspace}
\providecommand{\MR}{\relax\ifhmode\unskip\space\fi MR }
\providecommand{\MRhref}[2]{%
  \href{http://www.ams.org/mathscinet-getitem?mr=#1}{#2}
}
\providecommand{\href}[2]{#2}
\begin{thebibliography}{10}

\bibitem{AdemGomez}
A.~Adem and J.~M. G\'{o}mez, \emph{Equivariant {$K$}-theory of compact {L}ie
  group actions with maximal rank isotropy}, J. Topol. \textbf{5} (2012),
  no.~2, 431--457. \MR{2928083}

\bibitem{AdemLeidaRuan}
A.~Adem, J.~Leida, and Y.~Ruan, \emph{Orbifolds and stringy topology},
  Cambridge Tracts in Mathematics, vol. 171, Cambridge University Press,
  Cambridge, 2007. \MR{2359514}

\bibitem{AdemRuanTwisted}
A.~Adem and Y.~Ruan, \emph{Twisted orbifold {$K$}-theory}, Comm. Math. Phys.
  \textbf{237} (2003), no.~3, 533--556. \MR{1993337}

\bibitem{AtiyahSegal}
M.~Atiyah and G.~Segal, \emph{On equivariant {E}uler characteristics}, J. Geom.
  Phys. \textbf{6} (1989), no.~4, 671--677. \MR{1076708}

\bibitem{BehrendGinoteaStringTop}
K.~Behrend, G.~Ginot, B.~Noohi, and P.~Xu, \emph{String topology for stacks},
  Ast\'{e}risque (2012), no.~343, xiv+169. \MR{2977576}

\bibitem{BehrendXuDiffStackGerbes}
K.~Behrend and P.~Xu, \emph{Differentiable stacks and gerbes}, J. Symplectic
  Geom. \textbf{9} (2011), no.~3, 285--341. \MR{2817778}

\bibitem{BekeInvarEC}
T.~Beke, \emph{Topological invariance of the combinatorial {E}uler
  characteristic of tame spaces}, Homology Homotopy Appl. \textbf{13} (2011),
  no.~2, 165--174. \MR{2854333}

\bibitem{BochnakCosteRoyBook}
J.~Bochnak, M.~Coste, and M.-F. Roy, \emph{Real algebraic geometry}, Ergebnisse
  der Mathematik und ihrer Grenzgebiete (3) [Results in Mathematics and Related
  Areas (3)], vol.~36, Springer-Verlag, Berlin, 1998, Translated from the 1987
  French original, Revised by the authors. \MR{1659509}

\bibitem{BorelMoore}
A.~Borel and J.~C. Moore, \emph{Homology theory for locally compact spaces},
  Michigan Math. J. \textbf{7} (1960), 137--159. \MR{131271}

\bibitem{BrownSurvey}
R.~Brown, \emph{From groups to groupoids: a brief survey}, Bull. London Math.
  Soc. \textbf{19} (1987), no.~2, 113--134. \MR{872125}

\bibitem{BrumfielQuot}
G.~W. Brumfiel, \emph{Quotient spaces for semialgebraic equivalence relations},
  Math. Z. \textbf{195} (1987), no.~1, 69--78. \MR{888127}

\bibitem{BryanFulman}
J.~Bryan and J.~Fulman, \emph{Orbifold {E}uler characteristics and the number
  of commuting {$m$}-tuples in the symmetric groups}, Ann. Comb. \textbf{2}
  (1998), no.~1, 1--6. \MR{1682916}

\bibitem{BrylinskiCyclic}
J.-L. Brylinski, \emph{Cyclic homology and equivariant theories}, Ann. Inst.
  Fourier (Grenoble) \textbf{37} (1987), no.~4, 15--28. \MR{927388}

\bibitem{ChevalleyThryLieGrps}
C.~Chevalley, \emph{Theory of {L}ie groups. {I}}, Princeton University Press,
  Princeton, N. J., 1946 1957. \MR{0082628}

\bibitem{ChoiParkSuhSemialgSlice}
M.-J. Choi, D.~H. Park, and D.~Y. Suh, \emph{The existence of semialgebraic
  slices and its applications}, J. Korean Math. Soc. \textbf{41} (2004), no.~4,
  629--646. \MR{2068144}

\bibitem{CosteRealAlgSets}
M.~Coste, \emph{Real algebraic sets}, Arc spaces and additive invariants in
  real algebraic and analytic geometry, Panor. Synth\`eses, vol.~24, Soc. Math.
  France, Paris, 2007, pp.~1--32. \MR{2409684}

\bibitem{CrainicMestre}
M.~Crainic and J.~a.~N. Mestre, \emph{Orbispaces as differentiable stratified
  spaces}, Lett. Math. Phys. \textbf{108} (2018), no.~3, 805--859. \MR{3765980}

\bibitem{CrainicStruchiner}
M.~Crainic and I.~Struchiner, \emph{On the linearization theorem for proper
  {L}ie groupoids}, Ann. Sci. \'{E}c. Norm. Sup\'{e}r. (4) \textbf{46} (2013),
  no.~5, 723--746. \MR{3185351}

\bibitem{CurryGhristEAEulerCalc}
J.~Curry, R.~Ghrist, and M.~Robinson, \emph{Euler calculus with applications to
  signals and sensing}, Advances in applied and computational topology, Proc.
  Sympos. Appl. Math., vol.~70, Amer. Math. Soc., Providence, RI, 2012,
  pp.~75--145. \MR{2963602}

\bibitem{delHoyoFernandesMetricLieGpd}
M.~del Hoyo and R.~L. Fernandes, \emph{Riemannian metrics on {L}ie groupoids},
  J. Reine Angew. Math. \textbf{735} (2018), 143--173. \MR{3757473}

\bibitem{DelfsKnebuschIntro}
H.~Delfs and M.~Knebusch, \emph{An introduction to locally semialgebraic
  spaces}, vol.~14, 1984, Ordered fields and real algebraic geometry (Boulder,
  Colo., 1983), pp.~945--963. \MR{773141}

\bibitem{DixonHarveyEA}
L.~Dixon, J.~Harvey, C.~Vafa, and E.~Witten, \emph{Strings on orbifolds. {II}},
  Nuclear Phys. B \textbf{274} (1986), no.~2, 285--314. \MR{851703}

\bibitem{FarsiPflaumSeaton2}
C.~Farsi, M.~J. Pflaum, and C.~Seaton, \emph{Differentiable stratified
  groupoids and a de {R}ham theorem for inertia spaces},  (2015),
  \texttt{arXiv:1511.00371 [math.DG]}.

\bibitem{FarsiPflaumSeaton1}
\bysame, \emph{Stratifications of inertia spaces of compact {L}ie group
  actions}, J. Singul. \textbf{13} (2015), 107--140. \MR{3343617}

\bibitem{FarsiSeaGenTwistSec}
C.~Farsi and C.~Seaton, \emph{Generalized twisted sectors of orbifolds},
  Pacific J. Math. \textbf{246} (2010), no.~1, 49--74. \MR{2645879}

\bibitem{FarsiSeaVectorFld}
\bysame, \emph{Nonvanishing vector fields on orbifolds}, Trans. Amer. Math.
  Soc. \textbf{362} (2010), no.~1, 509--535. \MR{2550162}

\bibitem{FarsiSeaGenOrbEuler}
\bysame, \emph{Generalized orbifold {E}uler characteristics for general
  orbifolds and wreath products}, Algebr. Geom. Topol. \textbf{11} (2011),
  no.~1, 523--551. \MR{2783237}

\bibitem{GZEMH-HigherOrbEulerCompactGroup}
S.~M. Gusein-Zade, I.~Luengo, and A.~Melle-Hern\'{a}ndez, \emph{Higher-order
  orbifold {E}uler characteristics for compact {L}ie group actions}, Proc. Roy.
  Soc. Edinburgh Sect. A \textbf{145} (2015), no.~6, 1215--1222. \MR{3427605}

\bibitem{HiltonStammbach}
P.~J. Hilton and U.~Stammbach, \emph{A course in homological algebra},
  Springer-Verlag, New York-Berlin, 1971, Graduate Texts in Mathematics, Vol.
  4. \MR{0346025}

\bibitem{HirzebruchHoefer}
F.~Hirzebruch and T.~H\"{o}fer, \emph{On the {E}uler number of an orbifold},
  Math. Ann. \textbf{286} (1990), no.~1-3, 255--260. \MR{1032933}

\bibitem{KuhnCharacter}
N.~J. Kuhn, \emph{Character rings in algebraic topology}, Advances in homotopy
  theory ({C}ortona, 1988), London Math. Soc. Lecture Note Ser., vol. 139,
  Cambridge Univ. Press, Cambridge, 1989, pp.~111--126. \MR{1055872}

\bibitem{LeinsterEulerCharCategory}
T.~Leinster, \emph{The {E}uler characteristic of a category}, Doc. Math.
  \textbf{13} (2008), 21--49. \MR{2393085}

\bibitem{MackenzieLGLADiffGeom}
K.~Mackenzie, \emph{Lie groupoids and {L}ie algebroids in differential
  geometry}, London Mathematical Society Lecture Note Series, vol. 124,
  Cambridge University Press, Cambridge, 1987. \MR{896907}

\bibitem{McCroryParusinski}
C.~McCrory and A.~Parusi\'{n}ski, \emph{Algebraically constructible functions},
  Ann. Sci. \'{E}cole Norm. Sup. (4) \textbf{30} (1997), no.~4, 527--552.
  \MR{1456244}

\bibitem{MetzlerTopSmoothStacks}
D.~Metzler, \emph{Topological and smooth stacks},  (2003),
  \texttt{arXiv:math/0306176 [math.DG]}.

\bibitem{MoerdijkClassifRegular}
I.~Moerdijk, \emph{On the classification of regular groupoids},  (2002),
  \texttt{arXiv:math/0203099 [math.DG]}.

\bibitem{MoerdijkMrcun}
I.~Moerdijk and J.~Mr\v{c}un, \emph{Introduction to foliations and {L}ie
  groupoids}, Cambridge Studies in Advanced Mathematics, vol.~91, Cambridge
  University Press, Cambridge, 2003. \MR{2012261}

\bibitem{MoerdijkToposGpoid}
I.~Moerdijk, \emph{Toposes and groupoids}, Categorical algebra and its
  applications ({L}ouvain-{L}a-{N}euve, 1987), Lecture Notes in Math., vol.
  1348, Springer, Berlin, 1988, pp.~280--298. \MR{975977}

\bibitem{MrcunThesis}
J.~Mr\v{c}un, \emph{Stability and invariants of {H}ilsum-{S}kandalis maps},
  1996, Ph.D. Thesis, Universiteit Utrecht.

\bibitem{Munkres}
J.~R. Munkres, \emph{Topology}, Prentice Hall, Inc., Upper Saddle River, NJ,
  2000, Second edition of [ MR0464128]. \MR{3728284}

\bibitem{ParkSuhLinEmbed}
D.~H. Park and D.~Y. Suh, \emph{Linear embeddings of semialgebraic
  {$G$}-spaces}, Math. Z. \textbf{242} (2002), no.~4, 725--742. \MR{1981195}

\bibitem{PflaumOrbispace}
M.~J. Pflaum, \emph{On the deformation quantization of symplectic orbispaces},
  Differential Geom. Appl. \textbf{19} (2003), no.~3, 343--368. \MR{2013100}

\bibitem{PPTOrbitSpace}
M.~J. Pflaum, H.~Posthuma, and X.~Tang, \emph{Geometry of orbit spaces of
  proper {L}ie groupoids}, J. Reine Angew. Math. \textbf{694} (2014), 49--84.
  \MR{3259039}

\bibitem{PPTGrauertGrothendieck}
\bysame, \emph{The {G}rauert-{G}rothendieck complex on differentiable spaces
  and a sheaf complex of {B}rylinski}, Methods Appl. Anal. \textbf{24} (2017),
  no.~2, 321--332. \MR{3746621}

\bibitem{Pillay}
A.~Pillay, \emph{On groups and fields definable in {$o$}-minimal structures},
  J. Pure Appl. Algebra \textbf{53} (1988), no.~3, 239--255. \MR{961362}

\bibitem{RenaultGpdCstar}
J.~Renault, \emph{A groupoid approach to {$C^{\ast} $}-algebras}, Lecture Notes
  in Mathematics, vol. 793, Springer, Berlin, 1980. \MR{584266}

\bibitem{Roan}
S.-S. Roan, \emph{Minimal resolutions of {G}orenstein orbifolds in dimension
  three}, Topology \textbf{35} (1996), no.~2, 489--508. \MR{1380512}

\bibitem{SatakeGB}
I.~Satake, \emph{The {G}auss-{B}onnet theorem for {$V$}-manifolds}, J. Math.
  Soc. Japan \textbf{9} (1957), 464--492. \MR{95520}

\bibitem{Sea2GBPH}
C.~Seaton, \emph{Two {G}auss-{B}onnet and {P}oincar\'{e}-{H}opf theorems for
  orbifolds with boundary}, Differential Geom. Appl. \textbf{26} (2008), no.~1,
  42--51. \MR{2393971}

\bibitem{Slominska}
J.~S{\l}omi\'{n}ska, \emph{On the equivariant {C}hern homomorphism}, Bull.
  Acad. Polon. Sci. S\'{e}r. Sci. Math. Astronom. Phys. \textbf{24} (1976),
  no.~10, 909--913. \MR{461489}

\bibitem{TamanoiMorava}
H.~Tamanoi, \emph{Generalized orbifold {E}uler characteristic of symmetric
  products and equivariant {M}orava {$K$}-theory}, Algebr. Geom. Topol.
  \textbf{1} (2001), 115--141. \MR{1805937}

\bibitem{TamanoiCovering}
\bysame, \emph{Generalized orbifold {E}uler characteristics of symmetric
  orbifolds and covering spaces}, Algebr. Geom. Topol. \textbf{3} (2003),
  791--856. \MR{1997338}

\bibitem{Thurston}
W.~Thurston, \emph{The geometry and topology of 3-manifolds}, Princeton
  University Math Dept., Princeton, New Jersey, 1978, Lecture Notes.

\bibitem{tomDieckTGRepThry}
T.~tom Dieck, \emph{Transformation groups and representation theory}, Lecture
  Notes in Mathematics, vol. 766, Springer, Berlin, 1979. \MR{551743}

\bibitem{TrentinagliaRoleReps}
G.~Trentinaglia, \emph{On the role of effective representations of {L}ie
  groupoids}, Adv. Math. \textbf{225} (2010), no.~2, 826--858. \MR{2671181}

\bibitem{TrentinagliaGlobStruc}
\bysame, \emph{Some remarks on the global structure of proper {L}ie groupoids
  in low codimensions}, Topology Appl. \textbf{158} (2011), no.~5, 708--717.
  \MR{2774054}

\bibitem{TuNonHausdorff}
J.-L. Tu, \emph{Non-{H}ausdorff groupoids, proper actions and {$K$}-theory},
  Doc. Math. \textbf{9} (2004), 565--597. \MR{2117427}

\bibitem{vandenDriesBook}
L.~van~den Dries, \emph{Tame topology and o-minimal structures}, London
  Mathematical Society Lecture Note Series, vol. 248, Cambridge University
  Press, Cambridge, 1998. \MR{1633348}

\bibitem{vandenDriesMillerGeomCat}
L.~van~den Dries and C.~Miller, \emph{Geometric categories and o-minimal
  structures}, Duke Math. J. \textbf{84} (1996), no.~2, 497--540. \MR{1404337}

\bibitem{ViroIntegralEulerChar}
O.~Y. Viro, \emph{Some integral calculus based on {E}uler characteristic},
  Topology and geometry---{R}ohlin {S}eminar, Lecture Notes in Math., vol.
  1346, Springer, Berlin, 1988, pp.~127--138. \MR{970076}

\bibitem{Zelobenko}
D.~P. \v{Z}elobenko, \emph{Compact {L}ie groups and their representations},
  Translations of Mathematical Monographs, Vol. 40, American Mathematical
  Society, Providence, R.I., 1973, Translated from the Russian by Israel
  Program for Scientific Translations. \MR{0473098}

\bibitem{WangThesis}
K.~J. Wang, \emph{Proper {L}ie groupoids and their orbit spaces}, 2018, Ph.D.
  Thesis, Korteweg-de Vries Institute for Mathematics (KdVI).

\bibitem{WeinsteinLineariz}
A.~Weinstein, \emph{Linearization of regular proper groupoids}, J. Inst. Math.
  Jussieu \textbf{1} (2002), no.~3, 493--511. \MR{1956059}

\bibitem{Zung}
N.~T. Zung, \emph{Proper groupoids and momentum maps: linearization, affinity,
  and convexity}, Ann. Sci. \'{E}cole Norm. Sup. (4) \textbf{39} (2006), no.~5,
  841--869. \MR{2292634}

\end{thebibliography}

\end{document}